\documentclass[reqno,twoside,11pt]{amsart}

\usepackage{amsmath, amsfonts, amssymb, esint, amsthm, nicefrac, bbm, dsfont, enumitem}
\usepackage[foot]{amsaddr}
\usepackage{booktabs}
\usepackage{epsfig, graphicx}

\usepackage{hyperref, cite}
\hypersetup{colorlinks, breaklinks, linkcolor=blue, urlcolor=blue, anchorcolor=blue, citecolor=blue}
\usepackage{xcolor}

\def\sideremark#1{\ifvmode\leavevmode\fi\vadjust{\vbox to0pt{\vss
 \hbox to 0pt{\hskip\hsize\hskip1em
\vbox{\hsize2cm\small\raggedright\pretolerance10000
 \noindent #1\hfill}\hss}\vbox to8pt{\vfil}\vss}}}

\setlist[itemize]{leftmargin=*}
\DeclareMathAlphabet{\mathpzc}{OT1}{pzc}{m}{it}

\newcommand{\dy}{\;\mathrm{d}y}

\newcommand{\dt}{\;\mathrm{d}t}
\newcommand{\dbw}{\;\mathrm{d}\bar w}
\newcommand{\db}{\;\mathrm{d}b}
\newcommand{\dw}{\;\mathrm{d}w}

\newcommand{\dVol}{\;\mathrm{d{{Vol}}}_{\mathcal{M}}}

\newcommand{\R}{\mathbb{R}}
\newcommand{\N}{\mathbb{N}}
\newcommand{\E}{\mathbb{E}}
\renewcommand{\P}{\mathbb{P}}
\newcommand{\M}{\mathcal{M}}
\newcommand{\A}{\mathcal{A}}
\newcommand{\F}{\mathcal{F}}
\newcommand{\X}{\mathcal{X}}
\def\endproof{\hspace*{\fill}\mbox{\ \rule{.1in}{.1in}}\medskip }

\newcommand{\iid}{\emph{i.i.d.}}

\renewcommand{\O}{{\mathcal O}}
\newcommand{\eps}{\varepsilon}
\newcommand{\veps}{\varepsilon}
\newcommand{\one}{\mathds{1}}
\newcommand{\TL}{\mathrm{TL}}
\newcommand{\TLp}{\TL^p}

\newcommand{\Ec}{\mathcal{I}}
\newcommand{\Ext}{\mathcal{I}_{\epsilon,\mathcal{X}_n}}
\newcommand{\GL}{\Delta_{\epsilon,\mathcal{X}_n}}
\newcommand{\Deg}{d_{\epsilon,\mathcal{X}_n}}
\newcommand{\NL}{\Delta_{\epsilon}}
\DeclareMathOperator*{\osc}{osc}
\renewcommand{\epsilon}{\varepsilon}

\numberwithin{equation}{section}

\theoremstyle{plain}
\newtheorem{theorem}{Theorem}[section]
\newtheorem{lemma}[theorem]{Lemma}
\newtheorem{corollary}[theorem]{Corollary}
\newtheorem{proposition}[theorem]{Proposition}

\theoremstyle{definition}
\newtheorem{remark}[theorem]{Remark}

\begin{document}
\title[Lipschitz regularity of graph Laplacians on random data clouds]
{Lipschitz regularity of graph Laplacians on random data clouds} 
\author{Jeff Calder, Nicol\'as Garc\'ia Trillos, and Marta Lewicka}
\address{J. Calder: University of Minnesota, School of Mathematics,
  127 Vincent Hall, 206 Church St. S.E., Minneapolis, MN 55455}
\address{N. Garc\'ia Trillos: University of Wisconsin-Madison,
  Department of Statistics, 1300 University Avenue, Madison, WI 53706}  
\address{M. Lewicka: University of Pittsburgh, Department of
  Mathematics, 139 University Place, Pittsburgh, PA 15260} 
\email{jcalder@umn.edu, nicolasgarcia@stat.wisc.edu, lewicka@pitt.edu} 
	
\thanks{{\bf Acknowledgement:} Jeff Calder was supported by NSF-DMS grant
  1713691. Nicol\'as Garc\'ia Trillos was supported by NSF-DMS grant
  1912802. Marta Lewicka was supported by NSF-DMS grant 1613153. We
  thank the anonymous referee for careful reading of the manuscript
  and for many useful comments. In particular, we thank the referee
  for pointing out an alternative proof for some of our estimates (see
  Remark \ref{rem_reduction}) and for providing a sketch of a proof to
  relax some assumptions made in the paper (see Remark \ref{rem:91}).} 

\begin{abstract}
In this paper we study Lipschitz regularity of elliptic PDEs on
geometric graphs, constructed from random data points. The data points
are sampled from a distribution supported on a smooth manifold. The family of equations
that we study arises in data analysis in the context of graph-based
learning and contains, as important examples, the equations satisfied
by graph Laplacian eigenvectors. In particular, we prove high
probability interior and global Lipschitz estimates for solutions of
graph Poisson equations. Our results can be used to show that graph
Laplacian eigenvectors are, with high probability, essentially
Lipschitz regular with constants depending explicitly on their
corresponding eigenvalues. Our analysis relies on a probabilistic
coupling argument of suitable random walks at the continuum level, and
an interpolation method for extending functions on random point
clouds to the continuum manifold. As a byproduct of our general
regularity results, we obtain high probability $L^\infty$ and
approximate $\mathcal{C}^{0,1}$ convergence rates for the convergence of graph
Laplacian eigenvectors towards eigenfunctions of the corresponding
weighted Laplace-Beltrami operators. The convergence rates we obtain
scale like the $L^2$-convergence rates established in \cite{calder2019improved}. 
\end{abstract}

\maketitle
\tableofcontents

\section{Introduction}

        With the aim of expanding the theoretical understanding of graph-based
	methodologies in data analysis tasks, in this paper we study the
	Lipschitz regularity of graph Poisson equations. Of particular
	interest to us are fine properties of the spectra of graph Laplacians
	built from random data sets. Our main results are used to study the
	regularity of graph Laplacian eigenvectors and their strong
        convergence in the large data limit. Several authors have
        proposed the use of graphs to 
	endow data sets with geometric structure, and in particular have
	utilized graph Laplacians to understand how information propagates on
	the graph representing the data. Spectra of graph Laplacians are
	fundamental geometric descriptors that can be used to extract
	meaningful local and global summarized information from data sets. Graph Laplacians and
	their spectra form the basis of popular algorithms for supervised
	learning \cite{zhu2003semi,smola2003kernels,Belkin-Niyogi-Sindhwani,Ando},
	clustering \cite{ng2002spectral,vonLux_tutorial}, construction of embeddings and dimensionality
	reduction~\cite{belkin2002laplacian,Coifman7426,Portegies}. They are used in
	non-parametric statistics as local regularizers
	\cite{rosasco2013nonparametric,tibshirani2014adaptive,wang2016trend},
	as well as in Bayesian settings, to define covariance matrices for Gaussian priors
	\cite{zhu2003semi,kirichenko2017,BertoStuart}.  
	
	While a graph Laplacian can be associated to any arbitrary graph, many
	authors in the learning community have given particular attention to
	the important family of \textit{geometric} graphs since the
        2000's.  Since then, the following theoretical question has been
	explored from different points of view: if a data set $\X_n=\{ x_1,
	\dots, x_n \}$ is obtained by sampling a distribution supported on
	some manifold $\M$, and a graph representing the similarity between
	data points is built based on distance proximity, what can we learn
	about the manifold $\M$ from the data set? The term \textit{manifold
	learning} was coined to capture this general question, which was
	typically associated to the study of Laplacians. From a statistical
	perspective, manifold learning can be rephrased slightly differently:
	if an algorithm depends on the input of random data sampled from a
	distribution supported on a manifold, what can we say about the
	outcomes of said algorithm?, are outcomes consistent in the large data
	limit, and if so, how many data points are needed to reach a certain
	level of accuracy for the approximation of a ground truth defined at
	the continuum (manifold) level? Ultimately, an attempt to study a
	manifold learning question is an attempt to develop mathematical
	theory with the hope of providing a better understanding of a given
	learning algorithm. 
	
	In this paper we study regularity properties of a class of graph PDEs
	on geometric graphs, a manifold learning question that has not been
	studied in the already large literature on graph Laplacians from
	random samples. By studying said regularity properties we are able to
	investigate strong notions of convergence of graph Laplacian
	eigenvectors towards eigenvectors of Laplace-Beltrami operators (or
	weighted versions thereof) and provide high probability rates
	characterizing this convergence. Several other ramifications will be explored elsewhere.  
	
	The general type of results that we obtain can be
	described as follows. Suppose that associated to our dataset $\X_n$,
	we have weights $\omega_{ij}$ which represent the similarity between
	points $x_i$ and $x_j$. In the manifold learning setting, where
	$\X_n=\{ x_1, \dots, x_n \}$ are thought of as samples from a
	$m$-dimensional manifold $\M$ embedded in a possibly high dimensional
	ambient space $\R^d$ (i.e., the \emph{manifold
	assumption}~\cite{ssl}), the weights are typically determined by
	proximity between points. Here, we focus our attention on
	the special class of $\veps$-graphs whose weights (up to
	appropriate rescaling) are given by: 
	\[ \omega_{ij}= \eta \left(\frac{|x_i-x_j|}{\veps} \right),\]
	and in particular depend on a user-chosen connectivity length-scale
	$\veps>0$ as well as on a decreasing kernel $\eta$ (typically chosen
	to be a Gaussian in applications).  The edge weights $\omega_{ij}$
	induce a graph structure on $\X_n$ as well as an associated graph Laplacian:
	\begin{equation}\label{eq:GL}
	\GL f(x_i) = \frac{1}{n\eps^{m+2}}\sum_{j=1}^n \eta \left(\frac{|x_i-x_j|}{\veps} \right)\big(f(x_i) - f(x_j)\big)
	\end{equation}
	which acts on functions $f:\X_n\to \R$. 
	
	If we think of the data points in $\X_n$ as samples from
	some ground-truth distribution, the question of regularity of solutions to graph Poisson equations: 
	$$\GL f=g,$$ 
	becomes a probabilistic question, as the graph Laplacian depends on the random points used to
	build it. The main result in our paper essentially says that, provided
	$\veps$ is not too small (i.e., it satisfies \eqref{eqn:RegimeEps} below),
	then with very high probability solutions to graph Poisson
        equations satisfy an \emph{approximate} Lipschitz estimate: 
	\begin{equation}
	|f(x_i)- f(x_j)| \leq C (\lVert \GL f \rVert_{L^\infty(\X_n)} +
	\lVert f \rVert_{L^\infty(\X_n)}  )\cdot\left( d_\M(x_i,x_j) + \veps
	\right) \quad \mbox{ for all}\; x_i , x_j \in \X_n, 
	\label{eqn:RegularityInformal}
	\end{equation}
	where $d_\M$ is the geodesic distance on $\M$. We say
	\emph{approximate} Lipschitz estimate because, while at length-scales
	larger than $\veps$ (where we recall $\veps$ determines the
	connectivity of the graph), $f$ indeed behaves like a Lipschitz
	function, it is actually not possible to resolve its level of
	regularity within $\veps$-neighborhoods, a fact that should not be
	surprising as geometric graphs behave like complete graphs at
	length-scale $\veps$. We remark that the constant $C$ that appears in
	(\ref{eqn:RegularityInformal}) is independent of $f$, $\veps$ or $n$, and only depends
	on the underlying distribution (in particular also on the manifold
	$\M$) data points are drawn from. We also prove a local version of
	\eqref{eqn:RegularityInformal}. 

\subsection{Estimates for eigenvectors}

	The regularity estimates that we obtain in this paper are quite general and can
	be used to deduce a variety of results. Here we only use them
	towards regularity  estimates of eigenvectors of $\GL$, as well as
	to obtain uniform and approximate $\mathcal{C}^{0,1}$ convergence rates of said
	eigenvectors to continuum (manifold) counterparts. To see how our
	results apply to the eigenvector problem:
	\begin{equation}\label{eigen_prob}
	\GL f = \lambda f,
	\end{equation}
	we first directly apply \eqref{eqn:RegularityInformal} to obtain
	an estimate of the form:
	\[ |f(x_i)- f(x_j)| \leq C (\lambda+1) \lVert f
	\rVert_{L^\infty(\X_n)} \cdot\big( d_\M(x_i,x_j) + \veps \big) \quad \mbox{ for all } \;  x_i , x_j \in \X_n. \] 
	Although at first sight it would seem as if the above estimate was
	only meaningful if \emph{a priori} estimates on the $L^\infty$-norm of
	$f$ were available, in fact, we show that it follows from the
	Lipschitz estimate above that $\|f\|_{L^\infty(\X_n)}\leq
	C(\lambda+1)^m \|f\|_{L^1(\X_n)}$ and so if the eigenvector is
	properly normalized, namely: $\|f\|_{L^2(\X_n)}=1$, then we have: 
	\[ |f(x_i)- f(x_j)| \leq C (\lambda+1)^m \left( d_\M(x_i,x_j) + \veps
	\right) \quad \mbox{ for all } \;  x_i , x_j \in \X_n.\] 
	Besides obtaining regularity estimates for eigenvectors of $\GL$ with
	constants depending explicitly on $\lambda$, we also use our general
	regularity estimates \eqref{eqn:RegularityInformal} to bootstrap the
	$L^2$-convergence rates of graph Laplacian eigenvectors towards their
	continuum counterparts (as studied in \cite{calder2019improved}) and
	upgrade them to $L^\infty$ and approximate $\mathcal{C}^{0,1}$ convergence
	rates.  That is, suppose that $f$ is a (properly
	normalized) solution to (\ref{eigen_prob}), and let $\tilde{f}$ be an eigenfunction of the continuum
	(local) Laplacian counterpart of $\GL$ for which we have error
	estimates on $\lVert f -\tilde f\rVert_{L^2(\X_n)}$ (see
	\cite{calder2019improved}). Applying \eqref{eqn:RegularityInformal} to
	the difference $f - \tilde f$ (where $\tilde f$ is interpreted
	as the restriction of $\tilde f$ to $\X_n$) we are able to upgrade the high
	probability $L^2$-convergence rates of eigenvectors to $L^\infty$ and
	approximate $\mathcal{C}^{0,1}$ convergence  rates. It is interesting to notice
	that the rates that we obtain using  \eqref{eqn:RegularityInformal} are much better than the ones we would
	have obtained if we had used the fact that both $f$ and
	$\tilde{f}$ are Lipschitz, together with an interpolation
	inequality. Indeed, the standard interpolation inequality
	$\|f\|_{L^\infty(\X_n)}\leq C[f]_1^{m/(m+2)}\|f\|_{L^2(\X_n)}^{2/(m+2)}$, where
	$[f]_1$ is the Lipschitz seminorm of $f$, suffers from the curse of
	dimensionality in its dependence on $\|f\|_{L^2(\X_n)}$, while our results
	show that the $L^\infty$ rates are the same as the $L^2$
        rates. We remark that all the above discussion is meaningful  
	as long as $\veps$ is in the regime: 
	\begin{equation}\label{eqn:RegimeEps}
	\left(\frac{\log(n)}{n} \right)^{\frac{1}{m+4}} \lesssim \veps \lesssim 1.
	\end{equation}

\subsection{Literature perspective}

        Thus far, our discussion suggests that by studying regularity
	properties of graph Laplacians, we are able to deduce a variety of
	novel results that are of substantial theoretical importance for the
	graph-based learning community. To provide some perspective and
	to highlight our contributions, it is worth mentioning some of the
	related existing literature. Early work on consistency of graph
	Laplacians focused on pointwise consistency results for $\veps$-graphs
	(see, for example \cite{singer2006graph,hein2005graphs,hein2007graph,
		belkin2005towards,ting2010analysis,GK}). There,
	as in here, the data is assumed to be an \emph{i.i.d.}~sample of size
	$n$ from a ground truth measure, supported on a submanifold $\M$
	embedded in a high dimensional Euclidean space $\R^d$, where pairs of
	points that are within distance $\veps$ of each other are given high
	weights. Pointwise consistency results show that as $n\to \infty$ and
	the connectivity parameter $\veps \to 0$ (at a slow enough rate), the
	graph Laplacian applied to a fixed smooth test function converges to
	the value of a continuum operator, such as the Laplace-Beltrami.
	In the past few years, the literature has moved beyond
	pointwise consistency and started studying the sequence of solutions to
	graph-based learning problems and their continuum limits, using tools
	like $\Gamma$-convergence
	\cite{slepcev19,dunlop2019large,trillos2016consistency,calder18bAAA}, tools from PDE theory
	\cite{calder2018game,calder18AAA,calder18bAAA,RyanNicolas2019,flores2019algorithms}
	including the maximum principle and viscosity solutions, and
        martingale methods \cite{calder2020rates}, which are related
        to the techniques used in this paper. 
	
	Regarding spectral convergence of graph Laplacians, the regime
	$n \rightarrow \infty$ and $\veps\equiv const$ was studied in
	\cite{vLBeBo08}, and  in \cite{SinWu13} which analyzed connection Laplacians. Works that
	have studied regimes where $\veps\to 0$ include
	\cite{trillos2018variational}, \cite{Shi2015}, \cite{BIK}, and
	\cite{trillos2018spectral}. In \cite{calder2019improved} convergence
	rates for eigenfunctions under $L^2$-type distances are
	deduced, in the same regime for $\eps$
	given in \eqref{eqn:RegimeEps}. Namely, it has been proved that the rate of
	convergence of eigenvectors scales linearly in $\veps$, matching the
	convergence rate of eigenvalues as well as the pointwise convergence
	rates; these results are to the best of our
	knowledge state of the art. As shown in Theorem
        \ref{thm:eigenrate}, in this paper we are able to upgrade the
        results from \cite{calder2019improved} to $L^\infty$ and
        to almost $\mathcal{C}^{0,1}$ convergence. 
        
We point out that our work is one of only three very recent papers  that obtain $L^\infty$
convergence rates for graph Laplacian eigenvectors (see
\cite{dunson2019diffusion} and \cite{WormellReich}); these three works
use very different approaches. Uniform convergence is an important
notion for settings in machine learning such as semi-supervised
learning where it is key to formulate algorithms for which pointwise
evaluations are well posed asymptotically. We notice that 
convergence rates from \cite{dunson2019diffusion} are looser than
ours. On the other hand, the rates 
obtained in \cite{WormellReich} hold under restrictive assumptions:
$\M$ is assumed to be a flat torus, the probability density is
constant, the kernel used to build the graph is Gaussian. Also, no
regularity estimates are deduced from the analysis in \cite{WormellReich}.   

\subsection{New tools and arguments}

        The fact that our paper studies a manifold learning question that has
	not been studied in the past, suggests that the methods and techniques
	employed here are also novel in the analysis of graph-based learning, 
	and thus of interest on their own right. Our proofs
	contrast with those traditionally used to analyze graph Laplacians, that
	mostly rely on spectral and variational techniques. We now outline
	some of these new arguments. 
	
	First, in order to study the regularity of graph Laplacians, we
	analyze a closely related continuum \textit{non-local} Laplacian of the form: 
	\begin{equation}\label{eq:NL}
	\NL  f(x)\doteq \frac{1}{\eps^{m+2}}\int_\M \eta\left(
	\frac{d_\M(x,y)}{\veps}\right)\big( f(x)- f(y)\big ) \rho(y)\, \dVol (y)
	\quad \mbox{ for all } \; x \in \M, 
	\end{equation}
	which acts on functions $f : \M \rightarrow \R$, where $d_\M$
	represents the geodesic distance on $\M$, and $\rho$ is the density of
	the point cloud $\X_n$. The non-local Laplacian $\Delta_\eps$ can be
	thought of intuitively as the $n \rightarrow \infty$, $\veps>0$ fixed
	counterpart of the graph Laplacian $\GL$. Of possibly independent
	interest, we prove a continuum Lipschitz estimate of the form: 
	\begin{equation}\label{eq:nonlocalLip}
	|f(x)-f(y)|\leq C\big(\|\Delta_\veps f\|_{L^\infty(\M)} + \|f\|_{L^\infty(\M)}\big)\cdot (d_{\M}(x,y) + \eps)
	\end{equation}
	for functions $f:\M\to \R$. Our main Lipschitz estimates in the graph
	setting are proved by using a novel interpolation map that extends
	functions on $\X_n$ to functions on $\M$ in
	such a way that the non-local Laplacian $\Delta_\eps$ is controlled in
	an  appropriate norm by the graph Laplacian $\GL$, and then applying the
	Lipschitz estimate \eqref{eq:nonlocalLip} to the interpolated function. 
	
	Second, the Lipschitz estimate \eqref{eq:nonlocalLip} is proved with a
	probabilistic argument (not related to the randomness of the data points!:
	notice that $\Delta_\eps$ is deterministic) which is based on a 
	coupling for a suitable random walk. The argument goes as follows. For
	an arbitrary pair of points $x, y \in \M$ we consider discrete time
	random walks $\{ X_k \}_{k\in \N}$ and $\{ Y_k \}_{k \in \N}$   with
	state space $\M$ starting at $x$ and $y$ respectively, both of which
	have as generator an operator closely related to
        $\Delta_{\veps}$. These walks are coupled to encourage
        coalescence; we consider a stopping time $\tau$, essentially defined as the
	first time at which either the walks have gotten sufficiently close to
	each other or have drifted apart a certain order-one distance. For the
	appropriately coupled walks, we are able to provide basic
	quantitative estimates for $\tau$, show that $\tau$ is not
	expected to be too large, and also that the
	probability of the walks being close to each other at time $\tau$ is
	close to one (i.e. the walks do coalesce). We then use martingale
	techniques to bound the difference $|f(x) - f(y)|$ in terms of the
	difference $|f(X_\tau) - f(Y_\tau)|$, the point being that while $x,y$
	may be of order-one apart, the points $X_\tau$ and $Y_\tau$  will be
	closer together (with high probability), thus allowing to estimate
	$|f(x) - f(y)|$ in terms of $|f(\tilde{x})- f(\tilde{y})|$ for 
        $\tilde x , \tilde y$ that are closer together than the
	original $x,y$. From there, we follow an iteration argument to
	eventually obtain the desired regularity estimates.  All
        details of the idea outlined above will be given in
        section  \ref{sec:NonLocalLipschitz}. 
	
	The type of argument described above follows a line of work that has
	developed probabilistic techniques to study regularity properties of
	(continuum) PDEs. In particular, the reflection coupling
        method dates back to the work by Lindvall and Rogers on coupling
        of diffusion processes \cite{lindvall1986coupling}, where the
        Brownian parts of were coupled via a time-dependent field of orthogonal matrices. The
        Lindvall-Rogers coupling was used by Cranston
        \cite{cranston1991gradient} to prove gradient estimates for
        equations involving the Laplace-Beltrami operator on
        manifolds. The method has been significantly generalized and
        applied to parabolic and elliptic equations
        \cite{kusuoka2017holder,kusuoka2015holder,wang1994gradient,wang2004gradient,priola2006gradient}. 
        Coupling  methods in the discrete setting have also been used to
        establish H\"older and Lipschitz regularity in nonlinear
        potential theory, and in particular, for the $p$-Laplacian via
        the connection to stochastic two player tug-of-war games
        \cite{luiro2018regularity,parviainen2016local,arroyo2019asymptotic,arroyo2016tug,han2018local}. There
        are also recent application to H\"older regularity for the
        Robin problem \cite{lewicka2019robin, lewicka2019robin2}. 
	
	In an independent thread, the viscosity solutions community
        developed methods for proving H\"older regularity of
        degenerate elliptic equations via doubling the variables and
        utilizing the comparison principle for semi-continuous
        functions with an appropriately constructed supersolution
        (see, e.g., \cite[Section VII]{ishii1990viscosity}). It was
        later realized that, at a high level, the analytic techniques
        using the comparison principle are roughly equivalent to
        probabilistic coupling, with the doubling variables
        playing the role of coupling of diffusion processes. We refer to the appendix of
        \cite{porretta2013global} for a detailed discussion of the
        analytic versus probabilistic methods. 
	
	When the data density is constant, the coupling used in our paper can be viewed as a discrete analog of the
        Lindvall-Rogers coupling \cite{lindvall1986coupling}, adapted
        to a smooth manifold in \cite{cranston1991gradient}. When the data density is not
        constant, the corresponding random walk has a small drift
        component along the gradient of the density. The drift appears
        through a lack of symmetry in the random walk increment, and is
        not a simple additive drift, as in the Lindvall-Rogers
        framework. In order to couple the random walks with drift, we
        construct the random walks increments by probabilistically
        mixing a symmetric random walk step, to which the reflection
        coupling is applied, with a pure drift step, to which a parallel coupling is applied.  

\section{Setup and main results} \label{sec:Setup}

Let $\M$ be a compact, connected, orientable, smooth, $m$-dimensional
manifold embedded in $\R^d$. We give to $\M$ the Riemannian structure
induced by the ambient space $\R^d$. The geodesic distance between
$x,y\in\M$ is denoted $d_\M(x,y)$. We write $B_\M (x,r)$ for the
geodesic ball in $\M$ of radius $r$ centered at $x$, while 
$B(x,r)$ is used to denote a Euclidean ball in $\R^m$ or in $\R^d$,
depending on context.  By $\dVol$ we denote the volume form on
$\M$. Other tools and notation from Riemannian geometry  will be
introduced as needed in the sequel. We have compiled a list of
definitions and auxiliary geometric estimates in the Appendix \ref{appendixA}.

Let $\mu$ be a probability measure supported on $\M$, with density
$\rho : \M \rightarrow (0,\infty)$ (with respect to $\dVol$),
which we assume is bounded, bounded away from zero,  and at least $\mathcal{C}^2(\M)$. 
Let $\X_n= \{ x_1, \dots, x_n \}$ be a set
of $\iid$ samples from $\mu$. We denote by $L^2(\X_n)$ the space of functions
$f: \X_n \rightarrow \R$ endowed with the inner product: 
\[ \langle f , g \rangle_{L^2(\X_n)} \doteq \frac{1}{n}\sum_{i=1}^n f(x_i) g(x_i).\] 
This induces a norm $\|f\|_{L^2(\X_n)} = \langle f , f
\rangle_{L^2(\X_n)}^{1/2}$. We also define the $L^1$ and $L^\infty$ norms:
\[ \|f\|_{L^1(\X_n)} \doteq \frac{1}{n}\sum_{i=1}^n |f(x_i)|, \qquad \|f\|_{L^\infty(\X_n)} \doteq \max_{1\leq i \leq n} |f(x_i)|.\]
	
Let $\eta\colon [0,\infty) \rightarrow [0,\infty)$ be a
non-increasing function with support on the interval $[0,1]$, whose
restriction to $[0,1]$ is Lipschitz continuous.  Note that $\eta$ may
be discontinuous on $[0,\infty)$, and that we allow functions such as
$\eta(t) = \one_{[0,1]}$. We assume that: 
\begin{equation*}
\int_{\R^m} \eta(|w|) \dw =1,
\end{equation*}
and we define the constant:
\begin{equation} \label{def:sigma}
\sigma_\eta \doteq \int_{\R^m} \langle w, e_1\rangle^2 \eta(|w|) \dw.
\end{equation}
Let $\eps>0$. A weight between two points $x,y\in \M$ is defined by:
\begin{equation*}
w_{xy} \doteq \eta\left(\frac{|x-y|}{\eps}\right),
\end{equation*}
where $|x-y|$ is the Euclidean distance from $x$ to $y$ in $\R^m$.
The weights $w_{xy}$ endow $\X_n$ with the structure of a graph called
the \emph{random geometric graph}. We define the associated graph
Laplacian $\GL:L^2(\X_n)\to L^2(\X_n)$ by the expression in
\eqref{eq:GL}. We also define the nonlocal Laplacian $\NL:L^2(\M)\to
L^2(\M)$ by \eqref{eq:NL}, understood as the $n\to
\infty$ and $\eps=const$ continuum limit of the graph Laplacian $\GL$.  
In the continuum limit as $n\to \infty$ and $\eps\to 0$, the graph
Laplacian $\GL$ is consistent (see \cite{calder2019improved,
  hein2007graph}) with the weighted Laplace-Beltrami operator: 
\begin{equation}\label{eqn:LaplaceBeltrami}
\Delta_\M f \doteq - \frac{\sigma_\eta}{2\rho}\text{div}(\rho^2 \nabla f). 
\end{equation}
	
\subsection{Main results}
	
Unless otherwise noted, the
constants $C,c>0$ in the theorems depend only on $\M$, $\rho$ and
$\eta$. Also, in the statements of our results, as well as in the
remainder of the paper, whenever we write $a\ll 1$, we mean that the
positive quantity $a$ is assumed to be less than or equal to a
sufficiently small constant that may depend on $\M, \rho$ or $\eta$. 
	
\begin{theorem}[Global Lipschitz regularity] 
\label{thm:globalLip}
Let $\eps\ll 1$. Then, with probability at least $1-C\eps^{-6m}\exp\left( -c n\eps^{m+4} \right)$ we have:
\begin{equation*}
\begin{split}
|f(x_i) - f(x_j)|\leq & \; C\big( \|f\|_{L^\infty(\X_n)}+ \|\GL f\|_{L^\infty(\X_n)}\big)\cdot \big(d_\M(x_i, x_j)+\eps\big), \\
\end{split}
\end{equation*}
for all $f\in L^2(\X_n)$ and all $x_i,x_j\in \X_n$.
\end{theorem}
Theorem \ref{thm:globalLip} shows that the Lipschitz regularity of
$f\in L^2(\X_n)$ is controlled by the size of $f$ and $\GL f$ in
$L^\infty$. In particular, solutions of graph Poisson equations $\GL f
=g$ with bounded $f$ are Lipschitz continuous on the graph $\X_n$, at
length scales larger than $\eps$.  
	
We also prove a parallel interior estimate:	

\begin{theorem}[Interior Lipschitz regularity] \label{thm:interiorLip}
Let $\epsilon\ll r \ll 1$. Then, with probability at least 
$1-C\eps^{-6m}\exp\left( -c n\eps^{m+4} \right)$ we have: 
\begin{align*}\label{eq:interiorLip}
|f(x_i)-f(x_j)| \leq & \; C\Big(\frac{d_\M(x_i,x_j)}{r} + 
  \frac{\epsilon|\log(\veps)|}{r} \Big)  \cdot \|f\|_{L^\infty(\X_n\cap B_\M(x,7r))}
\\ & + C\Big(r d_\M(x_i, x_j) + \frac{\epsilon r}{|\log \epsilon|} \Big) \cdot \|\GL f\|_{L^\infty(\X_n\cap B_\M(x,7r))},
\end{align*}
for all $f\in L^2(\X_n)$, $x\in \M$, $r>0$, and $x_i,x_j\in B_\M(x,r)\cap \X_n$.
\end{theorem}
	
As a direct application of Theorem \ref{thm:globalLip}, we observe the following:

\begin{theorem}[Lipschitz regularity of graph Laplacian eigenvectors] \label{thm:eigenLip}
Let $\Lambda>0$, $\veps \ll 1$, such that $\eps \leq
\frac{c}{\Lambda+1}$. Then, with probability at least
$1-C\eps^{-6m}\exp\left( -c n\eps^{m+4} \right)-2n\exp\left(-cn(\Lambda+1)^{-m} \right)$ we have: 
\begin{equation*}
\begin{split}
 |f(x_i) - f(x_j)| \leq & \; C(\Lambda + 1)^{m+1}\lVert f \rVert_{L^1(\X_n)} (d_\M(x_i,x_j) + \veps),
\end{split}
\end{equation*}
valid for all non-identically zero $f\in L^2(\X_n)$ with
$\lambda_f<\Lambda$ and all $x_i,x_j\in \X_n$. Here:
\begin{equation*}
\lambda_f \doteq \frac{\|\GL f\|_{L^\infty(\X_n)}}{\|f\|_{L^\infty(\X_n)}}.
\end{equation*}
\end{theorem}

\begin{remark}
In Theorem \ref{thm:eigenLip}, if we take $f\in L^2(\X_n)$ to be a
normalized eigenvector of $\GL$ with eigenvalue $\lambda$, satisfying
(\ref{eigen_prob}) and $\|f\|_{L^2(\X_n)}=1$, then the result implies:
\begin{align}\label{eqn:aux09}
|f(x_i) - f(x_j)|&\leq C(\lambda + 1)^{m+1}(d_\M(x_i,x_j) + \veps),
\end{align}
since $\|f\|_{L^1(\X_n)}\leq \|f\|_{L^2(\X_n)}=1$, and
$\lambda_f=\lambda$. We also note that in the smallness condition for
$\veps$ and in the right hand side of the inequality, we can take $\Lambda$ as a small constant multiple
of a corresponding eigenvalue $\tilde{\lambda}$ of the continuum local
Laplacian $\Delta_\M$:  if $\lambda$ is the $k$-th
eigenvalue of $\GL$, then we can let
$\tilde{\lambda}$ be the $k$-th eigenvalue of the weighted
Laplace-Beltrami operator $\Delta_\M$. This is due to the consistency results
(with rates) for the eigenvalues of $\GL$; see \cite{calder2019improved}. 
\end{remark}

Theorem \ref{thm:eigenLip} allows us to estimate the $L^\infty$ norm of eigenvectors by their $L^1$ norms:

\begin{corollary}\label{cor:eigenLip}
Under the same conditions as in Theorem \ref{thm:eigenLip} and in the
same event where inequality \eqref{eqn:aux09} holds, we have:
\begin{equation*}
\|f\|_{L^\infty(\X_n)}\leq C(\Lambda+ 1)^{m+1}\|f\|_{L^1(\X_n)}.
\end{equation*}
\end{corollary}
 
Finally, we use our general regularity estimates from Theorem
\ref{thm:globalLip} to obtain the following uniform and approximate
$\mathcal{C}^{0,1}$ convergence rates for the eigenvectors of the graph
Laplacian towards eigenfunctions of the weighted Laplace-Beltrami
operator $\Delta_\M$. To make our statement precise, for $\delta>0$ we define the
$\delta$-approximate Lipschitz seminorm of $f\in L^2(\X_n)$ by: 
\begin{equation*}
[f]_{\delta,\X_n} \doteq \max_{x,y\in \X_n}\frac{|f(x)-f(y)|}{d_\M(x,y) + \delta}.
\end{equation*}

\begin{theorem}[Convergence rates for eigenvectors of the graph Laplacian]
Let $\eps \ll 1$ and suppose that $f$ is a normalized eigenvector of $\GL$ with eigenvalue
$\lambda$, i.e. it satisfies (\ref{eigen_prob}) and $\lVert f
\rVert_{L^2(\X_n)}=1$. Then, with probability at least
$1-C(n+\eps^{-6m})\exp\left( -c n\eps^{m+4} \right)$  there exists a
normalised eigenfunction $\tilde{f}$ of $\Delta_\M$ defined in \eqref{eqn:LaplaceBeltrami},
i.e. $\Delta_\M\tilde f = \tilde\lambda \tilde f$ and $\|\tilde f\|_{L^2(\M)}=1$, with:
\begin{equation*}
\|f - \tilde{f}\|_{L^\infty(\X_n)} + [f-\tilde{f}]_{\epsilon,\X_n} \leq C\veps.
\end{equation*}
where the constant $C$in the right hand side above depends additionally on ${\lambda}$. 
\label{thm:eigenrate}
\end{theorem} 

\begin{remark}
Compared to the results reported in \cite{dunson2019diffusion}
(e.g. Theorem 2) which state uniform convergence of eigenvectors with
high probability at the rate $n^{-1/(8m+30)}$ when picking the
connectivity parameter as $\veps \sim n^{-1/(4m+15)}$, our results
imply uniform convergence with rates scaling linearly in $\veps$ for
all $ \veps\gg \left(\frac{\log(n)}{n}\right)^{1/(m+4)}  $. In
particular, choosing $\veps \sim
\left(\frac{\log(n)}{n}\right)^{1/(m+4)}  $ we obtain uniform
convergence of eigenvectors at the rate
$\left(\frac{\log(n)}{n}\right)^{1/(m+4)} $. We highlight that our
results hold for a stronger almost $\mathcal{C}^{0,1}$ notion of
convergence. 

Finally, let us recall that the asymptotic almost sure spectral
convergence of graph Laplacians towards Laplace-Beltrami
operators (with eigenvector convergence understood in an
$L^2$-sense) can be guaranteed in the regime: 
\[  \left( \frac{\log(n)}{n} \right)^{1/m} \lesssim \veps \lesssim 1 \]
(see e.g. \cite{calder2019improved}), and  it is not unreasonable to expect that
similar $L^\infty$  consistency results can be obtained in
the same regime as well. However, we believe that new ideas are
actually needed in order to enlarge the regimes for $\veps$
and provide corresponding \textit{quantitative high
probability} error estimates. Our work, in particular our regularity
estimates in terms of suitable non-local operators, may be used directly in
a future analysis, and the focus for improvement can be put on
the probabilistic estimates relating graph Laplacians with these
non-local operators.
\end{remark}

\begin{remark}
If we take the length scale required in our main theorems $\veps \sim
\left( \frac{\log(n)}{n}\right)^{1/(m+4)}$, then the number of edges
in the resulting graph is
$n^{\frac{m+8}{m+4}}\log(n)^{\frac{m}{m+4}}$. Notice that as the
manifold dimension $m$ increases, the graph has fewer edges and better
sparsity properties. For larger $m$ the number of edges in the graph
is slightly larger than linear in the number of nodes.  
\end{remark}

The proofs of theorems listed so far rely strongly on the intermediate
results that are also of independent interest. We present them in the
next sections.

\subsection{Almost-interpolation maps: lifting the discrete to the continuum}
	\label{sec:ResultsInterp}
	
Several techniques have been developed recently for performing an
interpolation of functions defined on $\X_n$, in order to extend them
to the whole $\M$. For example, in 
the context of variational techniques and $\Gamma$-convergence, the
$\TLp$ topology was developed in \cite{garciatrillos16} for precisely
this purpose, and has been used numerous times since, for studying
discrete to continuum
convergence~\cite{slepcev19,garciatrillos18,trillos2018spectral,garciatrillos16a,garciatrillos17,dunlop2019large, 
garciatrillos17cAAA,garciatrillos18a,thorpe19,osting17,calder18bAAA}. The 
$\TLp$ topology involves defining measure preserving transportation
maps pushing the discrete probability measure $\mu_n\doteq \frac{1}{n} \sum_{i=1}^n
\delta_{x_i}$ onto the continuum probability distribution $\mu$, and is useful for
controlling energies in a $\Gamma$-convergence framework.  For
problems where a maximum principle is available, which gives very
strong discrete stability, interpolation maps are not needed and it is
sufficient to consider the restriction of smooth functions to the graph. We refer to
\cite{calder2018game,calder18AAA,yuan2020continuum,trillos2019maximum,shi18AAA}
for applications of the maximum principle to discrete to continuum convergence. 
	
Here, we are concerned with passing from a discrete graph problem to a
nonlocal equation, while controlling the values of
the graph and nonlocal Laplacians. For this purpose, we
develop a new technique for extending discrete functions on the graph
to functions on $\M$. We define the almost-interpolation operator
$\Ext:L^2(\X_n) \to L^2(\M)$ and the rescaled degree $\Deg(x)$:
\begin{equation}\label{eq:extension}
\begin{split}
& \Ec_{\epsilon,\X_n} f(x)  \doteq
\frac{1}{d_{\epsilon,\X_n}(x)}\cdot\frac{1}{n}\sum_{i=1}^n\frac{1}{\epsilon^m}\eta\Big(\frac{|x-x_i|)}{\epsilon}\Big) f(x_i)\\  
& d_{\epsilon,\X_n}(x) \doteq \frac{1}{n}\sum_{i=1}^n\frac{1}{\epsilon^m}\eta\Big(\frac{|x-x_i|}{\epsilon}\Big)
\quad\mbox{ for all } x\in\M,
\end{split}
\end{equation}
where in both formulas above $\eta$ is applied to the scaled Euclidean
distance between $x, x_i\in\R^d$. At this stage, for any $x\in \M$ for
which $\Deg(x) = 0$, we take $\Ext f(x)=0$. Later on, we will show
that with very high probability $d_{\veps,\X_n}(x)>0 $ for all $x \in
\M$ for appropriate scalings of $\veps$ and $n$. Also, notice that
$d_{\veps, \X_n}$ is nothing but a kernel density estimator for $\rho$.

The following result establishes discrete to nonlocal control for
$\Ext$. In order to make the statements precise we introduce the
oscillation of a function $u$ over a set $A$: 
\[\osc_A f \doteq \sup_{A }f - \inf_A f.\]

\begin{theorem}[Discrete to Nonlocal]\label{thm:DiscreteToNonlocal}
Let $\eps \ll 1$. With probability at least $1 - C\eps^{-6m} \exp\big( -cn\eps^{m+4} \big)$ we have:
\begin{equation*}
\left|\NL (\Ext f)(x)\right| \leq C\big( \|\GL f\|_{L^\infty(\X_n\cap B(x,\eps))} +\osc_{\X_n\cap B(x,2\eps)}f\big)
\end{equation*}
for all $f\in L^2(\X_n)$ and all $x\in \M$.
\end{theorem}

Notice that Theorem \ref{thm:globalLip} allows to control the
oscillation term in Theorem \ref{thm:DiscreteToNonlocal}, which leads
to the following improved discrete to nonlocal result:

\begin{corollary}[Improved Discrete to Nonlocal]\label{cor:ImpDiscreteToNonlocal}
Let $\eps \ll 1$. With probability at least $1-C\eps^{-6m}\exp\left( -cn\eps^{m+4} \right)$ we have:
\begin{equation*}
\|\NL (\Ext f)\|_{L^\infty(\M)} \leq C\left( \|\GL f\|_{L^\infty(\X_n)} +\eps\|f\|_{L^\infty(\X_n)}\right)
\end{equation*}
for all $f\in L^2(\X_n)$.
\end{corollary}

The proof of Theorem \ref{thm:interiorLip} combines the discrete to
nonlocal control from Theorem \ref{thm:DiscreteToNonlocal} with the 
nonlocal Lipschitz regularity estimates that we discuss next.

\subsection{Lipschitz regularity for nonlocal Laplacian} \label{sec:ResultsNonlocal} 
	
The results presented in this section are of
independent interest, but here we use them to prove our main
results. First we establish the following global nonlocal
Lipschitz regularity estimates towards the proof of Theorem \ref{thm:globalLip}.
	
\begin{theorem}[Global regularity]\label{th_main1}
Let $\epsilon\ll 1$. Then, for every bounded, Borel function
$f:\M\to\R$ and every $x, y\in\M$ there holds:
\begin{align*}
|f(x) - f(y)| \leq & \; C\big(\|f\|_{L^\infty(\M)} +\|\NL
f\|_{L^\infty(\M)}\big)\cdot \big(d_\M(x, y)+\eps\big).
\end{align*}
\end{theorem}

We also have the following interior estimate, used in the proof of Theorem \ref{thm:interiorLip}. 
\begin{theorem}[Interior estimate]\label{th_main2}
Let $\epsilon\ll r \ll 1$. Then, for every bounded Borel
function $f:\M\to\R$, every $x_0\in \M$ and $x,y\in B_\M(x_0, r) \subset \M$, we have: 
\begin{align*}
|f(x) - f(y)| \leq &\; C \Big (\frac{d_\M(x,y)}{r} + \frac{\epsilon |
  \log \veps|}{r}\Big) \cdot \lVert f \rVert_{L^\infty(B_\M(x_0,7r)} 
 \\ & + C \Big( r d_\M(x,y) + \frac{\veps r}{|\log\epsilon|}\Big)\cdot \|\Delta_\veps f\|_{L^\infty(B_\M(x_0, 7r))}.
\end{align*}
\end{theorem}
	
\begin{remark}
We note that Theorems \ref{th_main1} and \ref{th_main2} do not require
the manifold $\M$ to be embedded in Euclidean space $\R^d$, and they hold
for an abstract Riemannian manifold.  
\label{rem:embedded}
\end{remark}

\subsection{Outline} The rest of the paper is organized as follows. In
section \ref{sec33} we discuss properties of the discrete degree
$d_{\epsilon, \X_n}$ while in section \ref{sec_interpol} we deal with 
the almost-interpolation operator $\Ec_{\epsilon, \X_n}$ which allows us to
relate  discrete with continuum functions, and graph with nonlocal
Laplacians, as stated  in Theorem \ref{thm:DiscreteToNonlocal}. 
Towards further applications, in Section \ref{LCQ_sec} we present a result of independent interest,
namely a curvature-driven error estimate on geodesic distances in
the Levi-Civit\`a quadrilateral, frequently used in the following
analysis on manifolds.

Section \ref{secAver} discusses the two averaging operators in
connection with the weighted Laplace-Beltrami operator in (\ref{eqn:LaplaceBeltrami}).
Further, in section \ref{sec_77} we introduce a biased random walk,
which is a discrete process modelled on one of the averaging operators.
Sections \ref{sec:NonLocalLipschitz}, \ref{sec_second} and \ref{sec_closebound}
are directed towards the proofs of Theorems \ref{th_main1} and \ref{th_main2}
characterizing the regularity of nonlocal Poisson equations at the
continuum level. We wrap up the paper with the proofs of our main theorems in
section \ref{sec:Spectrum} where we use directly the results announced
in sections \ref{sec:ResultsInterp} and \ref{sec:ResultsNonlocal}.

We emphasize that sections \ref{sec33}-\ref{sec_interpol}
and sections \ref{secAver}-\ref{sec_closebound} are independent of each other. Readers
who decide to skip one of these sections may be able to do so and jump
directly to section \ref{sec:Spectrum} without missing the general
structure of the proofs of our main results.  Throughout the paper we will use several notions from Riemannian 
geometry; these are gathered in Appendix \ref{appendixA}.

\bigskip

\begin{center}
{\bf \large PART 1}
\end{center}

\section{Concentration of measure and the discrete degree \texorpdfstring{$d_{\epsilon, \X_n}$}{d}}\label{sec33}

In this section we work under the following hypotheses:\vspace{-2mm}

\begin{equation*}\label{H1}\tag{{\bf H1}}
\left[\quad \;\; \mbox{\begin{minipage}{13.1cm}\vspace{1mm}
\begin{itemize}[leftmargin=*]
\item[(i)] $(\M,g)$ is a smooth, compact, boundaryless, connected
and orientable  manifold of dimension $m$, embedded in $\R^d$,\vspace{1mm}
\item[(ii)] $\rho\in\mathcal{C}^2(\M)$ is a positive scalar field, normalised to:
$\int_\M\rho(x)\dVol(x) = 1$,\vspace{1mm}
\item[(iii)] $\eta:[0,\infty)\to\R$ is a nonnegative, nonincreasing density
function, which is Lipschitz continuous on $[0,1]$ and satisfies:
$\int_{B(0,1)\subset\R^m}\eta(|w|)\dw = 1.$ We then denote:
$\eta_\epsilon(s) = \frac{1}{\epsilon^m}\eta\big(\frac{s}{\epsilon}\big)$. \vspace{1mm}
\end{itemize}
\end{minipage}}\;\;\, \right]
\end{equation*}

\smallskip

\begin{remark}
The orientability assumption will not be used until Section
\ref{sec:Spectrum}, which builds on several previous results in the literature
invoking this condition explicitly. We believe that the
orientability assumption can be removed from those previous results as
well. For the sake of transparency we have decided to keep this assumption, which, we
should say, is quite mild for most applications in machine learning. 
\end{remark}

We denote $\mu=\rho(x)\dVol(x)$ the probability measure on the Borel subsets of $\M$. For each $n\geq
1$ we consider the following product space, with elements $\X_n$:
\begin{equation*}
\begin{split}
(\M^n, \mbox{Borel}, \P_n) &\doteq (\M, \mbox{Borel}, \mu)^n\quad
\\ & \mbox{with } \; \M^n = \big\{\X_n=\{x_1, \ldots, x_n\}; ~ x_i\in\M \;\mbox{
  for } i=1\ldots n\big\}.
\end{split}
\end{equation*}
Equivalently, $\X_n$ is simply described as a set of $\iid$ samples
from $\mu=\P_1$, while the subscript $n$ in $\P_n$ emphasizes the
dependence of events on the first $n$ data points $\{x_1, \dots, x_n \}$.  

Recall that to any discrete function $f:\X_n\to\R$ we associate
the continuum function $\Ec_{\epsilon,\X_n} f:\M\to \R$, defined in (\ref{eq:extension}).
As we shall see, the event $\{\X_n; ~d_{\epsilon, \X_n}(x)>0 \mbox{ for
  all } x\in\M\}\subset\M^n$, resulting in the corresponding operator $\Ec_{\epsilon, \X_n}$ being well
defined and returning a Borel, bounded function $\Ec_{\epsilon, \X_n}f$
for every discrete $f$, occurs with high probability $\P_n$ (see Corollary \ref{good_def}). 

In this section, the main point is to focus on $d_{\epsilon, \X_n}$
versus $d_\epsilon$ and develop a series of technical
tools towards the main result in Theorem \ref{thm:DiscreteToNonlocal}
that will be given in section \ref{sec_interpol}.
There are two facts from differential geometry that we will 
frequently use. These facts, related to the manifold $\M$ being embedded
in $\R^d$ can be found in section \ref{sec2.4}, but we also state
them presently, for the sake of clarity. First, there exists a constant
$C>0$ depending only on $\M$ such that, denoting 
$|B(0,1)|$ the volume of the unit ball in $\R^m$, we have for all $r\ll 1$:
\begin{equation}\label{eq:vol1}
\big|{Vol}_\M(B_\M(x,r))- |B(0,1)| r^m\big|\leq Cr^{m+2} \qquad\mbox{ for all } x\in\M.
\end{equation}
Second, for all $x,y\in\M$ such that $|x-y|\leq
\frac{R}{2}$ with a sufficiently small $R$, there holds:
\begin{equation}\label{eq:d1}
|x-y|\leq  d_\M (x,y)\leq |x-y|+ \frac{8}{R^2}|x-y|^3.
\end{equation}

\subsection{Concentration of measure lemmas}

We first state a basic concentration of measure result,
including its self-contained proof that has been sketched in \cite[Lemma 3.1]{calder2019improved}.

\begin{lemma}\label{lem:ch}
Given a bounded, Borel function $\psi:\M\to \R$ and $\epsilon>0$, for
each $x\in \M$ we consider the following random variable:
\begin{equation*}
\Psi^{\epsilon,x}(\X_n) \doteq \sum_{i=1}^n \mathds{1}_{\{|x_i-x|\leq \epsilon\}}
\psi(x_i) \qquad \mbox{for all } \; \X_n=\{x_1,\ldots, x_n\}\in\M^n.
\end{equation*}
Then, there exists a positive constant $C$, depending only on $\M$, such that
for all $\epsilon\ll 1$ and $t$ satisfying $\epsilon^2\leq t\leq 1$,
and all $x\in\M$, there holds:
\begin{equation}\label{eq:con}
\begin{split}
\P_n\Big( \Big| \Psi^{\epsilon, x} - n\int_{B_\M(x,\epsilon)}\psi(y)
    \rho(y)\dVol (y)\Big|\geq C t n\epsilon^m &\|\rho\|_{\mathcal{C}^0}
  \|\psi\|_{L^\infty(B_\M(x,2\epsilon))} \Big) \\ & \leq
2\exp\Big( -\frac{C}{8} t^2 n\epsilon^m \|\rho\|_{\mathcal{C}^0} \Big).
\end{split}
\end{equation}
\end{lemma}
\begin{proof}
{\bf 1.} By comparison of the Euclidean and geodesic distances in
(\ref{eq:d1}), one observes:
\begin{equation}\label{m1}
B_\M(x,\epsilon)\subset B(x,\epsilon)\cap\M\subset B_\M\big(x,
\epsilon+\frac{8}{R^2}\epsilon^3\big)\subset B_\M(x, 2\epsilon)
\end{equation}
for all $x\in\M$ and all $\epsilon\ll 1$. Further, in view of
(\ref{eq:vol1}) and denoting by $C$ a sufficiently large constant that
depends only on $\M$, it follows that:
\begin{equation}\label{m2}
\begin{split}
Vol_\M\big(B(x,\epsilon)\cap&\M\big)-Vol_\M(B_\M(x,\epsilon)) \leq 
Vol_\M\big(B_\M\big(x, \epsilon+\frac{8}{R^2}\epsilon^3\big)\big) - Vol_\M(B_\M(x,\epsilon)) 
\\ & \leq |B(0,1)| \big((\epsilon+ C\epsilon^3)^m - \epsilon^m\big) + C\epsilon^{m+2}
\leq C\epsilon^{m+2}.
\end{split}
\end{equation}

\smallskip

{\bf 2.} The argument towards (\ref{eq:con}) relies on applying Bernstein's inequality \cite{boucheron2013concentration}
to the independent random variables $\big\{Y_i\doteq \mathds{1}_{\{|x_i-x|\leq \epsilon\}}
\psi(x_i) \big\}_{i=1}^n$ on $\M^n$.  By (\ref{m1}), these obey the pointwise
bound: $|Y_i| \leq \|\psi\|_{L^\infty(B_\M(x,2\epsilon))}$ and the bound on the variance:
\begin{equation*}
\begin{split}
\E\big[(Y_i - &\E[Y_i])^2\big]  = \E[Y_i^2]-\E[Y_i]^2\leq
\E[Y_i^2]=\int_{B(x,\epsilon)\cap\M}\psi(y)^2\rho(y)\dVol (y) \\ & \leq 
\|\rho\|_{L^\infty(\M)} \|\psi\|^2_{L^\infty(B_\M(x,2\epsilon))} Vol_\M\big(B_\M(x,2\epsilon)\big)
\leq C \epsilon^m \|\rho\|_{L^\infty(\M)} \|\psi\|^2_{L^\infty(B_\M(x,2\epsilon))},
\end{split}
\end{equation*}
where in the last step we again used (\ref{eq:vol1}). We conclude that, for every $\delta>0$:
\begin{equation*}
\begin{split}
\P_n\Big( &\big| \Psi^{\epsilon, x} - \E[\Psi^{\epsilon, x}]\big|\geq \delta \Big) \leq
2\exp\Big( -\frac{\delta^2}{2Cn\epsilon^m\|\rho\|_{L^\infty(\M)}
  \|\psi\|^2_{L^\infty(B_\M(x,2\epsilon))} + \frac{4}{3}\delta \|\psi\|_{L^\infty(B_\M(x,2\epsilon))}}\Big),
\end{split}
\end{equation*}
which upon taking $\delta = C tn \epsilon^m \|\rho\|_{L^\infty(\M)}
  \|\psi\|_{L^\infty(B_\M(x,2\epsilon))}$ yields:
\begin{equation}\label{eq:con0}
\begin{split}
\P_n\Big( \big| \Psi^{\epsilon, x} - & \; \E[\Psi^{\epsilon, x}]\big|\geq
C tn \epsilon^m \|\rho\|_{L^\infty(\M)}
\|\psi\|_{L^\infty(B_\M(x,2\epsilon))} \Big) \\ & \leq
2\exp\Big( -\frac{C\epsilon^m n t^2 \|\rho\|_{L^\infty(\M)}}{2+
  \frac{4}{3}t}\Big) \leq 2\exp\Big( -\frac{C}{4} t^2n\epsilon^m \|\rho\|_{L^\infty(\M)}\Big).
\end{split}
\end{equation}
Recall now (\ref{m2}) and (\ref{m1}) to get, in view of $\epsilon^2\leq t$:
\begin{equation*}
\begin{split} 
\Big|\E[\Psi^{\epsilon, x}] - n&\int_{B_\M(x,\epsilon)}\psi(y) \rho(y)\dVol (y) \Big|  = 
n \Big|\int_{(B(x,\epsilon)\cap\M)\setminus B_\M(x,\epsilon)}\psi(y) \rho(y)\dVol (y) \Big|
\\ & \leq Cn \epsilon^{m+2}\|\rho\|_{L^\infty(\M)} \|\psi\|^2_{L^\infty(B_\M(x,2\epsilon))}
\leq Ctn \epsilon^m \|\rho\|_{L^\infty(\M)} \|\psi\|^2_{L^\infty(B_\M(x,2\epsilon))}.
\end{split}
\end{equation*}
Together with (\ref{eq:con0}), the above implies (\ref{eq:con}) with
constants $2C$ and $\frac{C}{4}$, instead of $C$ and $\frac{C}{8}$ in
the left and right hand sides, respectively. The result follows by
rescaling the constant $C$.
\end{proof}

\begin{corollary}\label{prop:ball}
There exist constants $C,C',c>0$, depending on $\M$ and $\rho$, such that for all $\epsilon\ll 1$ and all
$x\in \M$ there holds:
\begin{equation*}
\P_n\Big( C' n \epsilon^m \leq \sum_{i=1}^n \mathds{1}_{\{|x_i-x|\leq
  \epsilon\}} \leq C n \epsilon^m\Big) \geq 1-2\exp\left( -c n\epsilon^m \right).
\end{equation*}
\end{corollary}
\begin{proof}
Applying Lemma \ref{lem:ch} to $\psi\equiv 1$, yields for all
$\epsilon^2\leq t\leq 1$ with $\epsilon\ll 1$, and all $x\in\M$:
\begin{equation*}
\begin{split}
\P_n\Big( \Big| \sum_{i=1}^n \mathds{1}_{\{|x_i-x|\leq \epsilon\}} - n\int_{B_\M(x,\epsilon)}
    \rho(y)&\dVol (y)\Big|\leq C t n\epsilon^m \|\rho\|_{L^\infty(\M)} \Big) \\ & \geq
1- 2\exp\Big( -\frac{C}{8} t^2 n\epsilon^m \|\rho\|_{L^\infty(\M)} \Big).
\end{split}
\end{equation*}
In view of  (\ref{eq:vol1}) we note that for $\epsilon\ll 1$:
\begin{equation*}
\begin{split}
 n\int_{B_\M(x,\epsilon)}\rho(y)\dVol (y) \in & \; n
Vol(B_\M(x,\epsilon))\cdot \big[\min\rho, \max\rho\,\big]\\ &
=  n \Big(|B(0,1)|\epsilon^m + \mathcal{O}(\epsilon^{m+2})\Big)\cdot \big[\min\rho, \max\rho\,\big]
\\ & \subset n |B(0,1)|\epsilon^m \cdot \Big[\frac{1}{2}\min\rho, 2\max\rho\,\Big].
\end{split}
\end{equation*}
It follows that with probability at least $1-2\exp\left(
  -\frac{C}{8}t^2n\epsilon^m \|\rho\|_{\mathcal{C}^0} \right)$ there holds:
\begin{equation*}
\begin{split}
\sum_{i=1}^n \mathds{1}_{\{|x_i-x|\leq
  \epsilon\}} \in \Big[ & \frac{1}{2} n |B(0,1)|\epsilon^m \min\rho
- C t n\epsilon^m \|\rho\|_{L^\infty(\M)} , \\ & \quad 2 n |B(0,1)|\epsilon^m \max\rho
+ C t n\epsilon^m \|\rho\|_{L^\infty(\M)}\Big].
\end{split}
\end{equation*}
Taking $t$ appropriately small, in function of $m$ and $\rho$,
concludes the proof.
\end{proof}

\medskip

The next result provides another basic property of the random
geometric graph, with balls $\{|x_i-x|\leq \epsilon \}$
in Corollary \ref{prop:ball} replaced by the annuli $\{(1-t)\epsilon\leq |y-x|\leq (1+t)\epsilon\}$.

\begin{proposition}\label{prop:annulus}
There exist constants $C>c>0$, depending on $\M$ and $\rho$, such that
for all $\epsilon\ll 1$ and $t$ satisfying $\epsilon^2\leq t\leq 1$,
and all $x\in \M$ we have:
\begin{itemize}
\item[(i)] $Vol_\M\Big( \{y\in\M; ~ (1-t)\epsilon\leq |y-x|\leq
  (1+t)\epsilon\}\Big)\leq Ct\epsilon^m$,
\item[(ii)]
$ \displaystyle \P_n\Big( \sum_{i=1}^n \mathds{1}_{\{(1-t)\epsilon\leq |x_i-x|\leq
(1+t) \epsilon\}} \leq C tn \epsilon^m\Big) \geq 1-\exp\left( -c tn\epsilon^m \right).$
\end{itemize}
\end{proposition}
\begin{proof}
{\bf 1.} Given $x$ and $\epsilon$, denote the closed annulus, where $t\in [\epsilon^2, 1]$:
$$A_t = \big\{y\in\M; ~ (1-t)\epsilon\leq |y-x|\leq (1+t)\epsilon\big\}.$$
Using (\ref{eq:d1}) and similarly as in (\ref{m1}), we obtain for $\epsilon\ll 1$:
\begin{equation*}
B_\M\big(x,(1+t)\epsilon\big)\setminus B_\M\big(x,(1-t)\epsilon + C\epsilon^3\big) \subset A_t
\subset B_\M\big(x,(1+t)\epsilon + C\epsilon^3\big)\setminus B_\M\big(x,(1-t)\epsilon\big). 
\end{equation*}
Hence, for $t\in [\frac{1}{2},1]$ there follows in virtue of (\ref{eq:d1}):
\begin{equation*}
\begin{split}
& Vol_\M(A_t)\leq Vol_\M\big(B_\M(x,2\epsilon +C \epsilon^3)\big)\leq
C\epsilon^m\leq Ct\epsilon^m
\\ & Vol_\M(A_t)\geq Vol_\M\big(B_\M(x,\frac{3}{2}\epsilon)\big) -
Vol_M\big(B_\M(x,\epsilon)\big)\geq c\epsilon^m\geq ct\epsilon^m,
\end{split}
\end{equation*}

For $t<\frac{1}{2}$ we estimate more precisely, using the mean value property of the $m$-th power:
\begin{align*}
Vol_\M(A_t) & \geq |B(0,1)|\cdot\Big( (1+t)^m\epsilon^m -
\big((1-t)\epsilon + C\epsilon^3\big)^m \Big) + \mathcal{O}(\epsilon^{m+2}) \\
& = |B(0,1)|\epsilon^m \cdot\Big( (1+t)^m -
\big(1-t + C\epsilon^2\big)^m \Big) + \mathcal{O}(\epsilon^{m+2}) \\
& \geq ct\epsilon^m -C\epsilon^{m+2},
\end{align*}
valid when $\epsilon\ll 1$. Similarly, we get the upper bound: 
\begin{align*}
Vol_\M(A_t)&\leq |B(0,1)|\cdot\Big(
\big((1+t)\epsilon+C\epsilon^3\big)^m - (1-t)^m\epsilon^m\Big) +
\mathcal{O}(\epsilon^{m+2})\\ 
&\leq C\epsilon^{m}\Big(\big(1+t+ C\epsilon^2\big)^m - (1-t)^m\Big) +
C\epsilon^{m+2}\leq  Ct\epsilon^m.
\end{align*}
Both bounds allow to conclude the following estimate:
\begin{equation}\label{m2.5}
ct\epsilon^m\leq Vol_\M(A_t)\leq Ct\epsilon^m \qquad \mbox{ when }\; t\in\Big[
\frac{2C}{c}\epsilon^2, 1\Big].
\end{equation}

Finally, when $\epsilon^2\leq t<\frac{2C}{c}\epsilon^2$, we simply
note that with $\bar t\doteq \frac{2C}{c}\epsilon^2>t$ there holds the inclusion:
$A_t\subset A_{\bar t}$ and hence:
$$Vol_\M(A_t)\leq Vol_\M(A_{\bar t})\leq C\bar t \epsilon^m \leq C
\frac{\bar t}{\epsilon^2} \cdot t\epsilon^m = \frac{4C^2}{c} t\epsilon^m, $$
proving (i) in this last case.

\smallskip

{\bf 2.} To show (ii), we use Chernoff's bound
\cite{boucheron2013concentration} to the $\iid$ (Bernoulli) random variables $\{Y^i=\mathds{1}_{\{x_i\in
  A_t\}}\}_{i=1}^n$. We first consider the case
$t\in\big[\frac{2C}{c}\epsilon^2,1\big]$, indicated in (\ref{m2.5}). Denote:
$$p\doteq \E[Y^i] =\int_{A_t}\rho(y)\dVol (y)\in
Vol_\M(A_t)\cdot[\min\rho, \max\rho]\subset[ct\epsilon^m, Ct\epsilon^m],$$
so that the version of additive Chernoff's bound in \cite{calder2020Calculus}[Theorem 5.7] yields:
$$\P_n\Big( \sum_{i=1}^n Y^i\geq 2n p\Big)\leq \exp\Big( -\frac{3}{8}np\Big)\leq \exp\big(-ctn\epsilon^m\big).$$
We thus obtain (ii) in this particular case:
$$\P_n\Big( \sum_{i=1}^n Y^i\leq Ctn\epsilon^m \Big)\geq 1-
\exp\big(-ctn\epsilon^m\big) \qquad \mbox{ when }\; t\in\Big[
\frac{2C}{c}\epsilon^2, 1\Big].$$

To complete the argument when $\epsilon^2\leq t < \frac{2C}{c}\epsilon^2$,
we set  $\bar t =\frac{2C}{c}\epsilon^2$ as in step 1 and note that:
\begin{equation*}
\begin{split}
\P_n\Big( \sum_{i=1}^n \mathds{1}_{\{x_i\in A_t\}}\leq \frac{4C^2}{c}tn\epsilon^m \Big)&\geq 
\P_n\Big( \sum_{i=1}^n \mathds{1}_{\{x_i\in A_{ t}\}}\leq C\bar t n\epsilon^m \Big)\geq 
\P_n\Big( \sum_{i=1}^n \mathds{1}_{\{x_i\in A_{\bar t}\}}\leq C\bar t
n\epsilon^m \Big)\\ & \geq 
1 - \exp\big(-c\bar t n\epsilon^m\big) \geq 1- \exp\big(-ctn\epsilon^m\big).
\end{split}
\end{equation*}
The proof is done.
\end{proof}

\subsection{Comparison of the scaled degrees
  \texorpdfstring{$d_{\epsilon, \X_n}$}{d} and \texorpdfstring{$d_\epsilon$}{d}} 
We now check that $d_{\epsilon, \X_n}$ approximates, with high probability,
its continuum version, given by:
\begin{equation}\label{degi}
d_\epsilon(x) = \int_{B_\M(x,\epsilon)}\frac{1}{\epsilon^m}\eta\Big(\frac{ d_\M (x,y)}{\epsilon}\Big)
\rho(y)\;\dVol(y) \qquad\mbox{ for all } x\in\M.
\end{equation}
Using normal coordinates centered at the point $x$ (see
Appendix \ref{appendixA} and the proof of Theorem \ref{consistency} in
section \ref{secAver}), together with the Taylor expansion of $\rho$ around $x$, we may rewrite $d_\veps(x)$ as:
\begin{equation}\label{dep1}
\begin{split}
d_\epsilon (x) & = \int_{B(0,1) \subset T_x\M}\eta(|w|)\big(\rho(x)+\epsilon \langle\nabla \rho(x),
w\rangle +\mathcal{O}(\epsilon^2)\big)
\big(1+\mathcal{O}(\epsilon^2)\big) \dw \\ & = \rho(x) + \mathcal{O}(\epsilon^2).
\end{split}
\end{equation}

\begin{proposition}\label{prop:covering}
For each $h>0$ there exists a finite set $\M_h\subset \M$ such that,
with a constant $C>0$ depending only on $\M$, one has:
\begin{equation*}
\#\M_h \leq Ch^{-m}\quad \mbox{ and } \quad \sup_{x\in \M}\min_{y\in \M_h}d_\M(x,y)\leq h.
\end{equation*}
\end{proposition}
\begin{proof}
The proof follows by localizing and reducing to the flat case in which
$\M$ is replaced by a ball $B(0,\delta)\subset\R^m$. In this case,
a subset of the cubical grid $h\mathbb{Z}^m\subset \R^m$ fulfills the desired properties.
\end{proof}

\begin{lemma}\label{lem:CM1}
There exist constants $C,c>0$ depending only on $\M$, $\rho$ and $\eta$, such
that for all $\epsilon\ll 1$ and all $t$ satisfying $\epsilon^2 \leq t \leq 1$, there holds:
$$\P_n\Big(\sup_{x\in\M} \big|d_{\epsilon, \X_n}(x) - d_\epsilon(x)\big|
\leq Ct\Big)\geq 1-C(t\epsilon)^{-m}\exp\left( -ct^2 n\epsilon^m \right).$$
\end{lemma}
\begin{proof}
{\bf 1.} Applying Lemma \ref{lem:ch} to the functions $\psi(y) = \eta\big(\frac{|y-x|}{\epsilon}\big)$, it follows that:
\begin{equation}\label{m2.6}
\P_n\Big(\Big|d_{\epsilon, \X_n}(x) -
\int_{B_\M(x,\epsilon)}\eta_\epsilon(|y-x|)\rho(y)\dVol (y)\Big|
\leq Ct\Big)\geq 1- 2\exp\left( -ct^2 n\epsilon^m \right),
\end{equation}
for all $x\in\M$, with constants $C$ and $c$ independent of $x$. 
By monotonicity of $\eta$ and its Lipschitz continuity on $[0,1]$, for every $y\in B_\M(x,\epsilon)$ we have:
\begin{equation}\label{m2.7}
\eta_\epsilon(d_\M(x,y))\leq \eta_\epsilon(|y-x|) \leq
\eta_\epsilon\big(d_\M(x,y)-C\epsilon^3\big) \leq
\eta_\epsilon(d_\M(x,y)) +C\epsilon^{2-m}\cdot \mbox{Lip}_\eta,
\end{equation}
in view of (\ref{eq:d1}). Thus, for every $x\in\M$ there holds:
$$\Big|d_\epsilon(x) - \int_{B_\M(x,\epsilon)}
\eta_\epsilon (|y-x|)\rho(y)\dVol (y)\Big|\leq
C \mbox{Lip}_\eta \cdot \epsilon^{2-m}
\|\rho\|_{L^\infty(\M)}Vol_\M(x,\epsilon) \leq C\epsilon^2\leq Ct,$$
and we see that the estimate in (\ref{m2.6}) is still valid after
replacing the integral term by
the degree $d_\epsilon(x)$, and possibly changing the uniform
constant $C$.

Recalling Corollary \ref{prop:ball} and
Proposition \ref{prop:annulus}, we further conclude that each of the following
events (where $\epsilon, t,x$ are fixed):
\begin{equation}\label{m3}
|d_{\epsilon, \X_n}(x) - d_\epsilon(x) |\leq Ct, \quad \sum_{i=1}^n
\mathds{1}_{|x_i-x|\leq \epsilon}\leq Cn\epsilon^m, \quad 
\sum_{i=1}^n \mathds{1}_{(1-t)\epsilon\leq |x_i-x|\leq (1+t)\epsilon}\leq Ctn\epsilon^m,
\end{equation}
hold with probability $\P_n$ at least $1- 2\exp\left( -ct^2 n\epsilon^m \right)$, in view of $t\leq 1$. 
Let $\M_{t\epsilon}\subset\M$ be the discrete set provided by
Proposition \ref{prop:covering}, By the union bound, the event that
all conditions in (\ref{m3}) hold for all $x\in\M_{t\epsilon}$, has
probability at least $1- C(t\epsilon)^m\exp\left( -ct^2 n\epsilon^m \right)$.
For the rest of the proof we assume this event holds.

\smallskip

{\bf 2.} Let $x\in\M$ and $y\in\M_{t\epsilon}$ with $d_\M(x,y)\leq t\epsilon$. We write:
\begin{equation}\label{m4}
|d_{\epsilon, \X_n}(x) - d_\epsilon(x)|\leq |d_{\epsilon, \X_n} (x) -
d_{\epsilon,\X_n}(y)| + |d_{\epsilon, \X_n}(y) - d_\epsilon(y)| + |d_{\epsilon}(y) - d_\epsilon(x)|.
\end{equation}
We will show that each term in the right hand side above is bounded by
$Ct$. This is true for the second term, directly by (\ref{m3}), because
$y\in\M_{t\epsilon}$. The third term bound follows by (\ref{dep}):
$$|d_\epsilon(x)-d_\epsilon(y)|=|\rho(x)-\rho(y)|+\mathcal{O}(\epsilon^2)
\leq \mbox{Lip}_\rho \cdot d_\M(x,y) + \mathcal{O}(\epsilon^2)\leq Ct.$$

To treat the first term in (\ref{m4}), denote sets of indices $I_x=\{i;~ |x_i-x|\leq \epsilon\}$ and 
$I_y= \{i;~ |x_i-y|\leq \epsilon\}$. Then:
\begin{equation}\label{m45}
\begin{split}
& \epsilon < |x_i-y|\leq |x_i-x| + |x-y|\leq \epsilon + d_\M(x,y) \leq
(1+t)\epsilon \qquad\mbox{for all } i\in I_x\setminus I_y, \\ &
\epsilon \geq |x_i-y|\geq |x_i-x| - |x-y|> \epsilon - d_\M(x,y) \geq
(1-t)\epsilon \qquad\mbox{for all } i\in I_y\setminus I_x.
\end{split}
\end{equation}
In conclusion:
$$(1-t)\epsilon\leq |x_i-y|\leq (1+t)\epsilon \qquad\mbox{for all }\; i\in I_x{\triangle} I_y.$$
It thus follows that:
\begin{equation*}
\begin{split}
\big|d_{\epsilon, \X_n}(x) &- d_{\epsilon, \X_n}(y)\big|  \leq \frac{1}{n}\sum_{i=1}^n
\big|\eta_\epsilon(|x-x_i|) -\eta_\epsilon(|y-x_i|)\big| \\ & \leq
\frac{1}{n}\sum_{i\in I_x\cap I_y}
\big|\eta_\epsilon(|x-x_i|) - \eta_\epsilon(|y-x_i|) \big|
+ \frac{\|\eta\|_{L^\infty}}{n\epsilon^m}\sum_{i=1}^n
\mathds{1}_{\{(1-t)\epsilon\leq|x_i-y|\leq (1+t)\epsilon\}} \\ & \leq
\frac{\mbox{Lip}_\eta}{n\epsilon^m}\sum_{i\in I_y}
\Big|\frac{|x-x_i|}{\epsilon} - \frac{|y-x_i|}{\epsilon} \Big| + Ct
\leq \frac{C}{n\epsilon^{m+1}}\sum_{i\in I_y}|x-y| + Ct \\ &
\leq \frac{C}{n\epsilon^{m+1}}d_\M(x,y) \cdot \sum_{i=1}^n
\mathds{1}_{\{|x_i-y|\leq \epsilon\}} +Ct\leq Ct,
\end{split}
\end{equation*}
where we applied conditions guaranteed in (\ref{m3}). Thus, (\ref{m4}) becomes: $|d_{\epsilon,
  X_n}(x) - d_\epsilon(x)|\leq Ct$ for all $x\in\M$ on the event
indicated in (\ref{m3}). This ends the proof.
\end{proof}

From Lemma \ref{lem:CM1} and in view of (\ref{dep1}), the operator $\Ec_{\epsilon, \X_n}$ is well
defined with high probability:

\begin{corollary}\label{good_def}
There exist constants $C,c>0$ depending on $\M$, $\rho$ and $\eta$,
such that for all $\epsilon\ll 1$ and all $t$ satisfying $\epsilon^2\leq
t\leq 1$, there holds:
$$\P_n\Big(\sup_{x\in\M}\big|d_{\epsilon, \X_n}(x) - \rho(x)\big|\leq
Ct\Big) \geq 1-C(t\epsilon)^{-m}\exp\big(-ct^2 n \epsilon^m\big).$$
In particular, taking $t\ll 1$, we get:
$$\P_n\Big(d_{\epsilon, \X_n}(x) \in\big[\frac{1}{2}\min\rho,
2\max\rho\big] \quad \text{for all } x\in\M
\Big) \geq 1-C\epsilon^{-m}\exp\big(-c n \epsilon^m\big).$$
\end{corollary}

\begin{remark}\label{m47}
Applying the second claim of Corollary
\ref{good_def} to $\eta=\one_{[0,1]}$, results in the following improvement on Corollary
\ref{prop:ball}, with constants $C,C',c>0$ depending only on $\M$ and $\rho$:
\begin{equation*}
\P_n\Big( C' n \epsilon^m \leq \sum_{i=1}^n \mathds{1}_{\{|x_i-x|\leq
  \epsilon\}} \leq C n \epsilon^m \quad \text{for all } x\in\M
\Big) \geq 1-C\epsilon^{-m}\exp\big(-c n \epsilon^m\big).
\end{equation*}
\end{remark}

\section{The almost-interpolation operator \texorpdfstring{$\Ec_{\epsilon, \X_n}$}{Ec} and a
  proof of Theorem \ref{thm:DiscreteToNonlocal}} \label{sec_interpol}

In this section we work under hypothesis (\ref{H1}). Let us first give
a heuristic explanation of the ideas behind the new 
interpolation technique. When extending a function $f\in L^2(\X_n)$ to the
manifold $\M$, the function $f$ may \emph{a priori} have no
regularity. Even though $f$ is defined on a point cloud, and
could be extended to a smooth function on $\M$, there is no way to do
this while uniformly controlling any of the derivatives of the
interpolation. Thus, when extending from $\X_n$
to $\M$, it is important to keep in mind that we must
make no assumptions about the regularity of $f$, and our probabilistic estimates must be applied to the
graph $\X_n$ directly and be independent of the function we are extending. 

In this section, we prove Theorem
\ref{thm:DiscreteToNonlocal} which allows to control the nonlocal 
Laplacian $\Delta_\epsilon$ of the almost-interpolation $\Ec_{\epsilon, \X_n} f$, in terms of  $f$ and its
graph Laplacian $\Delta_{\epsilon, \X_n}$. 

\subsection{The double convolution and a heuristic idea of proof}

Let us first indicate a heuristic idea behind the proof of Theorem
\ref{thm:DiscreteToNonlocal}, via the introduced interpolation technique. 

Given $\X_n\in\M^n$ and $f:\X_n\to\R$, denote $F = f - \Ec_{\epsilon, \X_n} f$
and express the operator $\Ec_{\epsilon, \X_n} $ as:
\begin{equation}\label{eq:conv}
\begin{split}
\Ec&_{\epsilon, \X_n}  f(x)= \frac{1}{n d_{\epsilon, \X_n}(x)}\sum_{j=1}^n \eta_\epsilon(|x-x_j|)f(x_j)\\
&=\frac{1}{n d_{\epsilon, \X_n}(x)}\sum_{j=1}^n \eta_\epsilon(|x-x_j|)\big(\Ec_{\epsilon, \X_n}  f(x_j) + F(x_j)\big)\\
&=\frac{1}{n d_{\epsilon, \X_n}(x)}\sum_{k,j=1}^n
\frac{\eta_\epsilon(|x-x_j|)}{n d_{\epsilon, \X_n}(x_j)}
\eta_\epsilon(|x_j-x_k|)f(x_k) + \frac{1}{n d_{\epsilon, \X_n}(x)}\sum_{j=1}^n
\eta_\epsilon(|x-x_j|)F(x_j)\\ 
&=\frac{1}{n d_{\epsilon, \X_n}(x)}\sum_{k=1}^n \left[\sum_{j=1}^n
  \frac{\eta_\epsilon(|x_j-x_k|)}{n d_{\epsilon,  \X_n}(x_j)}\eta_\epsilon(|x-x_j|)\right] f(x_k) +
\Ec_{\epsilon, \X_n}  F(x). 
\end{split}
\end{equation}
The term in square brackets above depends only on the random point cloud
$\X_n$ and the choice of $\epsilon$. In particular, it is independent of
the function $f$ and of any of its regularity properties. 
The asymptotics of this expression can be
easily controlled with concentration of measure inequalities; this
is done in Lemma \ref{lem:CM2} below, similarly to Lemma
\ref{lem:CM1}. Consequently, the operator $\Ec_{\epsilon, \X_n}$ is
proved to be compatible
with the continuum averaging operator $\A_\epsilon$ in (\ref{Aepi}), and this is the core
idea behind the proof of Theorem \ref{thm:DiscreteToNonlocal}. 

\subsection{Asymptotics of the double convolution kernel in (\ref{eq:conv})}

We have the following:
\begin{lemma}\label{lem:CM2}
There exists $C,c>0$ depending on $\M$, $\rho$ and $\eta$, such that
for all $\epsilon\ll 1$ and all $t$ satisfying $\epsilon^2 \leq t \leq 1$, there holds:
\begin{equation*}
\begin{split}
\P_n&\Big(\sup_{x,y\in \M}\Big|\frac{1}{n}\sum_{i=1}^n
\frac{\eta_\epsilon(|x_i-x|)}{\rho(x_i)}\eta_\epsilon(|x_i-y|) -
\int_{B_\M(x,\epsilon)}\eta_\epsilon(|z-x|)\eta_\epsilon(|z-y|)
\dVol (z) \Big| \leq \frac{Ct}{\epsilon^m} \Big)
\\ & \geq 1-C(t\epsilon)^{-2m}\exp\left( -ct^2 n\epsilon^m \right).
\end{split}
\end{equation*}
\end{lemma}
\begin{proof}
{\bf 1.} For convenience, define the expressions:
\begin{equation*}
g_{n,\epsilon}(x,y) = \frac{1}{n}\sum_{i=1}^n \frac{\eta_\epsilon(|x_i-x|)}{\rho(x_i)}\eta_\epsilon(|x_i-y|),\quad
g(x,y) = \int_{B_\M(x,\epsilon)}\eta_\epsilon(|z-x|)\eta_\epsilon(|z-y|)\dVol (z). 
\end{equation*}
Applying Lemma \ref{lem:ch} to functions $\psi(z) =
\frac{\eta_\epsilon(|z-x|)\eta_\epsilon(|z-y|)}{n\rho(z)}$ for each $x,y\in\M$,
so that $\Psi^{\epsilon, x}=g_{n, \epsilon}(x,y)$, it follows that:
$$\P_n\Big(|g_{n,\epsilon}(x,y) - g(x,y)|\geq
Ct\epsilon^m\|\eta_\epsilon\|^2_{L^\infty(\M)}\Big)\leq 2\exp\big(-ct^2 n\epsilon^m\big).$$
By Corollary \ref{prop:ball} and
Proposition \ref{prop:annulus} (ii), we conclude that each of the
three types of events:
\begin{equation}\label{m5}
|g_{n,\epsilon}(x,y) - g(x,y) |\leq \frac{Ct}{\epsilon^m}, ~~~~ \sum_{i=1}^n
\mathds{1}_{|x_i-x|\leq \epsilon}\leq Cn\epsilon^m, ~~~~
\sum_{i=1}^n \mathds{1}_{(1-t)\epsilon\leq |x_i-x|\leq (1+t)\epsilon}\leq Ctn\epsilon^m
\end{equation} 
where $x,y,\epsilon, t$ are fixed, hold with probability $\P_n$ at
least $1-2\exp\big(-ct^2 n \epsilon^m\big)$. As in the proof of Lemma
\ref{lem:CM1}, let now $\M_{tn}\subset\M$ be the discrete set provided by
Proposition \ref{prop:covering}. Apply the union bound to obtain that
the event in which all conditions listed in (\ref{m5}) holds for all
$x,y\in\M_{t\epsilon}$, has probability at least
$1-C(t\epsilon)^{-2m}\exp\big(-ct^2 n \epsilon^m\big)$. For the rest
of the proof we assume this event holds.

\smallskip

{\bf 2.} Let $x,y\in\M$ and $\hat x, \hat y\in\M_{t\epsilon}$ with
$d_\M(x,\hat x)\leq t\epsilon$ and $d_\M(y,\hat y)\leq t\epsilon$.
Then we have:
\begin{equation}\label{m6}
\big|g_{n,\epsilon}(x,y) - g(x,y)\big| \leq \big|g_{n,\epsilon}(x,y) -
g_{n,\epsilon}(\hat x,\hat y)\big| + \big|g_{n,\epsilon}(\hat x,\hat
y) - g(\hat x,\hat y)\big| + \big|g(\hat x,\hat y) - g(x,y)\big|.
\end{equation}

We will estimate each term in the right hand side above. Firstly, by
the assumed (\ref{m5}), there holds: $\big|g_{n,\epsilon}(\hat x,\hat y) - g(\hat x,\hat
y)\big|\leq\frac{Ct}{\epsilon^m}$. For the third term, we observe that:
\begin{equation*}
\begin{split}
\big|g&(\hat x,\hat y) - g(x,y)\big|  \\ & = \Big|\int_{B_\M(\hat x,
  \epsilon)}\eta_\epsilon(|z-\hat x|) \eta_\epsilon(|z-\hat y|)\dVol (z) - 
\int_{B_\M(x, \epsilon)}\eta_\epsilon(|z-x|) \eta_\epsilon(|z-y|)\dVol (z)\Big|
\\ & \leq \int_{\bigcap_{q\in \{\hat x, \hat y, x, y\} \cap\M} B(q,
  \epsilon)} \Big|\eta_\epsilon(|z-\hat x|) \eta_\epsilon(|z-\hat y|) - 
\eta_\epsilon(|z-x|) \eta_\epsilon(|z-y|) \Big|\dVol (z) \\ & \quad 
+\int_{\big((B(\hat x, \epsilon) \cap B(\hat y, \epsilon))\triangle 
(B(x, \epsilon) \cap B(y, \epsilon))\big)\cap\M
} \|\eta_\epsilon\|_{L^\infty(\M)}^2.
\end{split}
\end{equation*}
The second integration domain above is included in:
$$ \Big(\big(B(\hat x, \epsilon)\triangle B(x,\epsilon)\big)\cup \big(B(\hat y,
\epsilon)\triangle B(y, \epsilon)\big)\Big)\cap\M \subset A_{t,\hat x} \cup A_{t,\hat y}$$
where we define $A_{t,\hat x}$ below (and $A_{t,\hat y}$ is defined analogously):
\begin{equation*}
\begin{split}
A_{t,\hat x}\doteq \big\{z\in\M;~ (1-t)\epsilon\leq |z-\hat x|\leq
(1+t)\epsilon\big\} 
\end{split}
\end{equation*}
The indicated inclusion follows from the simple set-theoretical fact that:
$(B_1\cap B_2)\triangle (B_3\cap B_4)\subset (B_1\triangle B_3)\cup
(B_2\triangle B_4)$, valid for any four sets $\{B_i\}_{i=1}^4$.

Estimating volumes of the annuli $A_{t,\hat x}$ and $A_{t,\hat y}$ by
Proposition \ref{prop:annulus} (i), we further obtain:
\begin{equation*}
\begin{split}
\big|g(\hat x,\hat y) - g(x,y)\big| &\leq
\mbox{Lip}_{\eta_\epsilon}\cdot \|\eta_\epsilon\|_{L^\infty(\M)}
\big(|x-\hat x| + |y-\hat y|\big) \cdot Vol_\M(B_\M(x,\epsilon)) \\ &
\quad + \|\eta_\epsilon\|_{L^\infty(\M)}^2\cdot \big(Vol_\M (A_{t,\hat x}) + Vol_\M (A_{t,\hat y})\big)
\\ & \leq \frac{Ct}{\epsilon^{2m+1}}\cdot t\epsilon\cdot\epsilon^m +
\frac{Ct}{\epsilon^m}\leq \frac{Ct}{\epsilon^m}.
\end{split}
\end{equation*}

Finally, we now bound the first term in the right hand side of
(\ref{m6}). Define the set of indices $I_x\doteq \{i;~ |x_i-x|\leq \epsilon\}$ and
the corresponding sets $I_{\hat x}$, $I_{y}$ and $I_{\hat
  y}$. Similarly to (\ref{m45}), we note:
\begin{equation*}
\begin{split}
& (1-t)\epsilon \leq |x_i-\hat x|\leq (1+t)\epsilon \qquad\mbox{for all }\; i\in I_x\triangle I_{\hat x},
\\ & (1-t)\epsilon \leq |x_i-\hat y|\leq (1+t)\epsilon \qquad\mbox{for all }\; i\in I_y\triangle I_{\hat y}.
\end{split}
\end{equation*}
This implies, in virtue of (\ref{m5}) and since, as above, $(I_x\cap
I_y)\triangle (I_{\hat x}\cap I_{\hat y})\subset (I_x\triangle I_{\hat x})\cup (I_y\triangle I_{\hat y})$:
\begin{equation*}
\begin{split}
\big| g_{n,\epsilon}&(x,y) - g_{n,\epsilon}(\hat x, \hat y)\big|  \leq
\frac{C}{n}\sum_{i=1}^n \Big|\eta_\epsilon(|x_i-x|) \eta_\epsilon(|x_i- y|) - 
\eta_\epsilon(|x_i-\hat x|) \eta_\epsilon(|x_i-\hat y|) \Big|\\ &
\leq \frac{C}{n}\sum_{i\in  I_x\cap  I_y\cap I_{\hat x}\cap I_{\hat y}
} \Big|\eta_\epsilon(|x_i-x|) \eta_\epsilon(|x_i- y|) - 
\eta_\epsilon(|x_i-\hat x|) \eta_\epsilon(|x_i-\hat y|) \Big| 
\\ & \quad 
+ \frac{C}{n}\sum_{i\in (I_x\triangle I_{\hat x})\cup (I_y\triangle I_{\hat y})
}\|\eta_\epsilon\|^2_{L^\infty(\M)}\\ & \leq
\frac{C}{n}\mbox{Lip}_{\eta_\epsilon}\cdot
\|\eta_\epsilon\|_{L^\infty}\sum_{i\in I_{\hat x}}\big(|x-\hat x| + |y-\hat y|\big)
 \\ & \quad + \frac{C}{n\epsilon^m} \Big(\sum_{i=1}^n\mathds{1}_{\{(1-t)\epsilon\leq
   |x_i-\hat x|\leq (1+t)\epsilon\}} + \sum_{i=1}^n\mathds{1}_{\{(1-t)\epsilon\leq
   |x_i-\hat y|\leq (1+t)\epsilon\}}\Big)
\\ & \leq
\frac{C}{\epsilon^{2m+1}}\cdot t\epsilon \sum_{i=1}^n\mathds{1}_{\{
   |x_i-\hat x|\leq \epsilon\}} + \frac{Ct}{\epsilon^m} \leq \frac{Ct}{\epsilon^m}.
\end{split}
\end{equation*}
As a result, (\ref{m6}) becomes: $|g_{n,\epsilon}(x,y) - g(x,y)|\leq\frac{Ct}{\epsilon^m}$. The proof is done.
\end{proof}

\subsection{Concluding steps of the proof}\label{subsec4.3}

We are now equipped to prove Theorem \ref{thm:DiscreteToNonlocal}. We
first state an alternative form of the theorem that may be of independent interest.  

Given a bounded, Borel function $g:\M\to\R$, define the averaging operator:
\begin{equation}\label{Aepi}
\A_\epsilon g(x) =
\frac{1}{d_\epsilon(x)}\int_{B_\M(x,\epsilon)}\frac{1}{\epsilon^m}\eta\Big(\frac{ d_\M (x,y)}{\epsilon}\Big)
g(y)\rho(y)\;\dVol(y),
\end{equation}
where the local scaled degree $d_\epsilon(x)$ has been introduced in (\ref{degi}).

\begin{theorem}\label{thm:DiscreteToNonlocal0}
There exists constants $C,c>0$ depending on $\M$, $\rho$, $\eta$, such that for
all $\epsilon, t\ll 1$ which satisfy $\epsilon^2 \leq t $, there holds:
\begin{equation*}
\begin{split}
&\P_n\Big( \Big|\mathcal{I}_{\epsilon, \X_n} f(x) - \A_{\epsilon}(\mathcal{I}_{\epsilon, \X_n} f)(x)\Big| \leq C\big(
  t \|f\|_{L^\infty(\X_n\cap B(x,2\epsilon))} + \|f - \mathcal{I}_{\epsilon, \X_n}
  f\|_{L^\infty(\X_n\cap B(x,\epsilon))} \big) \\ & 
  \mbox{ \hspace{8.5cm} for all } x\in\M \; \mbox{ and all } f:\X_n\to\R \Big)
\\ & \geq 1 - C(t \epsilon)^{-2m}\exp\left( -c t^2 n\epsilon^m\right).
\end{split}
\end{equation*} 
\end{theorem}
\begin{proof}
{\bf 1.} Let $\epsilon^2\leq t\leq 1$ with $\epsilon\ll 1$. 
By Remark \ref{m47}, Lemma \ref{lem:CM1} and Lemma \ref{lem:CM2} it follows that the event:
\begin{equation}\label{m7}
\begin{split}
& \; \sup_{x\in\M} \sum_{i=1}^n \mathds{1}_{\{|x_i-x|\leq 2\epsilon\}}\leq
Cn\epsilon^m \quad \mbox{ and }\quad  \sup_{x\in\M}\big|d_{\epsilon,
  \X_n}(x) - d_\epsilon(x)\big|\leq Ct, 
\\ & \sup_{x,y\in \M}\Big|\frac{1}{n}\sum_{i=1}^n \frac{\eta_\epsilon(|x_i-x|)}{\rho(x_i)}\eta_\epsilon(|x_i-y|) -
\int_{B_\M(x,\epsilon)}\eta_\epsilon(|z-x|)\eta_\epsilon(|z-y|) \dVol (z) \Big| \leq \frac{Ct}{\epsilon^m} 
\end{split}
\end{equation}
holds with probability at least $1-C(t\epsilon)^{-2m}\exp\big(-ct^2
n\epsilon^m\big)$. For the rest of the proof, we assume that this event occurs.
As in Corollary \ref{good_def}, the above imply
for $t\ll 1$, that both $d_{\epsilon, \X_n}$ and $d_\epsilon$ are
bounded away from $0$ by $\frac{1}{2}\min\rho$, and also:
\begin{equation}\label{m8}
\sup_{x\in\M} \big|d_{\epsilon, \X_n}(x)-\rho(x)\big|\leq Ct.
\end{equation}

\smallskip

{\bf 2.} To estimate $\big|\mathcal{I}_{\epsilon, \X_n} f(x) - \A_{\epsilon}(\mathcal{I}_{\epsilon, \X_n} f)(x)\big|$,
we introduce an intermediate quantity:
$$A_1\doteq \frac{1}{d_\epsilon(x)}\int_{B_\M(x,\epsilon)}
\eta_\epsilon(d_\M(x,z)) d_{\epsilon, \X_n}(z) \cdot \mathcal{I}_{\epsilon, \X_n} f(z)\dVol (z).$$ 
By the first and second conditions in (\ref{m7}) we get:
\begin{equation}\label{m9}
\begin{split}
\Big|\A_{\epsilon}(\mathcal{I}_{\epsilon, \X_n} f)(x) - A_1\Big| & \leq C\int_{B_\M(x,\epsilon)}\eta_\epsilon(d_\M(x,z))
\Big|d_{\epsilon, \X_n}(z)\rho(x)\Big| \cdot \big|\mathcal{I}_{\epsilon, \X_n} f(z)\big|\dVol (z) \\ & \leq
Ct\|\mathcal{I}_{\epsilon, \X_n} f\|_{L^\infty(B_\M(x,\epsilon))} \cdot\int_{B_\M(x,\epsilon)}\eta_\epsilon(d_\M(x,z))\dVol
(z) \\ & \leq Ct \|\mathcal{I}_{\epsilon, \X_n} f\|_{L^\infty(B_\M(x,\epsilon))}.
\end{split}
\end{equation}
Further, we may replace the quantity $A_1$ with:
\begin{equation*}
\begin{split} 
A_2 & \doteq \frac{1}{d_\epsilon(x)} \int_{B_\M(x,\epsilon)}\eta_\epsilon(|x-z|)d_{\epsilon, \X_n}(z)
\mathcal{I}_{\epsilon, \X_n} f(z) \dVol (z) \\ & =
\frac{1}{n d_\epsilon(x)} \sum_{k=1}^n
f(x_k)\int_{B_\M(x,\epsilon)}\eta_\epsilon(|x-z|) \eta_\epsilon(|x_k-z|)\dVol (z), 
\end{split}
\end{equation*}
at the expense of the following error, controlled with the help of (\ref{m2.7}):
\begin{equation}\label{m10}
\begin{split}
\big|A_1 - A_2\big| & \leq C\int_{B_\M(x,\epsilon)}\Big|\eta_\epsilon(|x-z|) -
\eta_\epsilon(d_\M(x,z))\Big| \dVol (z)\cdot \|\mathcal{I}_{\epsilon, \X_n} f\|_{L^\infty(B_\M(x,\epsilon))}
\\ & \leq \frac{C}{\epsilon^m} \mbox{Lip}_\eta\cdot \epsilon^2
Vol_\M(x,\epsilon) \cdot \|\mathcal{I}_{\epsilon, \X_n} f\|_{L^\infty(B_\M(x,\epsilon))}
\leq Ct \|\mathcal{I}_{\epsilon, \X_n} f\|_{L^\infty(B_\M(x,\epsilon))}.
\end{split}
\end{equation}
Next, we replace $A_2$ by the following new quantity:
$$A_3\doteq \frac{1}{n^2d_\epsilon(x)}
\sum_{k,j=1}^n\frac{\eta_\epsilon(|x-x_j|)}{\rho(x_j)}\eta_\epsilon(|x_k-x_j|) f(x_k),$$
using the first property in (\ref{m7}):
\begin{equation}\label{m11}
\big|A_2 - A_3\big| \leq \frac{Ct}{n\epsilon^m}
\sum_{k=1}^n|f(z_k)|\cdot \mathds{1}_{|x_k-x|\leq 2\epsilon}\leq Ct \|f\|_{L^\infty(\X_n\cap B(x,2\epsilon))}.
\end{equation}

\smallskip

{\bf 3.} Finally, the error between the quantity $A_3$ and $\mathcal{I}_{\epsilon, \X_n} f(x)$ is estimated by (\ref{eq:conv}):
\begin{equation*}
\begin{split}
\Big|\mathcal{I}_{\epsilon, \X_n} f(x) - A_3\Big| & \leq \frac{1}{n^2}
\sum_{k,j=1}^n|f(x_k)|\eta_\epsilon(|x-x_j|)\eta_\epsilon(|x_k-x_j|)\cdot
\Big|\frac{1}{d_{\epsilon, \X_n}(x) d_{\epsilon,\X_n}(x_j)} - \frac{1}{d_{\epsilon}(x) d_{\epsilon}(x_j)} \Big|
\\ & \quad + \big|\mathcal{I}_{\epsilon, \X_n} \big(f - \mathcal{I}_{\epsilon, \X_n} f\big)(x)\big|
\\ & \leq \frac{Ct}{n^2\epsilon^{2m}}\|f\|_{L^\infty(\X_n\cap
  B(x,2\epsilon))}\cdot \Big(\sum_{i=1}^n \mathds{1}_{|x_i-x|\leq
  2\epsilon}\Big)^2 + \big|\mathcal{I}_{\epsilon, \X_n} \big(f - \mathcal{I}_{\epsilon, \X_n} f\big)(x)\big|
\\ &  \leq Ct \| f\|_{L^\infty(\X_n\cap B(x, 2\epsilon))} +
\big|\mathcal{I}_{\epsilon, \X_n} \big(f - \mathcal{I}_{\epsilon, \X_n} f\big)(x)\big|, 
\end{split}
\end{equation*}
where the difference of the inverses of products of degrees is bounded
by $Ct$, in virtue of (\ref{m7}) and (\ref{m8}).
Together with (\ref{m9}), (\ref{m10}) and (\ref{m11}), the above estimate yields:
\begin{equation*}
\begin{split}
\Big|\mathcal{I}_{\epsilon, \X_n} f(x) - \A_{\epsilon}(\mathcal{I}_{\epsilon, \X_n} f)(x)\Big| & \leq 
Ct\Big( \|\mathcal{I}_{\epsilon, \X_n} f\|_{L^\infty(B_\M(x,\epsilon))} + \|f\|_{L^\infty(\X_n\cap B(x,2\epsilon))}\Big)
\\ & \quad + \big|\mathcal{I}_{\epsilon, \X_n} \big(f - \mathcal{I}_{\epsilon, \X_n} f\big)(x)\big|.
\end{split}
\end{equation*}
Noting the following two estimates, similar to the bounds obtained above:
\begin{equation*}
\begin{split}
& \|\mathcal{I}_{\epsilon, \X_n} f\|_{L^\infty(B_\M(x,\epsilon))} \leq\frac{1}{n\epsilon^m}
\|f\|_{L^\infty(\X_n\cap B(x,2\epsilon))}\cdot \sum_{i=1}^n \mathds{1}_{|x_i-x|\leq 2\epsilon}
\leq C \|f\|_{L^\infty(\X_n\cap B(x,2\epsilon))}, \\ &
\big|\mathcal{I}_{\epsilon, \X_n} \big(f - \mathcal{I}_{\epsilon, \X_n} f\big)(x)\big|\leq
C \|f - \mathcal{I}_{\epsilon, \X_n} f\|_{L^\infty(\X_n\cap B(x,\epsilon))},
\end{split}
\end{equation*}
the proof is done.
\end{proof}

\medskip

\noindent {\bf A proof of Theorem \ref{thm:DiscreteToNonlocal}.}

One directly checks that:
\begin{equation}\label{m12}
f(x_i) - \Ext f(x_i) = \frac{\epsilon^{2}}{\Deg(x_i)}\GL f(x_i)\qquad\mbox{for all } i=1,\ldots, n,
\end{equation}
and also:
$$\Ext f(x) - \A_\epsilon(\Ext f)(x)= \frac{\epsilon^{2}}{d_\epsilon(x)}\NL \big(\Ext f\big)(x)
\qquad\mbox{for all } \; x\in\M.$$
Assume that $\X_n$ is an element of the event with probability
estimated in Theorem \ref{thm:DiscreteToNonlocal0}, where we set
$t=\epsilon^2$. Given $f:\X_n\to\R$, we let $a\in\R$ and apply
Theorem \ref{thm:DiscreteToNonlocal0} to $f-a$, so that:
$$\left|\NL (\Ext f)(x)\right| \leq C\big( \|f-a\|_{L^\infty(\X_n\cap
    B(x,2\epsilon))}+ \|\GL f\|_{L^\infty(\X_n\cap
    B(x,\epsilon))}\big) \quad\mbox{ for all }\; x\in\M.$$ 
The first term in the right hand side above may be bounded by
$\osc_{\X_n\cap B(x,2\epsilon)}f$, upon choosing $a=u(x_j)$ for some $x_j\in B(x,2\epsilon)$.
Since the said event occurs with probability at least $1 - C(\epsilon^3)^{-2m}
\exp\big(-c\epsilon^{m+2}\big)$, the proof is done.
\endproof

\medskip

\begin{center}
{\bf \large PART 2}
\end{center}

\section{The Levi-Civit\`a quadrilateral}\label{LCQ_sec}

In this section, $(\M,g)$ is a smooth, compact, boundaryless, connected
and orientable manifold of dimension $m$. Towards further
applications, we derive a curvature-driven error estimate on geodesic distances in the
Levi-Civit\`a quadrilateral (see Figure \ref{fig_quadri}), named so in connection with the
Levi-Civit\`a parallelogram.

Recall that $d_\M$ denotes the geodesic distance, $Exp_x$ is
the exponential map and $\iota>0$
is the radius of injectivity (see section \ref{sec2.2}).
Given $x\neq y\in\M$ with $d_\M (x,y)<\frac{\iota}{3}$, and two
tangent vectors $v\in T_x\M, w\in T_y\M$ satisfying $|v|_x, |w|_y<\frac{\iota}{3}$, we consider the quantity:
$$L(s) =  d_\M \big(Exp_x(sv), Exp_y(sw)\big)^2\quad \mbox{ for all } s\in [0,1], $$
which keeps track of the squared distance between points along two geodesics emanating from the points $x$ and $y$ with directions $v$ and $w$ respectively--see Figure \ref{fig_quadri} below. 
Denote $t_0={dist}_\M(x,y)$ and define the flow of geodesics
$\gamma:[0,1]\times [0,t_0]\to\M$ as follows. Namely,
we set: $\gamma(s,0)=Exp_x(sv)$, $\gamma(s, t_0)=Exp_y(sw)$ and
request that $[0,t_0]\ni t\mapsto \gamma(s,t)$ is a
geodesic, for every $s\in [0,1]$. This implies that 
$\nabla_{\frac{d}{dt}\gamma}\frac{d}{dt}\gamma=0$ and further:
\begin{equation}\label{uno}
\Big|\frac{d}{dt}\gamma(s,t)\Big|_{\gamma(s,t)} =
\frac{1}{t_0} d_\M (\gamma(s,0), \gamma(s,t_0)\big) \quad \mbox{ for all } (s,t)\in [0,1]\times [0,t_0].
\end{equation}

\vspace{-3mm}

\begin{figure}[htbp]
\centering
\includegraphics[scale=0.5]{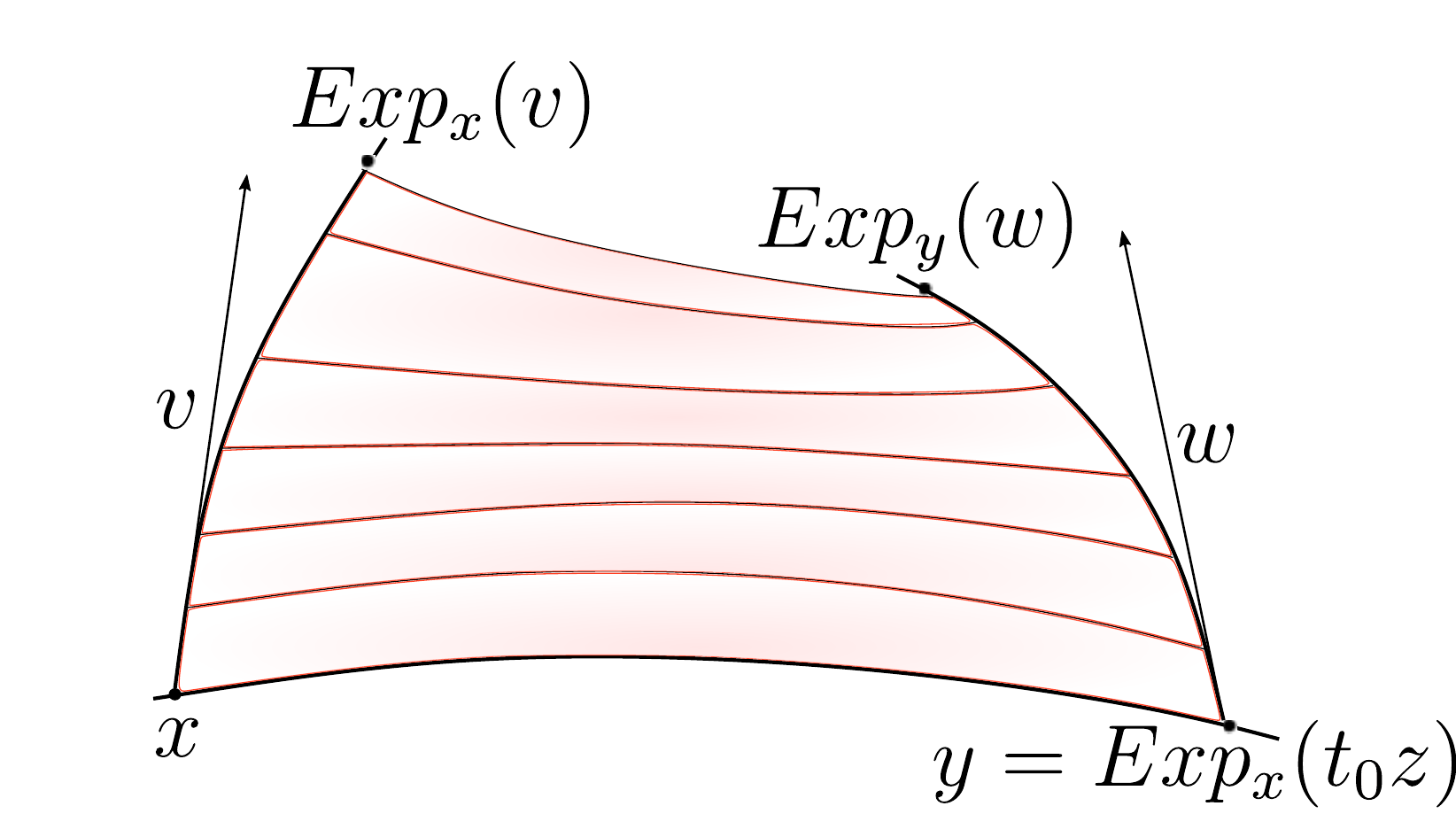}
\caption{{The quadrilateral in Lemma \ref{LCQ}, defined by its two
    vertices $x,y$ and the corresponding tangent vectors $v,w$.}}
\label{fig_quadri}
\end{figure}

In the following result we verify that the geodesic distance between
$Exp_x(v)$ and $Exp_y(w)$, which may be interpreted as the length of
one of the sides in a curvlinear quadrilateral on $\M$ (with other two
vertices $x$ and $y$, see Figure \ref{fig_quadri}), deviates from the corresponding length of the
side in a quadrilateral in $T_x\M$, only by controlled quadratic error terms:

\begin{lemma}\label{LCQ}
In the above context, for every $s\in[0,1]$ there holds:
\begin{equation}\label{due}
\begin{split}
& L'(s) = 2t_0 \Big(\Big\langle \frac{d}{ds}\gamma(s,t_0),
\frac{d}{dt}\gamma(s, t_0)\Big\rangle_{\gamma(s,t_0)} - \Big\langle \frac{d}{ds}\gamma(s,0),
\frac{d}{dt}\gamma(s, 0)\Big\rangle_{\gamma(s,0)}\Big)\\ 
& L''(s) = 2\Big|P_{\gamma(s,0),\gamma(s,t_0)} P_{x,\gamma(s,0)}v -
P_{y,\gamma(s,t_0)}w\Big|^2 + \mathcal{O}\big(t_0^2+s^2(|v|_x^2+|w|_y^2)\big),
\end{split}
\end{equation}
where $P_{a,b}$ denotes the parallel transport along the unique
geodesic connecting points $a,b\in\M$ with $ d_\M (a,b)<\iota$. In
particular, writing $y=Exp_x(t_0z)$ for some $z\in T_x\M$ with
$|z|_x=1$, we have the following expansion, valid for $s\in [0,1]$:
\begin{equation}\label{tre}
\begin{split}
L(s) & = t_0^2+2t_0s\langle P_{y,x}w-v,z\rangle_x +
s^2|P_{y,x}w-v|_x^2+\mathcal{O}\big(s^2t_0^2+s^4(|v|_x^2+|w|_y^2)
\big) \\ & = \big| t_0z + s(P_{y,x}w - v)\big|_x^2 + \mathcal{O}\big(s^2t_0^2+s^4(|v|_x^2+|w|_y^2) \big). 
\end{split}
\end{equation}
The Landau symbol $\mathcal{O}$ above is uniform with respect to $x,y, v,w$
and depends only on $(\M,g)$.
\end{lemma}
\begin{proof}
{\bf 1.} Note first that by (\ref{uno}), we have:
$$L(s) =  d_\M \big(\gamma(s,0),\gamma(s,t_0)\big) \cdot
\int_0^{t_0}\Big|\frac{d}{dt}\gamma(s,t)\Big|\dt = t_0 \int_0^{t_0}\Big|\frac{d}{dt}\gamma(s,t)\Big|^2\dt.$$
We now differentiate in $s$, use the symmetry lemma in
$\nabla_{\frac{d}{d t}\gamma}\frac{d}{d s}\gamma 
= \nabla_{\frac{d}{d s}\gamma}\frac{d}{d t}\gamma$, and the equation of geodesic
$\nabla_{\frac{d}{d t}\gamma}\frac{d}{d t}\gamma  =0$, to find that:
\begin{equation*}
\begin{split}
L'(s) & =  2t_0 \int_0^{t_0}\Big\langle
\frac{d}{dt}\gamma,\nabla_{\frac{d}{ds}\gamma}\frac{d}{dt}\gamma\Big\rangle\dt
= 2t_0 \int_0^{t_0}\Big\langle
\frac{d}{dt}\gamma,\nabla_{\frac{d}{dt}\gamma}\frac{d}{ds}\gamma\Big\rangle\dt\\
& = 2t_0 \int_0^{t_0}\frac{d}{dt}\Big\langle
\frac{d}{dt}\gamma,\frac{d}{ds}\gamma\Big\rangle\dt = 2t_0 \Big\langle \frac{d}{dt}\gamma(s,t),
\frac{d}{ds}\gamma(s, t)\Big\rangle\Big|_{t=0}^{t=t_0},
\end{split}
\end{equation*}
as claimed in (\ref{due}). To compute the second derivative of $L$, we observe that
$\nabla_{\frac{d}{d s}\gamma}\frac{d}{ds}\gamma(s,0)  =0$
and $\nabla_{\frac{d}{d s}\gamma}\frac{d}{ds}\gamma(s,t_0)  =0$, so that:
$$L''(s) = 2t_0 \Big\langle \nabla_{\frac{d}{ds}\gamma}\frac{d}{dt}\gamma(s,t),
\frac{d}{ds}\gamma(s,t)\Big\rangle\Big|_{t=0}^{t=t_0} = 
2t_0 \Big\langle \nabla_{\frac{d}{dt}\gamma}\frac{d}{ds}\gamma(s,t),
\frac{d}{ds}\gamma(s,t)\Big\rangle\Big|_{t=0}^{t=t_0}, $$
again in view of the symmetry lemma. Define the auxiliary vector field:
$$V(s,t) = \Big(1-\frac{t}{t_0}\Big)P_{\gamma(s,0),\gamma(s,t)}P_{x,\gamma(s,0)}v
+ \frac{t}{t_0} P_{\gamma(s,t_0),\gamma(s,t)}P_{y,\gamma(s,t_0)}w\in
T_{\gamma(s,t)}\M $$
for all $(s,t)\in [0,1]\times [0,t_0].$ Properties of parallel transport yield:
$$\nabla_{\frac{d}{dt}\gamma}V(s,t) = \frac{1}{t_0}\Big(P_{\gamma(s,t_0),\gamma(s,t)}P_{y,\gamma(s,t_0)}w
- P_{\gamma(s,0),\gamma(s,t)}P_{x,\gamma(s,0)}v\Big),$$
and since $V(s,0) = P_{x,\gamma(s,0)}v = \frac{d}{ds}\gamma(s,0)$ and
$V(s,t_0) = P_{y,\gamma(s,t_0)}w = \frac{d}{ds}\gamma(s,t_0)$, we obtain: 
\begin{equation}\label{quattro}
\begin{split}
L''(s) & =  2t_0 \int_0^{t_0}\frac{d}{dt}\Big\langle
\nabla_{\frac{d}{dt}\gamma}\frac{d}{ds}\gamma, V(s,t)\Big\rangle\dt \\ &
= 2 \int_0^{t_0}\Big\langle \nabla_{\frac{d}{dt}\gamma}\frac{d}{ds}\gamma, 
P_{\gamma(s,t_0),\gamma(s,t)}P_{y,\gamma(s,t_0)}w - P_{\gamma(s,0),\gamma(s,t)}P_{x,\gamma(s,0)}v
\Big\rangle\dt \\ & \quad\qquad
+ 2t_0 \int_0^{t_0}\Big\langle \Big(\nabla_{\frac{d}{dt}\gamma}\Big)^2\frac{d}{ds}\gamma, V(s,t)\Big\rangle\dt
\\ & = 2 \int_0^{t_0}\frac{d}{dt}\Big\langle \frac{d}{ds}\gamma, 
P_{\gamma(s,t_0),\gamma(s,t)}P_{y,\gamma(s,t_0)}w - P_{\gamma(s,0),\gamma(s,t)}P_{x,\gamma(s,0)}v
\Big\rangle\dt \\ & \quad\qquad + 2t_0 \int_0^{t_0}\Big\langle
R\Big(\frac{d}{dt}\gamma,\frac{d}{ds}\gamma\Big)\frac{d}{dt}\gamma, V(s,t)\Big\rangle\dt
\\ & \doteq 2A+2B.
\end{split}
\end{equation}
In the last step of (\ref{quattro}), we used the fact that
$\frac{d}{ds}\gamma(s,t)$ is a Jacobi field along the geodesic
$t\mapsto \gamma(s,t)$. Below, we estimate separately the two terms $A$ and $B$.

\smallskip

{\bf 2.} For the first term in the right hand side of (\ref{quattro}), we write:
\begin{equation*}
\begin{split}
A & =  \Big\langle
\frac{d}{ds}\gamma(s, t_0), P_{y,\gamma(s,t_0)}w - P_{\gamma(s,0),\gamma(s,t_0)}P_{x,\gamma(s,0)}v
\Big\rangle \\ & \quad\qquad - \Big\langle \frac{d}{ds}\gamma(s,0), 
P_{\gamma(s,t_0),\gamma(s,t)}P_{y,\gamma(s,t_0)}w - P_{x,\gamma(s,0)}v
\Big\rangle \\ & = \Big\langle
P_{y,\gamma(s, t_0)}w, P_{y,\gamma(s,t_0)}w - P_{\gamma(s,0),\gamma(s,t_0)}P_{x,\gamma(s,0)}v
\Big\rangle \\ & \quad\qquad - \Big\langle  P_{\gamma(s,0),\gamma(s,t_0)}P_{x,\gamma(s,0)}v,
P_{y,\gamma(s,t_0)}w - P_{\gamma(s,0),\gamma(s, t_0)}P_{x,\gamma(s,0)}v
\Big\rangle \\ & = \Big|P_{y,\gamma(s,t_0)}w - P_{\gamma(s,0),\gamma(s, t_0)}P_{x,\gamma(s,0)}v\Big|^2.
\end{split}
\end{equation*}
For the second term, we observe that:
\begin{equation*}
\begin{split}
|B| & \leq Ct_0^2\sup_{t\in [0,t_0]}  \Big| R\Big(\frac{d}{dt}\gamma, \frac{d}{ds}\gamma\Big)
\frac{d}{dt}\gamma(s,t)\Big|\cdot |V(s,t)| \\ & \leq Ct_0^2
\sup_{t\in [0,t_0]}  \Big(\Big| \frac{d}{dt}\gamma(s,t)\Big|^2\cdot
\Big|\frac{d}{ds}\gamma(s,t)\Big|\cdot |V(s,t)| \Big) \\ & \leq
C d_\M \big(\gamma(s,0), \gamma(s,t_0)\big)^2\cdot \big(|v|+|w|\big)\cdot 
\sup_{t\in [0,t_0], s\in [0.1]} \Big| \frac{d}{ds}\gamma(s,t)\Big|
\\ & \leq C t_0^2\big(|v|+|w|\big)+Cs^2\big(|v|+|w|\big)^3,
\end{split}
\end{equation*}
because: 
$$ d_\M \big(\gamma(s,0), \gamma(s,t_0)\big)\leq
 d_\M (\gamma(s,0),x)+ d_\M (x,y)+ d_\M (y,\gamma(s,t_0)) = t_0+s\big(|v|_x+|w|_y\big),$$ 
and because $\big|\frac{d}{ds}\gamma(s,
t)\big|\leq C$ for all $(s,t)\in [0,1]\times [0,t_0]$. This last
assertion, following from general properties of the exponential
map, will be justified in Proposition \ref{ddsbounded}. The proof of (\ref{due}) is complete.

\smallskip

{\bf 3.} To show (\ref{tre}), we Taylor expand $L$ at $0$ to get:
$$L(s) = L(0) + sL'(0) + \frac{s^2}{2}L''(\bar s) \qquad \mbox{ for
  some } \bar s\in [0,s],$$
where we noted that $L(0) = t_0$ and $L'(0) = 2t_0 \big(\langle
w,P_{x,y}z\rangle -\langle v,z\rangle\big) = 2t_0 \langle P_{y,x}w -
v,z\rangle$ by (\ref{due}). On the other hand:
\begin{equation}\label{cinque}
\begin{split}
P_{y,\gamma(\bar s,t_0)}&w - P_{\gamma(\bar s,0), \gamma(\bar s,t_0)}
P_{x,\gamma(\bar s,0)}v \\ & = 
P_{y,\gamma(\bar s,t_0)}\big(w - P_{x,y}v\big) + \Big( P_{y,\gamma(\bar s,t_0)} P_{x,y}v - 
P_{\gamma(\bar s,0), \gamma(\bar s,t_0)} P_{x,\gamma(\bar s,0)}v \Big).
\end{split}
\end{equation}
The first vector in the right hand side above has length $|w-P_{x,y}v|_y
= |P_{y,x}w-v|_x$. The second vector compares the parallel transport
of $v$ from $x$ to $\gamma(\bar s,t_0)$ by two different routes, and has
length equal to:
$$\big| v - P_{y,x} P_{\gamma(\bar s,t_0),y} P_{\gamma(\bar s,0), \gamma(\bar s,t_0)} P_{x,\gamma(\bar s,0)}v \big|_x
\leq C\big(t_0^2 + \bar s^2(|v|^2_x + |w|_y^2)\big),$$
as easily seen, for example, from equations of parallel transport
written in normal coordinates centered at $x$ (see section
\ref{sec2.3}). By (\ref{due}) and (\ref{cinque}), 
for every $\bar s\in [0,s]$ we thus get: 
\begin{equation*}
\begin{split}
\frac{1}{2}L''(\bar s) & = \Big|P_{y,\gamma(\bar s,t_0)}w -
P_{\gamma(\bar s,0), \gamma(\bar s,t_0)} P_{x, \gamma(\bar s,0)}
v\Big|^2 +\mathcal{O}\big(t_0^2 + s^2(|v|^2_x + |w|_y^2)\big) \\ & =
|P_{y,x}w - v|^2 + \mathcal{O}\big(t_0^2 + s^2(|v|^2_x + |w|_y^2)\big),
\end{split}
\end{equation*}
which ends the proof of (\ref{tre}) and of the Lemma.
\end{proof}

In the argument above we used the following observation:

\begin{proposition}\label{ddsbounded}
In the above context, there holds:
$$\Big|\frac{d}{ds}\gamma(s, t)\Big| = \mathcal{O}(1)\qquad\mbox{ for all }
(s,t)\in [0,1]\times [0,t_0], $$
where the bound $\mathcal{O}$ depends only on $(\M, g)$.
\end{proposition}
\begin{proof}
Recall (see section \ref{sec2.2}) that given $x\in\M$ and $v\in T_x\M$ with $|v|_x<\iota$, we have
$\psi(x,v) \doteq \big(x, Exp_x(v)\big)\in\M\times \M$, and that the
mapping $\psi$ is a smooth diffeomorphism onto its image. For
$x,y,v,w$ as in the definition of the Levi-Civit\`a quadrilateral, and
for all $s\in [0,1]$, $t\in [0, d_\M (x,y)]$, we observe that the following
composite mapping:
\begin{equation*}
\begin{split}
(x,y,v,w,s,&\; t)  \mapsto \Big(Exp_x(sv), Exp_y(sw), \frac{t}{ d_\M (x,y)} \Big) \\ & ~
{\overset{(\psi^{-1}, id)}\mapsto} \Big(Exp_x(sv), Exp^{-1}_{Exp_x(sv)}\big(Exp_y(sw)\big), \frac{t}{ d_\M (x,y)} \Big)
\\ & \quad \mapsto \Big(Exp_x(sv), \frac{t}{ d_\M (x,y)}
Exp^{-1}_{Exp_x(sv)}\big(Exp_y(sw)\big) \Big) \\ &
~ \mbox{  } \;  \mbox{ }
{\overset{\psi}\mapsto} \Big(Exp_x(sv), Exp_{Exp_x(sv)}\Big(\frac{t}{ d_\M (x,y)}  Exp^{-1}_{Exp_x(sv)}\big(Exp_y(sw)\big) \Big)\Big)
\\ & \quad = \big(Exp_x(sv), \gamma(s,t)\big)
\end{split}
\end{equation*}
is smooth in $s$, independently of $x,y,v,w$ and of the scaled
parameter $\frac{t}{ d_\M (x,y)}\in [0,1]$. In particular,
$\frac{d}{ds}\gamma(s,t)$ is uniformly bounded, as claimed.
\end{proof}

\section{The averaging operators
  \texorpdfstring{${\mathcal{A}}_\epsilon$}{Abare}, \texorpdfstring{
    ${\bar{\mathcal{A}}}_\epsilon$}{Ae}  and their mean value expansions} 
\label{secAver}

In this section, we work under the following hypothesis, that are less
restrictive than (\ref{H1}):

\vspace{-2mm}

\begin{equation*}\label{H2}\tag{{\bf H2}}
\left[\quad \;\; \mbox{\begin{minipage}{13.1cm}\vspace{1mm}
\begin{itemize}[leftmargin=*]
\item[(i)] $(\M,g)$ is a smooth, compact, boundaryless, connected
and orientable  manifold of dimension $m$,\vspace{1mm}
\item[(ii)] $\rho\in\mathcal{C}^2(\M)$ is a positive scalar field, \vspace{1mm}
\item[(iii)] $\eta:[0,1]\to\R$ is Borel regular, nonnegative,
  satisfying 
$\int_{B(0,1)\subset \R^m}\eta(|w|)\dw = 1$. It yields the positive coefficient:
$\sigma_\eta = \int_{B(0,1)}\langle w, e_1\rangle^2\eta(|w|) \dw.$
\vspace{1mm}
\end{itemize}
\end{minipage}}\;\;\, \right]
\end{equation*}

We consider the differential operator (see section \ref{sec2.3}):
$$Af \doteq \frac{1}{\rho^2}Div \big(\rho^2\nabla^*f) = \Delta f + 2\langle
\nabla^*f,\nabla^*(\log\rho)\rangle_x,$$
where $\Delta$ is the unweighted Laplace-Beltrami operator. In turn,
$A$ is a scaled version of $\Delta_\M$ as defined in
\eqref{eqn:LaplaceBeltrami}.  
The purpose of this section is to discuss two families of ``averaging''
operators, relative to the chosen $\eta$,  that approximate $A$ when
the domains of averaging shrink to a point.

For a bounded, Borel function $f:\M\to\R$ and $\epsilon\ll 1$, the first averaging operator (already
encountered in section \ref{subsec4.3}) is given via integrating on the geodesic
ball $B_\M(x,\epsilon)$:
\begin{equation*}
\begin{split}
\A_\epsilon f(x) & =
\frac{1}{d_\epsilon(x)}\int_{B_\M(x,\epsilon)}\frac{1}{\epsilon^m}\eta\Big(\frac{ d_\M (x,y)}{\epsilon}\Big)
f(y)\rho(y)\;\dVol(y)\\  
& \mbox{ where } ~~ d_\epsilon(x) = \int_{B_\M(x,\epsilon)}\frac{1}{\epsilon^m}\eta\Big(\frac{ d_\M (x,y)}{\epsilon}\Big)
\rho(y)\;\dVol(y) \quad\mbox{ for all } x\in\M.
\end{split}
\end{equation*}
The second operator involves integrating on the unit ball
$B(0,1)\subset T_x\M \simeq \R^m$ and interpreting the integrated
quantities in the normal coordinates centered at a given point $x$:
\begin{equation*}
\bar\A_\epsilon f(x) =
\int_{B(0,1)}\eta(|w|)\Big(1+ \epsilon\big\langle w, \nabla
(\log\rho)(0)\big\rangle\Big) f(\epsilon w)\dw \quad\mbox{ for all } x\in\M.
\end{equation*}
It is not hard to observe that both functions $\A_\epsilon f, \bar\A_\epsilon
f:\M\to\R$ are continuous and bounded by $C\sup_{x\in\M}|f(x)|$
where the constant $C$ depends only on $\rho$. 
The relation of $\A_\epsilon, \bar\A_\epsilon$ to $A$ is revealed in the next result, which may be also deduced
from the proof of \cite[Lemma 3.5]{calder2019improved}.

\begin{theorem}\label{consistency}
For every $f\in\mathcal{C}^3(\M)$ there holds:
$$ \A_\epsilon f(x), ~\bar\A_\epsilon f(x) = f(x) +
\frac{1}{2}\sigma_\eta\epsilon^2 Af(x) + \mathcal{O}(\epsilon^3)\|\nabla
f\|_{\mathcal{C}^2(\M)} \qquad\mbox{ for all } x\in\M,$$
where the Landau symbol $\mathcal{O}$ depends only on $(\M, g)$ and
$\rho$, and where $\sigma_\eta$ is as in (\ref{H2}).
\end{theorem}
\begin{proof}
{\bf 1.} We fix $x\in\M$ and calculate in normal coordinates
centered at $x$ (recall that in these coordinates
$Exp_x(\epsilon w)\in\M$ corresponds to $\epsilon w\in
B(0,\epsilon)\in T_x\M  \simeq\R^m)$:
\begin{equation*}
\begin{split}
\A_\epsilon f(x) - f(x) & = \frac{1}{d_\epsilon(x)} \int_{B_\M(x,\epsilon)}
\frac{1}{\epsilon^m}\eta\Big(\frac{ d_\M (x,y)}{\epsilon}\Big)\big(f(y)
- f(x)\big)\rho(y) \dVol (y) \\ & =
\frac{1}{d_\epsilon(x)}\int_{B(0,1)}\eta(|w|)\big(f(\epsilon
w)-f(0)\big)\rho(\epsilon w) \sqrt{\det[g_{ij}(\epsilon w)]}\dw,
\end{split}
\end{equation*}
where the last integration is on the Euclidean ball $B(0,1)\subset
\R^m$. Similarly, we have: $d_\epsilon(x) = \int_{B(0,1)}\eta(|w|)
\rho(\epsilon w) \sqrt{\det[g_{ij}(\epsilon w)]}\dw$. Since in normal
coordinates $g_{ij}(\epsilon w) =
\delta_{ij}+\mathcal{O}(\epsilon^2)$, it immediately follows
that: $\sqrt{\det[g_{ij}(\epsilon w)]} = 1+
\mathcal{O}(\epsilon^2)$. Taylor expanding $\rho\circ Exp_x$, we now get:
\begin{equation}\label{dep}
\begin{split}
d_\epsilon (x) = \int_{B(0,1)}\eta(|w|)\big(\rho(0)+\epsilon \langle\nabla \rho,
w\rangle +\mathcal{O}(\epsilon^2)\big)
\big(1+\mathcal{O}(\epsilon^2)\big) \dw = \rho(0) + \mathcal{O}(\epsilon^2),
\end{split}
\end{equation}
and thus $\frac{1}{d_\epsilon (x)} = \frac{1}{\rho(0)} +
\mathcal{O}(\epsilon^2)$. Similarly, Taylor expanding $f\circ Exp_x$ we obtain:
\begin{equation*}
\begin{split}
d_\epsilon(x)&\Big(\A_\epsilon f(x) - f(x)\Big)  =
\int_{B(0,1)}\eta(|w|)\Big(\epsilon\langle\nabla f(0),w\rangle
+\frac{\epsilon^2}{2}\langle \nabla^2 f(0) : w^{\otimes 2}\rangle +
\mathcal{O}(\epsilon^3)\|\nabla^3f\|_{\mathcal{C}^0(\M)} \Big)\\ & \qquad \qquad \qquad
\qquad \cdot\big(\rho(0) + \epsilon\langle\nabla\rho(0),w\rangle +
\mathcal{O}(\epsilon^2)\big)\cdot (1+\mathcal{O}(\epsilon^2))\dw  \\ &
= \epsilon \rho(0)\int_{B(0,1)}\eta(|w|)\langle \nabla f(0), w\rangle \dw 
\\ & \quad + \frac{\epsilon^2}{2}\rho(0) \int_{B(0,1)}\eta(|w|)\Big(\langle
\nabla^2 f(0): w^{\otimes 2}\rangle + 2\langle\nabla f(0),w\rangle\cdot\Big\langle
\frac{\nabla\rho (0)}{\rho(0)},w\Big\rangle\Big) \dw \\ & \quad +
\mathcal{O}(\epsilon^3)\big(\|\nabla f\|_{\mathcal{C}^0(\M)} +
\|\nabla^2f\|_{\mathcal{C}^0(\M)}  + \|\nabla^3f\|_{\mathcal{C}^0(\M)} \big) 
\\ & =  \epsilon^2 \rho(0)\Big\langle \nabla^2 f(0) + 2\nabla
f(0)\otimes \nabla(\log\rho)(0): \frac{1}{2}\int_{B(0,1)}\eta(|w|)
w^{\otimes 2}\dw\Big\rangle + \mathcal{O}(\epsilon^3)\|\nabla f\|_{\mathcal{C}^2(\M)}.
\end{split}
\end{equation*}
Since $\int_{B(0,1)}\eta(|w|)w^{\otimes 2}\dw=\sigma_\eta
Id_m$, it follows that:
$$\A_\epsilon f(x) - f(x) = \frac{1}{2}\epsilon^2\sigma_\eta \Big(\Delta f + 2\langle
\nabla f(0), \nabla(\log\rho) (0)\rangle\Big) +
\mathcal{O}(\epsilon^3)\|\nabla f\|_{\mathcal{C}^2(\M)}, $$
proving the claim.

\smallskip

{\bf 2.} For the averaging operator $\bar\A_\epsilon$, we likewise 
expand and calculate in normal coordinates: 
\begin{equation*}
\begin{split}
& \bar\A_\epsilon f(x) - f(x)  =
\int_{B(0,1)}\eta(|w|)\Big(1+\epsilon\langle w, \nabla(\log\rho)(0)\rangle\Big) 
\\ & \qquad \qquad \qquad \qquad \qquad \cdot
\Big(\epsilon\langle\nabla f(0),w\rangle
+\frac{\epsilon^2}{2}\langle \nabla^2 f(0) : w^{\otimes 2}\rangle +
\mathcal{O}(\epsilon^3)\|\nabla^3f\|_{\mathcal{C}^0(\M)} \Big) \dw \\ &
\quad = \frac{\epsilon^2}{2}\rho(0) \int_{B(0,1)}\eta(|w|)\Big(\langle
\nabla^2 f(0): w^{\otimes 2}\rangle + 2\langle\nabla f(0),w\rangle\cdot\langle
\nabla(\log\rho) (0),w\rangle\Big) \dw +
\mathcal{O}(\epsilon^3)\|\nabla f\|_{\mathcal{C}^1} 
\\ & \quad =  \frac{1}{2}\epsilon^2 \sigma_\eta \Big(\Delta f + 2\langle
\nabla f(0), \nabla(\log\rho) (0)\rangle\Big) +
\mathcal{O}(\epsilon^3)\|\nabla^2 f\|_{\mathcal{C}^1(\M)}. 
\end{split}
\end{equation*}
The proof is complete.
\end{proof}

\begin{remark}\label{rem_explA}
(i) When $\eta\equiv \frac{1}{|B(0,1)|}$, then $\sigma_\eta=\fint_{B(0,1)}\langle w, e_1\rangle^2\dw =
\frac{1}{2(m+2)}$, so that:
$$\frac{1}{\int_{B_\M(x,\epsilon)}\rho(y)\dVol
  (y)}\cdot\int_{B_\M(x,\epsilon)}f(y)\rho(y)\dVol (y) = f(x) +
\frac{\epsilon^2}{2(m+2)}\Delta f(x) +
\mathcal{O}(\epsilon^3)\quad\mbox{ as }\epsilon\to 0.$$
This observation  is consistent with the mean-value
expansion, valid for $\rho\equiv C$ and $A=\Delta$:
$$\fint_{B_\M(x,\epsilon)}f(y)\dVol (y) = f(x) +
\frac{\epsilon^2}{2(m+2)}\Delta f(x) + \mathcal{O}(\epsilon^3).$$
The above expansion is, in turn, the quantitative version of one of the
equivalent properties of harmonic functions defined on $\mathcal U\subset
\R^m$, namely: $f(x) = \fint_{B(x,r)}f(y)\dy$ for all $\bar
B(x,r)\subset\mathcal{U}$.

(ii) The relation of the averaging operator $\mathcal{A}_\epsilon$ to
$\Delta_\epsilon$ defined in (\ref{eq:NL}) is: 
$$\Delta_\epsilon f(x) = - d_\epsilon(x)\cdot \frac{
  \mathcal{A}_\epsilon f(x)-f(x)}{\epsilon^2}.$$
\end{remark}

\smallskip

We may directly estimate the difference of the two studied averages:

\begin{lemma}\label{AepAepbar}
For every Borel, bounded function $f:\M\to\R$ there holds:
$$\big|\A_\epsilon f(x) - \bar\A_\epsilon f(x) \big| =
\mathcal{O}(\epsilon^2) \| f\|_{L^\infty(B_\M(x,\epsilon))} \qquad\mbox{ for all } x\in\M.$$ 
When $f$ is continuous, with the modulus of continuity
$\omega(f,\cdot)$, then:
$$\big|\A_\epsilon f(x) - \bar\A_\epsilon f(x) \big| =
\mathcal{O}(\epsilon^2) \cdot \omega(f,\epsilon)\qquad\mbox{ for all } x\in\M.$$ 
In both bounds, the quantity $\mathcal{O}$ depends only on $(\M,g)$ and $\rho$.
\end{lemma}
\begin{proof}
{\bf 1.} For a fixed $x\in\M$, we compute in normal coordinates centered at $x$:
\begin{equation}\label{difAA}
\begin{split}
&\A_\epsilon f(x) - \bar\A_\epsilon f(x) \\ & = \frac{1}{d_\epsilon(x)}
\int_{B(0,1)} \eta(|w|) f(\epsilon w) \Big(\rho(\epsilon w) \sqrt{\det [g_{ij}(\epsilon w)]}
- d_\epsilon(x) \big(1+\epsilon\langle w, \nabla (\log \rho)(0)\rangle\big)\Big)\dw.
\end{split}
\end{equation}
We recall that $[g_{ij}]_{i,j=1\ldots m}$ denotes the Grammian matrix in normal
coordinates and that, as before, $\veps w$ is identified with $Exp_x(\veps w)$.
Observe that by (\ref{Tayg}) we get: 
\begin{equation*}
\begin{split}
\det[g_{ij}(\epsilon w)] & =\det\big[\delta_{ij} -
\frac{\epsilon^2}{3}R^i_{kjs}(0)w^kw^s +\mathcal{O}(\epsilon^3)\big]_{i,j=1\ldots m}
\\ & = 1 - \frac{\epsilon^2}{3}R_{kis}^i(0)w^kw^s + \mathcal{O}(\epsilon^3)
= 1 - \frac{\epsilon^2}{3}\big\langle [Ric_{ij}(0)]:w^{\otimes
  2}\big\rangle + \mathcal{O}(\epsilon^3),
\end{split}
\end{equation*}
where $Ric$ denotes the Ricci curvature tensor, here computed at $x$. Consequently:
$$ \sqrt{\det [g_{ij}(\epsilon w)]} = 1 - \frac{\epsilon^2}{6} \big\langle [Ric_{ij}(0)]:w^{\otimes
  2}\big\rangle + \mathcal{O}(\epsilon^3),$$
allowing to improve  (\ref{dep}) to a more precise expansion:
\begin{equation}\label{dep2}
\begin{split}
d_\epsilon (x) & = \int_{B(0,1)}\eta(|w|)\Big(\rho(0)+\epsilon \langle\nabla \rho,
w\rangle + \frac{\epsilon^2}{2}\langle \nabla^2\rho(0) : w^{\otimes 2}
\rangle + o(\epsilon^2) \Big) \\ & \qquad \qquad \cdot
\Big(1-\frac{\epsilon^2}{6} \langle [Ric_{ij}(0)]:w^{\otimes
  2}\big\rangle + \mathcal{O}(\epsilon^3)\Big)\dw \\ & = \rho(0) +
\epsilon^2 \Big\langle \frac{1}{2} \int_{B(0,1)} \eta(|w|) w^{\otimes
  2}\dw : \nabla^2\rho(0) - \frac{1}{3}\rho(0) [Ric_{ij}(0)]
\Big\rangle + o(\epsilon^2) \\ & = \rho(0) + \frac{1}{2}\sigma_\eta\epsilon^2
\Big(\Delta\rho (0) - \frac{1}{3}\rho(0) R(0)\Big) + o(\epsilon^2).
\end{split}
\end{equation}
Here, $R(0)$ denotes the scalar curvature at $x$ and $\Delta$ is the
unweighted Laplacian.

\smallskip

{\bf 2.} In conclusion, the integration factor in (\ref{difAA}) is:
\begin{equation*}
\begin{split}
\rho(\epsilon w)&\sqrt{\det [g_{ij}(\epsilon w)]} - d_\epsilon (x)
\big(1 +\epsilon\langle w, \nabla (\log\rho)(0)\rangle\big)  \\ & = 
\Big(\rho(0)+\epsilon \langle\nabla \rho(0),
w\rangle + \frac{\epsilon^2}{2}\langle \nabla^2\rho(0) : w^{\otimes 2}
\rangle + o(\epsilon^2) \Big) \cdot
\Big(1-\frac{\epsilon^2}{6} \langle [Ric_{ij}(0)]:w^{\otimes
  2}\big\rangle + \mathcal{O}(\epsilon^3)\Big)  \\ & \quad - \Big(\rho(0)
+\kappa_\eta\epsilon^2\big(\Delta\rho(0) -
\frac{1}{3}\rho(0)R(0)\big)+o(\epsilon^2)\Big)\cdot \big(1+\epsilon
\langle w, \nabla (\log\rho)(0)\rangle\big) \\ & =
\frac{\epsilon^2}{2}\Big\langle \nabla^2\rho(0) -
\frac{1}{3}\rho(0)[Ric_{ij}(0)] : w^{\otimes 2}\Big\rangle -
\frac{1}{2}\sigma_\eta \epsilon^2 \big(\Delta\rho(0) - \frac{1}{3}\rho(0)R(0)\big) + o(\epsilon^2).
\end{split}
\end{equation*}
Since the $\epsilon^2$-order term above integrates to $0$, against
$\eta(|w|)$ on $B(0,1)$, we see that for continuous $f$, the
difference in (\ref{difAA}) becomes:
\begin{equation*}
\begin{split}
\A_\epsilon &f(x) - \bar\A_\epsilon f(x)  \\ & = \frac{\epsilon^2}{d_\epsilon(x)}
\int_{B(0,1)} \eta(|w|)\big( f(\epsilon w) - f(0)\big) \\ &
\qquad\qquad \cdot\Big(\frac{1}{2}\Big\langle \nabla^2\rho(0) -
\frac{1}{3}\rho(0)[Ric_{ij}(0)] : w^{\otimes 2}\Big\rangle -
\frac{1}{2}\sigma_\eta\Delta\rho(0) - \frac{1}{6}\sigma_\eta\rho(0)R(0) + o(1)\Big)\dw
\\ & = \mathcal{O}(\epsilon^2) \cdot \omega(f,\epsilon).
\end{split}
\end{equation*}
The estimate for bounded $f$ is a consequence of the first equality above and of (\ref{dep2}).
\end{proof}

\section{A biased random walk modeled on the averaging operator
  \texorpdfstring{$\bar{\mathcal{A}}_\epsilon$}{Abare}} \label{sec_77}

In this section, we work under hypothesis (\ref{H2}).
We introduce a discrete, $\M$-valued
process whose dynamic programming principle reflects the averaging
operation in $\bar\A_\epsilon$. Ultimately, this process
$\{X_n^{\epsilon, x_0}\}_{n=0}^\infty$ will serve for tracking the
values of a solution $f$ in function of its average. 

\subsection{The local isometry field \texorpdfstring{$Q$}{Q}}\label{sec4.1}

We start by the following easy observation:

\begin{lemma}\label{withQ}
For a given $x\in\M$, let $Q$ be a linear isometry of $T_x\M$, such that for some $e\in
T_x\M$ with $|e|_x=1$ there holds:
\begin{equation}\label{Qprop}
|\nabla (\log \rho)(x)|_xQe = \nabla (\log \rho)(x).
\end{equation}
Then, for every bounded, Borel $f:\M\to\R$ we have:
\begin{equation*}
\begin{split}
\bar\A_\epsilon f(x) & = \big(1-\epsilon |\nabla (\log
\rho)(0)|\big)\int_{B(0,1)}\eta(|w|) f(\epsilon w)\dw \\ & \quad + \epsilon
|\nabla (\log \rho)(0)|\int_{B(0,1)}\eta(|w|) (1+\langle w, e\rangle)
f(\epsilon Qw)\dw,
\end{split}
\end{equation*}
where the formula above is written in the normal coordinates centered
at $x$, and the integration is on $B(0,1)\subset\R^m\simeq T_x\M$.
\end{lemma}
\begin{proof}
We change the variable in the second term, to obtain:
\begin{equation*}
\begin{split}
\epsilon |&\nabla (\log \rho)(0)|\int_{B(0,1)}\eta(|w|) (1+\langle w, e\rangle) f(\epsilon Qw)\dw
\\ & = \epsilon |\nabla (\log \rho)(0)|\int_{B(0,1)}\eta(|w|) (1+\langle w, Qe\rangle) f(\epsilon w)\dw
\\ & = \epsilon |\nabla (\log \rho)(0)|\int_{B(0,1)}\eta(|w|) f(\epsilon Qw)\dw
+ \epsilon \int_{B(0,1)}\eta(|w|) \langle w, \nabla(\log\rho)(0)\rangle) f(\epsilon w)\dw.
\end{split}
\end{equation*}
Summing with $\big(1-\epsilon|\nabla (\log\rho)(0)|\big)
\int_{B(0,1)}\eta(|w|) f(\epsilon w)\dw$, we get $\bar\A_\epsilon
f(x)$, as claimed.
\end{proof}

\smallskip

In what follows, we will define a specific field of isometries $Q$ on
a small geodesic ball $B_\M(x_0, \iota)$. The field $Q$ will be Borel-regular and will satisfy
(\ref{Qprop}) for the chosen unit vectors $e_1(x)= P_{x_0, x}e_1 \in T_x\M$. More
precisely, given the normal coordinates on $B_\M(x_0, \iota)$ centered
at $x_0\in\M$, we write $\big\{e_i=\frac{\partial}{\partial
  x^i}_{\mid x_0}\big\}_{i=1}^m$ for the orthonormal frame that spans
$T_{x_0}\M \simeq\R^m$. Then, at each $x\in B_\M(x_0, \iota)$ we consider the
parallel transported orthonormal frame $\big\{e_i(x) \doteq P_{x_0,
  x}e_i\big\}_{i=1}^m$ in $T_x\M \simeq\R^m$.

Define the spherical cups (interpreted as subsets of $T_{x_0}\M$):
\begin{equation*}
\begin{split} 
S_0^\pm=\Big\{w\in&\R^m;~ |w|=1 \mbox{ and } \langle w, \pm
e_1\rangle > \cos \frac{8\pi}{9}\Big\} \\ & \supset  
S_1^\pm=\Big\{w\in\R^m;~ |w|=1 \mbox{ and } \langle w, \pm
e_1\rangle > \cos \frac{5\pi}{9}\Big\},
\end{split}
\end{equation*}
and fix one rotation $Q_0\in SO(m)$ satisfying:
$$Q_0(S_j^-) = S_j^+ \quad\mbox{ for } j=0,1 \qquad \mbox{with } Q_0(-e_1)=e_1.$$
Let $\bar Q:S_0^+\to SO(m)$ be a smooth map such that:
$$\bar Q(e)e_1 = e \qquad\mbox{ for all } e\in S_0^+.$$
For $x\in B_\M(x_0,\iota)$, using the notation $\xi(x) = \frac{\nabla(\log\rho
  (x)}{|\nabla(\log\rho (x)|_x}$ whenever defined, we set $Q(x) \in SO(T_x\M)$ as follows:
\begin{equation}\label{Qfield}
Q(x) =\left\{\begin{array}{ll} P_{x_0, x}\circ \bar Q(P_{x,
      x_0}\xi(x))\circ P_{x, x_0} & \mbox{if } P_{x, x_0}\xi(x)\in  S_1^+\\
P_{x_0,x} \circ Q_0^T\bar Q\big(Q_0\cdot P_{x, x_0}\xi(x)\big) \circ
P_{x, x_0}  & \mbox{if } P_{x, x_0}\xi(x)\in  S_1^-\setminus S_1^+\\  Id_{T_x\M} &\mbox{otherwise.}
\end{array}\right.
\end{equation}
Clearly, property (\ref{Qprop}) holds for each $x$, with $e=e_1(x) = P_{x_0, x}e_1$.


\subsection{The probability spaces and the discrete process}\label{sec4.2}

We denote the following probability spaces on the collection of Borel subsets of $B(0,1)\subset \R^m$:
$$\Omega_1 = \Big(B(0,1), \mbox{Borel}, \eta(|w|)\dw\Big),\qquad 
\bar\Omega_1 = \Big(B(0,1), \mbox{Borel}, \big(1+\langle \bar w,
e_1\rangle\big)\eta(|\bar w|)\dbw\Big),$$
together with the probability space $\big((0,1), \mbox{Borel},
\db\big)$. The product measure on $\Omega_1\times
\bar\Omega_1\times (0,1)$ will be denoted $\nu_1$.
We also consider the infinite product measure space $(\Omega,\mathcal{F},\nu)$:
\begin{equation*}
\begin{split}
\Omega & = \big(\Omega_1\times\bar\Omega_1\times (0,1)\big)^{\mathbb{N}} \\ & =
\Big\{\omega = (w_i, \bar w_i, b_i)_{i=1}^\infty \mbox{ with } w_i,
\bar w_i\in B(0,1) \mbox{ and } b_i\in (0,1) \mbox{ for all }
i=1\ldots \infty\Big\}.
\end{split}
\end{equation*}
The product $\sigma$-algebra $\F$ has the natural filtration by
$\{\F_n\}_{n=0}^\infty$, each generated by finite products (identified as
subsets of $\F$) of $n$ Borel subsets of $B(0,1)^2\times (0,1)$. We
also write $\F_0=\{\emptyset, \Omega\}$.

Equivalently, $\{w_i\}_{i=1}^\infty$, $\{\bar w_i\}_{i=1}^\infty$ are
collections of independent random variables sampled from the
distributions on $T_{x_0}\M\simeq \R^m$ with respective densities: $\big(\eta
\one_{[0,1]}\big)(|w|)$ and $(1+\langle \bar w, e_1\rangle)\big(\eta
\one_{[0,1]}\big)(|w|)$, while $\{b_i\}_{i=1}^\infty$ are
$\iid$ samples uniformly distributed on the interval $(0,1)$.

\medskip

Fix $x_0\in\M$, $\epsilon\in (0,\frac{\iota}{2})$, and some sufficiently small radius
$r<\frac{\iota}{2}$. We introduce the sequence of random
variables $\{X_n^{\epsilon, x_0}:\Omega\to\M\}_{n=0}^\infty$, via the recursive formula:
\begin{equation}\label{processX}
\begin{split}
& X_0\equiv x_0\\
& X_{n+1}=\left\{\begin{array}{ll} X_n & \mbox{if } X_n\not\in B_\M(x_0, r)\\
Exp_{X_n}(\epsilon P_{x_0, X_n} w_{n+1}) &
    \mbox{otherwise and if } b_{n+1}\geq \epsilon|\nabla  (\log\rho)(X_n)|_{X_n} \\
Exp_{X_n}(\epsilon Q(X_n)\cdot P_{x_0, X_n}\bar w_{n+1}) &
\mbox{otherwise and if } b_{n+1}< \epsilon|\nabla (\log\rho)(X_n)|_{X_n}
\end{array}\right.
\end{split}
\end{equation}
where the isometry field $Q$ is defined in (\ref{Qfield}).
As explained in section \ref{sec4.1}, we identify each $w_{n+1}\in B(0,\epsilon)\subset\R^m $ with:
$\sum_{i=1}^m\langle w_{n+1}, e_i\rangle e_i(x)\in T_{x_0}\M$.
Consequently, whenever $X_n\in B_\M(x_0, r)$ and $b_{n+1}< \epsilon|\nabla (\log\rho)(X_n)|_{X_n}$,
then the updated process position $X_{n+1}$ equals the exponential map $Exp_{X_n}$ applied on the
following tangent vector $v\in T_{X_n}\M$:
\begin{equation*}
v=\left\{\begin{array}{ll} 
\epsilon P_{x_0, X_n} \bar Q\big( P_{X_n, x_0}\xi(X_n)\big) \bar w_{n+1} &
    \mbox{if } P_{X_n, x_0}\xi(X_n)\in S_1^+ \\
\epsilon P_{x_0, X_n} Q_0^T \bar Q\big(Q_0\cdot P_{X_n, x_0}\xi(X_n)\big) \bar w_{n+1} &
    \mbox{if } P_{X_n, x_0}\xi(X_n)\in S_1^-\setminus S_1^+ \\
0 & \mbox{otherwise.}
\end{array}\right.
\end{equation*}
By a further restriction on the smallness of the parameter $\epsilon$, so that
$\epsilon |\nabla (\log\rho)(x)|_{x}<\frac{1}{2}$ for all $x\in
B_\M(x_0, \iota/2)$, we ensure that the
update $X_{n+1}=Exp_{X_n}(\epsilon P_{x_0, X_n} w_{n+1})$ occurs with probability
at least $\frac{1}{2}$ at each step before exiting $B_\M(x_0, r)$. More precisely, defining the transition probability: 
$$\nu_1\doteq \eta(|w|)\big(1+\langle w, e_1\rangle\big)\eta(|\bar w|)\dw\dbw\db,$$ 
there holds: $\nu_1\big(X_{n+1}=Exp_{X_n}(\epsilon P_{x_0, X_n}
w_{n+1})\big)>\frac{1}{2}$ whenever $X_n\in B_\M(x_0, r)$.

\medskip

Above, as in (\ref{processX}) and in the sequel, we omit the
superscripts $^{\epsilon, x_0}$ and write $X_n$ instead of
$X_n^{\epsilon, x_0}$ when no ambiguity arises. 
By Lemma \ref{withQ} and a change of variable we directly obtain:
\begin{proposition}\label{process_good}
Given a bounded, Borel function $f:\M\to\R$, there holds for all $n\geq 0$:
$$\E \big(f\circ X_{n+1}\mid \F_n \big) = \left\{\begin{array}{ll}
   {\bar{\mathcal{A}}}_\epsilon f(X_n) & \mbox{if } X_n\in B_\M(x_0, r)\\
f(X_n) & \mbox{if } X_n\not\in B_\M(x_0, r). \end{array}\right.$$
\end{proposition}

\smallskip

We close this section by noting a bound related to stopping times. The
specific stopping time that we use will be chosen in the next section.

\begin{lemma}\label{exitfromr}
In the above context, let $\tau$ be a stopping time with respect to
the filtration $\{\F_n\}_{n=0}^\infty$. Assume that:
$$\tau\leq \min\{n\geq 0; ~ X_n\not\in B_\M(x_0, r)\}.$$ 
Then, if only $r\ll 1$ is sufficiently small (depending on $(\M,g)$
and $\rho, \eta$), and $\epsilon\ll r$, it follows that:
$$\E\big[ d_\M (X_\tau, x_0)^2\big] \leq \frac{3}{2}m\sigma_\eta\epsilon^2\E[\tau].$$
\end{lemma}
\begin{proof}
Fix $n\geq 0$. On the event $\{\tau>n\}$, Proposition \ref{process_good} yields:
\begin{equation*}
\begin{split}
\E\Big( d_\M &(X_{n+1}, x_0)^2 -  d_\M (X_n, x_0)^2\mid\F_n\Big) \\ & =
\int_{B(0,1)}\eta(|w|)\big(1+\epsilon\big\langle w,
\nabla(\log\rho)(0)\big\rangle\big)\Big( d_\M (Exp_{X_n}(\epsilon w), X_n)^2
-  d_\M (X_n, x_0)^2\Big)\dw,
\end{split}
\end{equation*}
where the integration on $B(0,1)\subset T_{X_n}\M$ is written in the
normal coordinates centered at $X_n$. We now apply 
Lemma \ref{LCQ} with $s=\epsilon$, $x=x_0$, $y=X_n$,
$v=0$ and the vector $w$ appropriately scaled, to get by (\ref{tre}):
\begin{equation*}
\begin{split}
 d_\M (Exp_{X_n}(\epsilon w), X_n)^2 = & \,  d_\M (X_n, x_0)^2
+ 2\epsilon\big\langle P_{X_n, x_0}w, Exp^{-1}_{x_0}(X_n)\big\rangle_{x_0}
\\ & +\epsilon^2|w|^2 +\mathcal{O}(\epsilon^2r^2 + \epsilon^4).
\end{split}
\end{equation*}
Observe that $\int_{B(0,1)}\eta(|w|)\langle w, e\rangle\dw=0$
for any fixed vector $e$, so consequently:
\begin{equation*}
\begin{split}
\E\Big( d_\M &(X_{n+1}, x_0)^2 -  d_\M (X_n, x_0)^2\mid\F_n\Big)
\\ & =\epsilon^2\int_{B(0,1)}\eta(|w|)|w|^2\dw +
2\epsilon^2\int_{B(0,1)}\eta(|w|) \big\langle w, \nabla
(\log\rho)(0)\big\rangle \mathcal{O}(r)\dw
+\mathcal{O}(\epsilon^3+\epsilon^2r^2) \\ & = m\sigma_\eta \epsilon^2 +\mathcal{O}(\epsilon^3+\epsilon^2r^2),
\end{split}
\end{equation*}
where the symbol $\mathcal{O}$ depends only on $(\M, g)$ and $\rho$.

It thus follows that for $r$ sufficiently small, the above quantity is
bounded from above by $3m\kappa_\eta$, implying that the sequence of random
variables:
$$\Big\{M_n \doteq  d_\M (X_{\tau\wedge n}, x_0)^2-\frac{3}{2}m\sigma_\eta
\epsilon^2 (\tau\wedge n)\Big\}_{n=0}^\infty$$ 
is a supermartingale with respect to the filtration $\{\F_n\}_{n=0}^\infty$. By Doob's
optional stopping theorem and in view of the stopping time $\tau$ being integrable, we obtain:
$$0=\E[M_0]\geq \E[M_\tau] = \E\big[ d_\M (X_\tau, x_0)^2\big] - \frac{3}{2}m\sigma_\eta\epsilon^2\E[\tau],$$
as claimed.
\end{proof}

\section{The coupling argument and the first approximate Lipschitz estimate}
\label{sec:NonLocalLipschitz} 

In this section, we work under hypothesis (\ref{H2}).
We derive a weaker estimate than that announced in
Theorem \ref{th_main1}, valid for the pairs of points whose distance
is bounded away from $0$ at the scale $\epsilon$. This 
restriction will be removed by the complementary estimate in section
\ref{sec_second}, whereas the main bounds in Theorems \ref{th_main1}
and \ref{th_main2}, will be closed in section \ref{sec_closebound}.
For the proofs, we use the probabilistic interpretation of the averaging operator
$\bar\A_\epsilon$ developed in section \ref{sec4.2}; we trace the value
of $f$ along the biased random walk $\{X_n^{\epsilon, x_0}\}_{n=0}^\infty$
started at $x_0$ and its coupled walk $\{Y_n\}_{n=0}^\infty$ started at a given $y_0$.
We choose a stopping time when the two processes almost coalesce (at
the scale $\epsilon$) or drift apart (at the scale $r$).

\begin{theorem}\label{Thestimate1}
Let $\epsilon\ll r\ll 1$. Then, for every bounded, Borel function
$f:\M\to\R$ and every $x_0, y_0\in\M$ satisfying $ d_\M (x_0, y_0)\in (3\epsilon, r)$, there holds:
\begin{equation*}
\begin{split}
|f(x_0) - f(y_0)|\leq ~& \sup \Big\{|f(x) - f(y)|;~ x\in B_\M(x_0, r+\epsilon), ~
 d_\M (x,y)\leq 3\epsilon\Big\} \\ &
 +\bigg(\frac{6(m+1)}{r}\|f\|_{L^\infty(B_\M(x_0, 2r+3\epsilon))} +
\frac{4r}{\sigma_\eta} \frac{\|\bar\A_\epsilon f -
  f\|_{L^\infty(B_\M(x_0, 2r))}}{\epsilon^2}\bigg) d_\M (x_0, y_0).
\end{split}
\end{equation*}
\end{theorem}
\begin{proof}
{\bf 1.} Fix $x_0, y_0, \epsilon$ as in the statement and recall the
process $\{X_n^{\epsilon, x_0}\}_{n=0}^\infty$ introduced in section
\ref{sec4.2}, relative to $x_0, \epsilon, r$. We now define a coupled process $\{Y_n:\M\to
\R\}_{n=0}^\infty$ by the following recursive construction:\vspace{1mm}
\begin{equation}\label{processY}
\begin{split}
& Y_0\equiv y_0\\
& Y_{n+1}=\left\{\begin{array}{ll} Y_n & \hspace{-3.33cm} \mbox{or }  d_\M (X_n, Y_n)\leq 3\epsilon
\mbox{ and }  d_\M (X_n, Y_n)\geq r \\ & \mbox{or } X_n\not\in B_\M(x_0, r)\\
Exp_{Y_n}\big(\epsilon P_{X_n, Y_n} \mbox{Refl}_n (P_{x_0, X_n} w_{n+1}) \big) &
    \hspace{-2mm} \mbox{otherwise}\\ &  \hspace{-2mm} \mbox{and if } b_{n+1}\geq \epsilon|\nabla  (\log\rho)(Y_n)|_{Y_n} \\
Exp_{Y_n}\big(\epsilon Q^{X_n}(Y_n)\cdot P_{X_n, Y_n} P_{x_0, X_n}\bar w_{n+1}\big) &
\mbox{otherwise}\\ &  \mbox{and if } b_{n+1}< \epsilon|\nabla (\log\rho)(Y_n)|_{Y_n}
\end{array}\right.
\end{split}
\end{equation}
Here, the linear map $\mbox{Refl}_n:T_{X_n}\M\to T_{X_n}\M$ denotes
the Householder projection, namely the reflection
across the hyperplane in $T_{X_n}\M$ which is perpendicular to the
tangent vector $Exp_{X_n}^{-1}(Y_n)$. Note that this vector is
well-defined and nonzero in the indicated sub-case of
(\ref{processY}). More precisely, we put:
$$ \mbox{Refl}_n(w) = w - \frac{2\big\langle w, Exp_{X_n}^{-1}(Y_n)\big\rangle}{ d_\M (X_n, Y_n)^2} Exp_{X_n}^{-1}(Y_n)
\qquad\mbox{for all } w\in T_{X_n}\M.$$
The rotation field $Q^{X_n}$ is defined for every $y\in B_\M(x_0, \iota)$ similarly to (\ref{Qfield}),
by means of $Q_0$ and $Q$ that have been introduced in section \ref{sec4.1}:
\begin{equation}\label{QfieldY}
Q^{X_n}(y) =\left\{\begin{array}{ll} P_{X_n, Y_n} P_{x_0, X_n}\circ
    \bar Q\big(P_{X_n, x_0} P_{y, X_n}\xi(y)\big)\circ P_{X_n x_0}
    P_{y, X_n} & \\ &  \hspace{-8.5cm} \mbox{if: } P_{X_n, x_0}\xi(X_n)\in  S_1^+ \mbox{ and }
    P_{X_n, x_0}P_{Y_n, X_n}\xi(y)\in S_0^+\\
& \hspace{-8.5cm} \mbox{or if: } P_{X_n, x_0}\xi(x)\in  S_1^-\setminus S_1^+ \mbox{ and }
    P_{X_n, x_0}P_{Y_n, X_n}\xi(y)\in S_0^+\setminus S_0^-\\
& \hspace{-8.5cm} \mbox{or if: } P_{X_n, x_0}\xi(x)  \mbox{ not defined ~ and }
    P_{X_n, x_0}P_{Y_n, X_n}\xi(y)\in S_1^+\\
P_{X_n, Y_n} P_{x_0, X_n} \circ Q_0^T\bar Q\big(Q_0\cdot P_{X_n, x_0} P_{y, X_n}\xi(y)\big) \circ
P_{X_n x_0} P_{y, X_n} & \\ &  \hspace{-8.5cm} 
\mbox{if: } P_{X_n, x_0}\xi(X_n)\in  S_1^-\setminus S_1^+ \mbox{ and }
    P_{X_n, x_0}P_{Y_n, X_n}\xi(y)\in S_0^-\\
& \hspace{-8.5cm} \mbox{or if: } P_{X_n, x_0}\xi(x)\in S_1^+ \mbox{ and }
    P_{X_n, x_0}P_{Y_n, X_n}\xi(y)\in S_0^-\setminus S_0^+\\
& \hspace{-8.5cm} \mbox{or if: } P_{X_n, x_0}\xi(x)  \mbox{ not defined ~ and }
    P_{X_n, x_0}P_{Y_n, X_n}\xi(y)\in S_1^-\setminus S_1^+\\
Id_{T_y\M} & \hspace{-8cm} \mbox{if } P_{X_n, x_0}P_{Y_n, X_n}\xi(y) \mbox{ not defined.}
\end{array}\right.
\end{equation}
We observe that $Q^{X_n}$ is Borel-regular and that each
$Q^{X_n}(y)\in SO(T_y\M)$ satisfies the property (\ref{Qprop}) with $e=P_{X_n, y}P_{x_0, X_n}e_1$.

\smallskip

Define now the random variable:
$$\tau^{\epsilon, x_0, y_0} \doteq \min\big\{ d_\M (X_n, Y_n)\leq 3\epsilon
~\mbox{ or } ~ d_\M (X_n, Y_n)\geq r ~\mbox{ or } ~X_n\not\in B_\M(x_0, r)\big\},$$
which is a stopping time relative to the filtration
$\{\F_n\}_{n=0}^\infty$. In particular $\nu(\tau<\infty)=1$, as in view of
the definitions (\ref{processX}), (\ref{processY}) and the assumed smallness
$\epsilon$ it follows that:
\begin{equation*}
\begin{split}
\nu_1\big(&X_{n+1}= Exp_{X_n}\big(\epsilon P_{x_0, X_n} w_{n+1}
\big)\big)>\frac{1}{2} \\ & \mbox{ and } ~\nu_1\big(Y_{n+1}=Exp_{Y_{n}}
\big(\epsilon P_{X_n, Y_n} \mbox{Refl}_n (P_{x_0, X_n} w_{n+1}) \big)
\big)>\frac{1}{2} \qquad \mbox{for } n<\tau.
\end{split}
\end{equation*}
Similarly to Proposition \ref{process_good}, we apply the
change of variable $P_{X_n, x_0}P_{Y_n, X_n}$ to conclude:
\begin{equation}\label{sei}
\E \big(f\circ Y_{n+1}\mid \F_n \big) = \left\{\begin{array}{ll}
   {\bar{\mathcal{A}}}_\epsilon f(Y_n) & \mbox{if } n<\tau\\
f(Y_n) & \mbox{if } n\geq\tau. \end{array}\right.
\end{equation}

\smallskip

{\bf 2.} We now observe that the sequence of random variables:
$$\big\{M_n\doteq |f\circ X_{\tau\wedge n} - f\circ Y_{\tau\wedge n} |
+ 2\|\bar\A_\epsilon f - f\|_{L^\infty(B_\M(x_0, 2r))}\cdot (\tau\wedge n)\big\}_{n=0}^\infty$$
is a submartingale with respect to the filtration
$\{\F_n\}_{n=0}^\infty$, in virtue of the following bound:
\begin{equation*}
\begin{split}
\E\big(|f\circ &X_{n+1} - f\circ Y_{n+1} | - |f\circ X_{n} - f\circ
Y_{n} | ~ \mid \F_n\big)  \\ & \geq \big|\E\big(f\circ X_{n+1} -
f\circ Y_{n+1} \mid \F_n\big)\big| - |f\circ X_{n} - f\circ Y_{n} | \\ & =  
\big|\bar\A_\epsilon f(X_{n}) - \bar\A_\epsilon f(Y_{n}) \big| - |f\circ X_{n} - f\circ Y_{n} | \\ & 
\geq -\big( |\bar\A_\epsilon f(X_{n}) - f(X_n)| + |\bar\A_\epsilon f(Y_{n}) - f(Y_n)|\big)
\\ & \geq -2 \|\bar\A_\epsilon f - f\|_{L^\infty(B_\M(x_0, 2r))},
\end{split}
\end{equation*}
valid on the event $\{\tau>n\}$ by Proposition \ref{process_good} and
(\ref{sei}). On the other hand, when $\{n\geq \tau\}$, it trivially
follows that: $\E\big(|f\circ X_{\tau\wedge(n+1)} - f\circ
Y_{\tau\wedge(n+1)} | - |f\circ X_{\tau\wedge n} - f\circ
Y_{\tau\wedge n} | ~ \mid \F_n\big)  = 0 \geq -2 \|\bar\A_\epsilon f - f\|_{L^\infty(B_\M(x_0, 2r))}$.
Consequently:
\begin{equation}\label{sette}
\begin{split}
|f(&x_0) - f(y_0)|  = \E[M_0]\leq \E[M_\tau] = \E\big[|f\circ X_\tau -
f\circ Y_\tau|\big] + 2\|\bar\A_\epsilon f - f\|_{L^\infty (B_\M(x_0, 2r))}\cdot
\E[\tau] \\ & \leq \int_{\{ d_\M  (X_\tau, Y_\tau)\leq 3\epsilon\}} |f(X_\tau) -
f(Y_\tau)|\;\mbox{d}\nu + 2\|f\|_{L^\infty (B_\M(x_0, 2r+3\epsilon))}\cdot
\nu\big( d_\M (X_\tau,Y_\tau)>3\epsilon\big) \\ & \quad + 2\|\bar\A_\epsilon f - f\|_{L^\infty (B_\M(x_0, 2r))}\cdot \E[\tau],
\end{split}
\end{equation}
by Doob's optional stopping theorem. We further write:
\begin{equation}\label{otto}
\nu\big( d_\M (X_\tau,Y_\tau)>3\epsilon\big) \leq \nu\big( d_\M (X_\tau, Y_\tau)\geq r\big) 
+ \nu\big( d_\M (X_\tau, x_0)\geq r\big).
\end{equation}
In the remaining part of the proof, we will estimate the three
quantities:
\begin{equation}\label{otto2}
\nu\big( d_\M (X_\tau, Y_\tau)\geq r\big), \qquad \nu\big( d_\M (X_\tau, x_0)\geq r\big),
\qquad \E[\tau],
\end{equation}
and derive the almost Lipschitz estimate claimed in the Theorem, from (\ref{sette}).

\smallskip

{\bf 3.} We start by estimating the following conditional expectation,
on the event $\{\tau>n\}$:
\begin{equation}\label{nove}
\begin{split}
\E\big(& d_\M (X_{n+1}, Y_{n+1})^2 -  d_\M (X_n, Y_n)^2\mid\F_n\big)  
\\  = & ~\big(1 -\epsilon\max\big\{|\nabla (\log\rho)(X_n)|_{X_n}, |\nabla (\log\rho)(Y_n)|_{Y_n}\big\}\big)\cdot
\\ & \quad \cdot \int_{B(0,1)}\eta(|w|) \Big( d_\M \big(Exp_{X_n}(\epsilon w),
Exp_{Y_n}(\epsilon P_{X_n, Y_n}\mbox{Refl}_n w)\big)^2 -  d_\M (X_n,
Y_n)^2\Big)\dw \\ & + \epsilon\|\nabla (\log
\rho)\|_{L^\infty(\M)}\cdot \int_{B(0,1)^2}\big|  d_\M (Exp_{X_n}(\epsilon
q_1),  Exp_{Y_n}(\epsilon q_2))^2 -   d_\M (X_n, Y_n)^2\big|\;\mbox{d}\nu_1
\end{split}
\end{equation}
where both integrations on $B(0,1)\subset T_{X_n}\M$ are written
in normal coordinates centered at $X_n$. In the second integrand
above, vectors $q_1\in T_{X_n}\M$, $q_2\in T_{Y_n}\M$
satisfy $|q_1|_{X_n}, |q_2|_{Y_n}\leq 1$. Thus, it follows by (\ref{tre}) that the error quantity in (\ref{nove}) is bounded by:
$$\|\nabla (\log\rho)\|_{L^\infty (B_\M(x_0, 2r))}\cdot \mathcal{O}(\epsilon^2
 d_\M (X_n, Y_n) + \epsilon^3) \leq
 \|\nabla(\log\rho)\|_{L^\infty(\M)}\cdot \mathcal{O}(\epsilon^2 r + \epsilon^3).$$
The first term in the right hand side of
(\ref{nove}) may be estimated through
applying Lemma \ref{LCQ} to $x=X_n$, $y=Y_n$, $s=\epsilon$, and the
corresponding variation vectors $w\in T_{X_n}\M$ and $P_{X_n, Y_n}(\mbox{Refl}_nw)\in T_{Y_n}\M$. We obtain:
\begin{equation*}
\begin{split}
 d_\M \big(&Exp_{X_n}(\epsilon w),
Exp_{Y_n}(\epsilon P_{X_n, Y_n}\mbox{Refl}_n w)\big)^2 -  d_\M (X_n,
Y_n)^2 \\ & = 2\epsilon \big\langle \mbox{Refl}_n w - w, Exp_{X_n}^{-1}(Y_n)\big\rangle_{X_n}
+\epsilon^2 |\mbox{Refl}_n w - w|^2_{X_n} + \mathcal{O}
(\epsilon^2 r^2 + \epsilon^4) \\ & = -4\epsilon \big\langle w, Exp_{X_n}^{-1}(Y_n)\big\rangle_{X_n}
+ 4\epsilon^2 \Big\langle w, \frac{Exp_{X_n}^{-1}(Y_n)}{|Exp_{X_n}^{-1}(Y_n)|}\Big\rangle_{X_n}^2 + \mathcal{O}
(\epsilon^2 r^2 + \epsilon^4),
\end{split}
\end{equation*}
and thus (\ref{nove}) becomes, provided that $r$ and $\epsilon$ are sufficiently small:
\begin{equation}\label{dieci}
\begin{split} 
&\E\big( d_\M (X_{n+1}, Y_{n+1})^2 -  d_\M (X_n, Y_n)^2\mid\F_n\big)  
\\ & = (1-\mathcal{O}(\epsilon)) \cdot
\Big(4\epsilon^2\int_{B(0,1)}\eta(|w|) \Big\langle w,
\frac{Exp_{X_n}^{-1}(Y_n)}{|Exp_{X_n}^{-1}(Y_n)|}\Big\rangle^2\dw + \mathcal{O} (\epsilon^2 r^2 +
\epsilon^4)\Big) \\ & \quad + \mathcal{O}(\epsilon^2 r + \epsilon^3) \\ &
= 4\epsilon^2\sigma_\eta + \mathcal{O} (\epsilon^2 r^2 + \epsilon^3)
\geq 3\epsilon^2\sigma_\eta.
\end{split}
\end{equation}

\smallskip

{\bf 4.} We continue by estimating another conditional expectation on the event $\{\tau>n\}$:
\begin{equation}\label{undici}
\begin{split}
\E\big(& d_\M (X_{n+1}, Y_{n+1}) -  d_\M (X_n, Y_n)\mid\F_n\big)  
\\  \leq & ~\big(1 -\epsilon\max\big\{|\nabla (\log\rho)(X_n)|_{X_n}, |\nabla (\log\rho)(Y_n)|_{Y_n}\big\}\big)\cdot
\\ & \quad \cdot \int_{B(0,1)}\eta(|w|) \Big|  d_\M \big(Exp_{X_n}(\epsilon w),
Exp_{Y_n}(\epsilon P_{X_n, Y_n}\mbox{Refl}_n w)\big) -  d_\M (X_n, Y_n)\Big|\dw \\ & + \Big| 
|\nabla (\log\rho)(X_n)|_{X_n} - |\nabla (\log\rho)(Y_n)|_{Y_n}\Big|
\cdot \mathcal{O}(\epsilon^2)
\\ & + \epsilon\min\big\{|\nabla (\log\rho)(X_n)|_{X_n}, |\nabla (\log\rho)(Y_n)|_{Y_n}\big\}\cdot 
\\ & \quad \cdot \int_{B(0,1)} \Big| d_\M \big(Exp_{X_n}(\epsilon Q(X_n)\bar
w) , Exp_{Y_n}(\epsilon Q^{X_n}(Y_n) P_{X_n, Y_n}\bar w)\big) -
 d_\M (X_n, Y_n)\Big|\mbox{d}\nu_1(\bar w) \\ & \doteq A+ \mathcal{O}(\epsilon^2 r) + B,
\end{split}
\end{equation}
where the integration on $B(0,1)\subset T_{X_n}\M$ is
written in the normal coordinates centered at $X_n$. 

To estimate the term $A$, we first rewrite the expansion (\ref{tre}) in the following form:
$$L(\epsilon) = |t_0 z + \epsilon(P_{y,x} w -v)|^2 + \mathcal{O}(\epsilon^2t_0^2+\epsilon^4).$$
When $t_0>3\epsilon$ and $|v|_x, |w|_y\leq 1$, which imply: 
$|t_0z + \epsilon(P_{y,x}w - v)|\geq t_0-2\epsilon \geq \frac{t_0}{3}$, Taylor
expanding the square root function yields:
\begin{equation*}
\begin{split}
L(\epsilon)^{1/2} & = |t_0z + \epsilon(P_{y,x}w - v)| + \frac{1}{2 |t_0z
  + \epsilon(P_{y,x}w - v)|}\mathcal{O}(\epsilon^2 t_0^2 + \epsilon^4)
\\ & = |t_0z + \epsilon(P_{y,x}w - v)| + \mathcal{O}(\epsilon^2 t_0 + \epsilon^3).
\end{split}
\end{equation*}
Hence, for all $w\in B(0,1)\subset T_{X_n}\M$ we obtain:
\begin{equation*}
\begin{split}
 d_\M &\big(Exp_{X_n}(\epsilon w),
 Exp_{Y_n}(\epsilon P_{X_n, Y_n}\mbox{Refl}_n w)\big) -  d_\M (X_n, Y_n) \\ & = L(\epsilon)^{1/2} - |t_0z| 
\\ & = \Big| Exp^{-1}_{X_n}(Y_n) - \frac{2\epsilon\langle w,
  Exp^{-1}_{X_n}(Y_n) \rangle}{ d_\M (X_n, Y_n)^2} \cdot
  Exp^{-1}_{X_n}(Y_n) \Big| - |Exp^{-1}_{X_n}(Y_n)| 
+ \mathcal{O}(\epsilon^2 r + \epsilon^3).
\end{split}
\end{equation*}
Using once more the fact that:
$\epsilon|\mbox{Refl}_n w - w|\leq 2\epsilon |w|<2\epsilon < |Exp^{-1}_{X_n}(Y_n)|, $
we conclude that the the difference of the first two terms in the
right hand side above is symmetric in $w$, and hence integrates to $0$ on $B(0,1)$. This
yields the bound on $A$:
\begin{equation}\label{dodici}
A\leq \mathcal{O}(\epsilon^2 r + \epsilon^3).
\end{equation}

\smallskip

{\bf 5.} Similarly, we now estimate the quantity $B$ in (\ref{undici}) by:
\begin{equation*}
\begin{split}
B  \leq  &\;\epsilon\min\big\{|\nabla (\log\rho)(X_n)|_{X_n}, |\nabla (\log\rho)(Y_n)|_{Y_n}\big\}\cdot 
\\ & \quad \cdot \int_{B(0,1)} \Big| \big|Exp_{X_n}^{-1}(Y_n) +
\epsilon \big(P_{Y_n, X_n} Q^{X_n}(Y_n) P_{X_n, Y_n}\bar w - Q(X_n)\bar w\big)\big| 
\\ & \qquad\qquad\qquad \qquad\qquad\qquad \qquad\qquad\qquad
\qquad\qquad\qquad - |Exp_{X_n}^{-1}(Y_n)|\Big| 
+\mathcal{O}(\epsilon^2)\;\mbox{d}\nu_1 (\bar w) \\  \leq & \;\epsilon^2
\min\big\{|\nabla (\log\rho)(X_n)|_{X_n}, |\nabla (\log\rho)(Y_n)|_{Y_n}\big\}\cdot 
\\ & \quad \cdot \sup_{\bar w\in B(0,1)\subset T_{X_n}\M} \Big|
P_{Y_n, X_n} Q^{X_n}(Y_n) P_{X_n, Y_n}\bar w - Q(X_n)\bar w\Big|_{X_n} + \mathcal{O}(\epsilon^3).
\end{split}
\end{equation*}
Clearly, the supremum in the above expression is bounded by $2$. This implies:
$$B\leq \frac{1}{2}\epsilon^2\sigma_\eta + \mathcal{O}(\epsilon^3) \qquad
\mbox{when } \min\big\{|\nabla (\log\rho)(X_n)|_{X_n}, |\nabla
(\log\rho)(Y_n)|_{Y_n}\big\}\leq \frac{\sigma_\eta}{4}.$$

In the opposite case, when $ \min\big\{|\nabla (\log\rho)(X_n)|_{X_n}, |\nabla
(\log\rho)(Y_n)|_{Y_n}\big\}> \frac{\sigma_\eta}{4}$, we get:
\begin{equation}\label{dodici2}
\begin{split}
\big|P_{Y_n, X_n}\xi(Y_n) - \xi(X_n)\big| & \leq 2\cdot\frac{\big|P_{Y_n,
    X_n}\nabla (\log\rho)(Y_n) - \nabla (\log\rho)(X_n)\big|_{X_n}}{ \min\big\{|\nabla (\log\rho)(X_n)|_{X_n}, |\nabla
(\log\rho)(Y_n)|_{Y_n}\big\}} \\ & \leq \frac{2\cdot \mbox{Lip}_{\nabla(\log\rho)}}{ \min\big\{|\nabla (\log\rho)(X_n)|_{X_n}, |\nabla
(\log\rho)(Y_n)|_{Y_n}\big\}} d_\M (X_n, Y_n),
\end{split}
\end{equation}
where $\mbox{Lip}_{\nabla(\log\rho)}$ denotes the Lipschitz constant of $\nabla(\log\rho)$ on $\M$.
In particular, for $r$ sufficiently small there holds: $\big|P_{Y_n, X_n}\xi(Y_n) - \xi(X_n)\big| \leq
\frac{8\cdot \mbox{Lip}_{\nabla(\log\rho)}}{\sigma_\eta} r\ll 1$. In
conclusion, either $P_{X_n, x_0}\xi(X_n)\in S_1^+$ and $P_{X_n,
  x_0}P_{Y_n, X_n}\xi(Y_n)\in S_0^+$ so that:
\begin{equation*}
\begin{split} 
P_{Y_n, X_n}Q^{X_n}(Y_n)&P_{X_n, Y_n}\bar w - Q(X_n)\bar w \\ & = P_{X_0,
  X_n}\Big( \bar Q\big(P_{X_n, x_0}P_{Y_n, X_n}\xi(Y_n)\big) - \bar
Q\big(P_{X_n, x_0}\xi(X_n)\big)\Big) P_{X_n, x_0} \bar w,
\end{split}
\end{equation*}
or else $P_{X_n, x_0}\xi(X_n)\in S_1^-\setminus S_1^+$ and $P_{X_n,
  x_0}P_{Y_n, X_n}\xi(Y_n)\in S_0^-$, so that:
\begin{equation*}
\begin{split}
P_{Y_n, X_n}Q^{X_n}(Y_n)&P_{X_n, Y_n}\bar w - Q(X_n)\bar w \\ & = P_{X_0,
  X_n}\circ Q_0^T\Big( \bar Q\big(Q_0\cdot P_{X_n, x_0}P_{Y_n, X_n}\xi(Y_n)\big) - \bar
Q\big(Q_0\cdot P_{X_n, x_0}\xi(X_n)\big)\Big) P_{X_n, x_0} \bar w.
\end{split}
\end{equation*}
In both cases, we obtain by (\ref{dodici2}):
\begin{equation*}
\begin{split}
\sup_{\bar w\in B(0,1)\subset T_{X_n}\M}\big|&P_{Y_n,
  X_n}Q^{X_n}(Y_n)P_{X_n, Y_n}\bar w - Q(X_n)\bar w\big|  \leq
\mbox{Lip}_{\bar Q} \cdot \big|P_{Y_n, X_n}\xi(Y_n) -
\xi(X_n)\big|_{X_n}\\ &\qquad\qquad\qquad \leq \frac{2\cdot
  \mbox{Lip}_{\nabla(\log\rho)}\cdot\mbox{Lip}_{\bar Q}}{ \min\big\{|\nabla (\log\rho)(X_n)|_{X_n}, |\nabla
(\log\rho)(Y_n)|_{Y_n}\big\}}r,
\end{split}
\end{equation*}
which implies:
$$B\leq \mathcal{O}(\epsilon^2 r + \epsilon^3) \qquad
\mbox{when } \min\big\{|\nabla (\log\rho)(X_n)|_{X_n}, |\nabla
(\log\rho)(Y_n)|_{Y_n}\big\}> \frac{\sigma_\eta}{4}.$$

In conclusion, it follows that:
\begin{equation}\label{tredici}
B\leq \frac{1}{2}\epsilon^2\sigma_\eta + \mathcal{O}(\epsilon^2 r +\epsilon^3).
\end{equation}

\smallskip

{\bf 6.} The estimates (\ref{undici}), (\ref{dodici}), (\ref{tredici}) imply that on the event $\{\tau>n\}$:
\begin{equation} \label{tredici2}
\E\big( d_\M (X_{n+1}, Y_{n+1}) -  d_\M (X_n, Y_n)\mid\F_n\big) \leq
\frac{1}{2}\epsilon^2\sigma_\eta + \mathcal{O}(\epsilon^2 r +\epsilon^3).
\end{equation}
We are now ready to estimate the quantities in (\ref{otto2}). Firstly,
in virtue of (\ref{dieci}) and (\ref{tredici2}), we obtain the
following estimate, valid on $\{\tau>n\}$:
\begin{equation*}
\begin{split}
\E\Big( d_\M (X_{n+1}, &Y_{n+1}) - \frac{1}{2r}  d_\M (X_{n+1},
Y_{n+1})^2\mid\F_n\Big) \\ & \leq  d_\M (X_{n}, Y_{n}) - \frac{1}{2r}  d_\M (X_{n}, Y_{n})^2 -
\frac{\epsilon^2\sigma_\eta}{2r},
\end{split}
\end{equation*}
because $\frac{1}{2}\epsilon^2\sigma_\eta +\mathcal{O}(\epsilon^2 r) -
\frac{3\sigma_\eta\epsilon^2}{2r}\leq -\frac{\epsilon^2\sigma_\eta}{2r}$
if only $r$ is sufficiently small. Consequently, the following
sequence of random variables:
$$\Big\{\bar M_n\doteq  d_\M (X_{\tau\wedge n}, Y_{\tau\wedge n}) -
\frac{1}{r} d_\M (X_{\tau\wedge n}, Y_{\tau\wedge n})^2 +
\frac{\sigma_\eta\epsilon^2}{2r}(\tau\wedge n)\Big\}_{n=0}^\infty$$
is a supermartingale with respect to the filtration
$\{\F_n\}_{n=0}^\infty$. By an application of Doob's optional stopping theorem, we infer that:
\begin{equation*}
\begin{split}
\E\Big[ d_\M (X_\tau, &Y_\tau) - \frac{1}{2r} d_\M (X_\tau, Y_\tau)^2\Big]
+ \frac{\sigma_\eta\epsilon^2}{2r}\E[\tau] \\ & \quad = \E[\bar M_\tau] \leq
\E[\bar M_0] =  d_\M (x_0, y_0) - \frac{1}{2r} d_\M (x_0, y_0)^2\leq 
 d_\M (x_0, y_0).
\end{split}
\end{equation*} 
Since for $\epsilon$ sufficiently small there holds $ d_\M (X_\tau,
Y_\tau) \leq\frac{4r}{3}$, it further follows that: $ d_\M (X_\tau,
Y_\tau) - \frac{1}{2r} d_\M (X_\tau, Y_\tau)^2\geq
\frac{1}{3} d_\M (X_\tau, Y_\tau)$, hence the above displayed inequality yields:
\begin{equation}\label{quattordici}
\frac{1}{3}\E\big[ d_\M (X_\tau, Y_\tau)\big] +
\frac{\sigma_\eta\epsilon^2}{2r}\E[\tau]\leq  d_\M (x_0, y_0).
\end{equation}

In particular, we get:
\begin{equation}\label{quindici}
\E[\tau]\leq \frac{2r}{\sigma_\eta\epsilon^2}  d_\M (x_0, y_0)
\end{equation}
and also, via Markov's inequality:
\begin{equation*}
\nu\big( d_\M (X_\tau, Y_\tau)\geq r\big)\leq \frac{\E[ d_\M (X_\tau,
  Y_\tau)]}{r}\leq \frac{3}{r}  d_\M (x_0, y_0).
\end{equation*}

Finally, Lemma \ref{exitfromr} and (\ref{quindici}) give:
\begin{equation*}
\nu\big( d_\M (X_\tau, x_0)\geq r\big)\leq \frac{\E[ d_\M (X_\tau,
  x_0)^2]}{r^2}\leq \frac{3m}{r}  d_\M (x_0, y_0).
\end{equation*}
Recalling (\ref{sette}) and (\ref{otto}), we conclude:
\begin{equation*}
\begin{split}
|f(x_0) - f(y_0)|\leq ~& \sup \Big\{|f(x) - f(y)|;~ x\in B_\M(x_0, r+\epsilon) ~
 d_\M (x,y)\leq 3\epsilon\Big\} \\ & + \frac{6(m+1)}{r}\|f\|_{L^\infty(B_\M(x_0, 2r+3\epsilon))}
 d_\M (x_0, y_0) \\ & + \frac{4r}{\sigma_\eta} \cdot \frac{\|\bar\A_\epsilon f -
  f\|_{L^\infty(B_\M(x_0, 2r))}}{\epsilon^2} d_\M (x_0, y_0),
\end{split}
\end{equation*} 
as claimed.
\end{proof}

\section{The second approximate Lipschitz estimate}\label{sec_second}

In this section, we work under hypothesis (\ref{H2}) and develop a complementary estimate to
Theorem \ref{Thestimate1}.
Additionally, we assume there exists an open set $S\subset B(0,1)$
and $c>0$ with:
\begin{equation}\label{asuS}
\eta(|w|)\geq c \qquad \mbox{for all } \;w\in S.
\end{equation}
Clearly, $S$ can be taken to be rotationally invariant, for example:
$S=B(0,\delta_1)$ or $S=B(0,\delta_2)\setminus \bar B(0, \delta_1)$ where $0<\delta_1<\delta_2<1$. The condition \eqref{asuS} is a very natural assumption for applications and thus sufficient for our main goal in this paper. However, we remark that it can be relaxed as we discuss in Remark \ref{rem:91}.

\begin{theorem}\label{th2}
There exists a constant $\theta\in (0,1)$ depending only on $\eta$,
such that the following holds if only $\epsilon\ll 1$. Let $f:\M\to\R$ be a bounded, Borel
function. Then for every $x_0, y_0\in\M$ satisfying $ d_\M (x_0,
y_0)\leq 3\epsilon$ there holds:
\begin{equation*}
\begin{split}
|f(x_0) - f(y_0)|\leq ~& \theta \cdot \sup \Big\{|f(x) - f(y)|;~ x\in B_\M(x_0, \bar N\epsilon), ~
 d_\M (x,y)\leq 5\epsilon\Big\} \\ & + \bar N\epsilon \|\nabla
(\log\rho)\|_{L^\infty(\M)}\cdot \|f\|_{L^\infty(B_\M(x_0, \bar N\epsilon))} + C\epsilon^2\|f\|_{L^\infty(B_\M(x_0, \bar N\epsilon))} 
\\ & + \bar N\|\bar\A_\epsilon f -  f\|_{L^\infty(B_\M(x_0, \bar N\epsilon))}.
\end{split}
\end{equation*} 
The constant $C$ above depends on $(\M, g)$ and $\eta$, whereas $\bar
N$ depends only on $\eta$.
\end{theorem}

The proof pursues the argument in the local (normal) coordinates.
Given $x_0, y_0\in\M$ such that $ d_\M (x_0, y_0)\leq 3\epsilon\ll 1$,
consider the smooth diffeomorphism $\phi_{\epsilon, x_0, y_0}$ of
$B(0,1)\subset T_{x_0}\M$ onto its image in $T_{x_0}\M$:
\begin{equation}\label{fe}
\phi_{\epsilon, x_0, y_0} (w) =
\frac{1}{\epsilon}Exp_{x_0}^{-1}Exp_{y_0}\big(\epsilon P_{x_0, y_0}
w\big) \quad \mbox{ for all } w\in B(0,1).
\end{equation}

We start by the following observation:

\begin{proposition}\label{S_prop}
Let $S\subset B(0,1)\subset \R^m$ be a rotationally symmetric, open set of the
form: $S = B(0,\delta_2)\setminus \bar B(0,\delta_1)$ for some $0<\delta_1
< \delta_2$, and let $c>0$ satisfy (\ref{asuS}). Then there
exist a constant $\alpha>0$ and an integer $N>1$, such that:
\begin{itemize}
\item[(i)] $\displaystyle \alpha\doteq  \int_{\{\langle w, e_1\rangle \in
    (\frac{1}{4N}, \frac{1}{2N})\}}\eta(|w|)\dw \in (0,1)$,\vspace{1mm}
\item[(ii)] $|S\cap \phi_{\epsilon, x_0, y_0}(S)|> \frac{1}{2}|S|$,
  whenever $ d_\M (x_0, y_0)\leq\frac{\epsilon}{N}$ and $\epsilon\ll 1$.
\end{itemize}
\end{proposition}
\begin{proof}
Condition (i) is evident in view of (\ref{asuS}), for any
$N>\frac{1}{\delta_2}$. To prove (ii), we write
$y_0=Exp_{x_0}^{-1}(\epsilon z)$ where $z\in B(0,1)\subset T_{x_0}\M$
and note that $\phi_{\epsilon, x_0, y_0}(w)=\gamma(1)$, where one solves:
\begin{equation*}
\left\{\begin{array}{ll}
\dot w^i(t) = -\epsilon\Gamma_{jk}^i(\epsilon t z) w^j(t) z^k, & t\in [0,1]\\
w(0) = w. &
\end{array}\right.
\quad \left\{\begin{array}{ll}
\ddot \gamma^i(s) = -\epsilon\Gamma_{jk}^i(\epsilon \gamma(s))
\dot\gamma^j(s)\dot\gamma^k(s), & s\in [0,1]\\
\dot\gamma(0) = w(1), \quad \gamma(0) = z. &
\end{array}\right.
\end{equation*}
Since the map:
$$\big(z\in B(0,1), w\in B(0,1), |\epsilon|\ll 1, s\in [0,1]\big)\mapsto \gamma(s)$$
is smooth, with all its derivatives bounded, and since $\Gamma_{jk}^i(0) = 0$, 
standard arguments in the theory of systems of ODEs imply that:
\begin{equation}\label{pomoc}
\phi_{\epsilon, x_0, y_0}(w)=z+w + \mathcal{O}(\epsilon^2),\qquad
\nabla \phi_{\epsilon, x_0, y_0}(w)= Id_m + \mathcal{O}(\epsilon^2).
\end{equation}
Here, the Landau symbol $\mathcal{O}$ above refers to any quantity that is
bounded together with its derivatives, with a bound depending only on $(\M, g)$.

For a given $\delta>0$, let $\psi\in\mathcal{C}_0^\infty(S)$ be such
that $\int |\mathbbm{1}_S - \psi|<\delta$. We then estimate the
symmetric difference of the sets $S$ and $\phi_{\epsilon, x_0,
  y_0}(S)$ by changing variable in view of (\ref{pomoc}):
\begin{equation}\label{pomoc2}
\begin{split}
|S \triangle &\phi_{\epsilon, x_0, y_0}(S)|  = \int |\mathbbm{1}_S - \mathbbm{1}_{\phi_{\epsilon, x_0, y_0}(S)}|\\
& \leq \int |\mathbbm{1}_S - \psi| + \int |\psi -
\psi\circ\phi_{\epsilon, x_0, y_0}^{-1}| + \int |\psi\circ\phi_\epsilon^{-1} - \mathbbm{1}_{\phi_{\epsilon, x_0, y_0}(S)}|
\\ & = \int |\mathbbm{1}_S - \psi| + \int |\psi \circ\phi_{\epsilon,
  x_0, y_0} - \psi|\cdot|\det\nabla \phi_{\epsilon, x_0, y_0}| 
+ \int |\mathbbm{1}_S - \psi|\cdot|\det\nabla \phi_{\epsilon, x_0, y_0}| \\ & 
\leq \delta + 2  \int |\psi \circ\phi_{\epsilon,
  x_0, y_0} - \psi| + 2\delta.
\end{split}
\end{equation}
By choosing an appropriate approximating function $\psi$, the integral
in the right hand side above may be bounded by:
$\frac{C}{\delta} \big(|z| + \mathcal{O}(\epsilon^2)\big)^2$. Taking
$3\delta< \frac{|S|}{4}$ and $|z|<\frac{1}{N}$ to the
effect that also $\frac{C}{\delta} \big(|z| +
\mathcal{O}(\epsilon^2)\big)^2< \frac{|S|}{4}$ if only $N$ is large enough
and $\epsilon\ll 1$, we conclude that $|S \; \triangle \; \phi_{\epsilon,
  x_0, y_0}(S)|<\frac{|S|}{2}$. Consequently, there follows (ii) because:
$$|S\cap \phi_{\epsilon, x_0, y_0}(S)| = |S| - |S\setminus \phi_{\epsilon, x_0,
  y_0}(S)| \geq |S| -  |S \; \triangle \; \phi_{\epsilon, x_0, y_0}(S)| \geq \frac{1}{2}|S|,$$
and because $ d_\M (x_0, y_0)=\epsilon|z|$.
\end{proof}

\medskip

\noindent {\bf A proof of Theorem \ref{th2}.}

{\bf 1.} Given $x_0, y_0\in \M$ such that $0< d_\M (x_0, y_0)\leq 3\epsilon$, we write:
\begin{equation}\label{bd1}
\begin{split}
\big|f(x_0)-f(y_0)\big| & \leq |\bar\A_\epsilon f(x_0) - f(x_0)| +
|\bar\A_\epsilon f(y_0) - f(y_0)| +  |\bar\A_\epsilon f(x_0) - \bar\A_\epsilon f(y_0)| \\ & \leq 
2  \|\bar\A_\epsilon f  - f\|_{L^\infty(B_\M(x_0, 4\epsilon))} + 2\epsilon \|\nabla
(\log\rho)\|_{L^\infty(\M)}\cdot\|f\|_{L^\infty(B_\M(x_0, 5\epsilon))}
\\ & \quad + \Big|\int_{B(0,1)} \eta(|w|)
\Big(f(Exp_{x_0}(\epsilon w)) - f(Exp_{y_0}(\epsilon P_{x_0, y_0} w))\Big)\dw\Big|.
\end{split}
\end{equation}
As usual, the integration on $B(0,1)\subset T_{x_0}\M  \simeq \R^m$ is written
in the normal coordinates centered at $x_0$.  We denote $\epsilon
z = Exp_{x_0}^{-1}(y_0)$ and express the integral in (\ref{bd1}) as:
\begin{equation}\label{bd1.5}
\int_{B(0,1)} \eta(|w|) \Big(f(Exp_{x_0}(\epsilon w)) - f(Exp_{y_0}(\epsilon P_{x_0, y_0}
\mbox{Refl} (w)))\Big)\dw,
\end{equation}
where $\mbox{Refl}:T_{x_0}\M\to T_{x_0}\M$ is the reflection
across the hyperplane perpendicular to $z$:
$$\mbox{Refl}(w) = w - \frac{2\langle w,
  Exp_{x_0}^{-1}(y_0)\rangle}{ d_\M (x_0, y_0)}
Exp_{x_0}^{-1}(y_0)\qquad \mbox{for } w\in T_{x_0}\M.$$
Our goal is now to estimate the quantity in (\ref{bd1.5}). Below, we consider
two cases towards closing the bound in (\ref{bd1}). Recall that $N, s,
\alpha$ are as in Proposition \ref{S_prop}. 

\smallskip

{\bf 2.} We first analyze the case $ d_\M (x_0, y_0)\leq\frac{\epsilon}{N}$.
We split the integral in (\ref{bd1.5}) on the following two
integration sets, recalling that $\phi_\epsilon=\phi_{\epsilon, x_0, y_0}$ is defined in (\ref{fe}):
\begin{equation}\label{bd2}
\begin{split}
& \int_{S\cap \phi_{\epsilon, x_0, y_0}(S)} \eta(|w|)
\Big(f(Exp_{x_0}(\epsilon w)) - f(Exp_{x_0}(\epsilon \phi_{\epsilon,
  x_0, y_0} \mbox{Refl} (w)))\Big)\dw \\
& + \int_{B(0,1)\setminus (S\cap \phi_{\epsilon, x_0, y_0}(S) )}\eta(|w|)
\Big(f(Exp_{x_0}(\epsilon w)) - f(Exp_{y_0}(\epsilon P_{x_0, y_0} \mbox{Refl} (w)))\Big)\dw.
\end{split}
\end{equation}
To deal with the first integral, we change the variable through the
piecewise $\mathcal{C}^1$ diffeomorphism $w\mapsto \phi_\epsilon
\mbox{Refl}\;w$, and recall (\ref{pomoc}) to compute:
\begin{equation*}
\begin{split}
& \int_{S\cap \phi_{\epsilon}(S)} f(Exp_{x_0}(\epsilon w)) - f(Exp_{x_0}(\epsilon \phi_{\epsilon} \mbox{Refl} (w)))\dw \\
& = \int_{S\cap \phi_{\epsilon}(S)} f(Exp_{x_0}(\epsilon w)) \dw -
\int_{\phi_{\epsilon}(S)\cap  \phi_{\epsilon}\mbox{\scriptsize
    Refl}\,\phi_\epsilon(S)} f(Exp_{x_0}(\epsilon w)) \cdot |\det\nabla
\phi_\epsilon|^{-1}\dw \\  
& \leq \int_{S\cap \phi_{\epsilon}(S)} f(Exp_{x_0}(\epsilon w)) \dw -
\int_{\phi_{\epsilon}(S)\cap \phi_{\epsilon}\mbox{\scriptsize
    Refl}\,\phi_\epsilon(S)} f(Exp_{x_0}(\epsilon w)) \dw 
+ \mathcal{O}(\epsilon^2) \|f\|_{L^\infty(B_\M(x_0, 2\epsilon))}\\  
&\leq \|f\|_{L^\infty(B_\M(x_0, 2\epsilon))}\Big(\mathcal{O}(\epsilon^2) +
\big|\big(S\cap\phi_\epsilon (S)\big) \triangle
\big(\phi_\epsilon (S)\cap \phi_{\epsilon}
\mbox{Refl}\,\phi_\epsilon(S)\big)\big| \Big) \\  
&\leq \|f\|_{L^\infty(B_\M(x_0, 2\epsilon))}\Big(\mathcal{O}(\epsilon^2)+ \big|S \triangle \phi_{\epsilon}
\mbox{Refl}\,\phi_\epsilon(S)\big| \Big)\leq \mathcal{O}(\epsilon^2)
\cdot \|f\|_{L^\infty(B_\M(x_0, 2\epsilon))}. 
\end{split}
\end{equation*}
The final bound above follows by using (\ref{pomoc}) in:
\begin{equation*}
\begin{split}
\phi_{\epsilon} \mbox{Refl} \,\phi_\epsilon(w) & = \frac{1}{\epsilon}
Exp_{x_0}^{-1}(y_0) + \mbox{Refl}\, \phi_\epsilon(w) +
\mathcal{O}(\epsilon^2) \\ & = \frac{1}{\epsilon}
Exp_{x_0}^{-1}(y_0) + \mbox{Refl}\Big(\frac{1}{\epsilon}
Exp_{x_0}^{-1}(y_0) + w\Big) + \mathcal{O}(\epsilon^2) = \mbox{Refl}\,
w + \mathcal{O}(\epsilon^2)
\end{split}
\end{equation*}
and applying the same argument as in (\ref{pomoc2}) with $\delta = \epsilon^2$.

\smallskip

In conclusion, we may use the splitting in (\ref{bd2}) to estimate:
\begin{equation*}
\begin{split}
\Big|\int_{B(0,1)} &\eta(|w|) \Big(f(Exp_{x_0}(\epsilon w)) -
f(Exp_{y_0}(\epsilon P_{x_0, y_0}\mbox{Refl} (w)))\Big)\dw \Big| \\ & \leq 
\Big|\int_{S\cap \phi_{\epsilon, x_0, y_0}(S)} (\eta(|w|)-c)\cdot
\Big(f(Exp_{x_0}(\epsilon w)) - f(Exp_{x_0}(\epsilon \phi_{\epsilon,
  x_0, y_0} \mbox{Refl} (w)))\Big)\dw \Big| \\ & \quad + 
\Big| \int_{B(0,1)\setminus (S\cap \phi_{\epsilon, x_0, y_0}(S) )}\eta(|w|)
\Big(f(Exp_{x_0}(\epsilon w)) - f(Exp_{y_0}(\epsilon P_{x_0, y_0}
\mbox{Refl} (w)))\Big)\dw\Big| \\ & \quad + \mathcal{O}(\epsilon^2) \cdot \|f\|_{L^\infty(B_\M(x_0, 2\epsilon))},
\end{split}
\end{equation*}
and further:
\begin{equation*}
\begin{split}
\Big|\int_{B(0,1)} &\eta(|w|) \Big(f(Exp_{x_0}(\epsilon w)) -
f(Exp_{y_0}(\epsilon P_{x_0, y_0}\mbox{Refl} (w)))\Big)\dw \Big| 
\\ & \leq \mathcal{O}(\epsilon^2) \|f\|_{L^\infty(B_\M(x_0, 2\epsilon))} + \Big(\int_{S\cap \phi_{\epsilon}(S)}
(\eta(|w|)-c)\dw +  \int_{B(0,1)\setminus (S\cap \phi_{\epsilon}(S)
  )}\eta(|w|)\dw\Big) \\ & \qquad \qquad\qquad \qquad \quad \cdot
\sup \Big\{|f(x) - f(y)|;~ x\in B_\M(x_0, \epsilon), ~  d_\M (x,y)\leq 3\epsilon\Big\}
\\ & \leq \mathcal{O}(\epsilon^2)\|f\|_{L^\infty(B_\M(x_0, 2\epsilon))} \\ & \quad + \big(1- c|S\cap \phi_{\epsilon}(S)|\big) \cdot
\sup \Big\{|f(x) - f(y)|;~ x\in B_\M(x_0, \epsilon), ~  d_\M (x,y)\leq 3\epsilon\Big\}
\\ & \leq \mathcal{O}(\epsilon^2) \|f\|_{L^\infty(B_\M(x_0,
  2\epsilon))} \\ & \quad +  \Big(1- \frac{c|S|}{2}\Big) \cdot
\sup \Big\{|f(x) - f(y)|;~ x\in B_\M(x_0, \epsilon), ~  d_\M (x,y)\leq 3\epsilon\Big\}
\end{split}
\end{equation*}
by Proposition \ref{S_prop} (ii) since $ d_\M (x_0,
y_0)\leq\frac{\epsilon}{N}$. Inserting the above estimate in (\ref{bd1}) yields:
\begin{equation}\label{bd5}
\begin{split}
\big|f(x_0)-f(y_0)\big| & \leq 
2  \|\bar\A_\epsilon f  - f\|_{L^\infty(B_\M(x_0, 4\epsilon))} +
2\epsilon \|\nabla (\log\rho)\|_{L^\infty(\M)}\cdot\|f\|_{L^\infty(B_\M(x_0, 5\epsilon))} 
\\ & \quad + \mathcal{O}(\epsilon^2) \|f\|_{L^\infty(B_\M(x_0, 2\epsilon))}  \\ & \quad + \Big(1- \frac{c|S|}{2}\Big) \cdot
\sup \Big\{|f(x) - f(y)|;~ x\in B_\M(x_0, \epsilon), ~  d_\M (x,y)\leq 3\epsilon\Big\}.
\end{split}
\end{equation}

\smallskip

{\bf 3.} In the remaining case $ d_\M (x_0, y_0)>\frac{\epsilon}{N}$, we
split the integral (\ref{bd1.5}) on the two sets:
\begin{equation}\label{bd3}
\begin{split}
& \Big|\int_{\{\langle w, \frac{z}{|z|}\rangle\in (\frac{1}{4N}, \frac{1}{2N})\}} \eta(|w|)
\Big(f(Exp_{x_0}(\epsilon w)) -  f(Exp_{y_0}(\epsilon P_{x_0, y_0} \mbox{Refl} (w)))\Big)\dw \Big| \\
& + \Big|\int_{B(0,1)\setminus \{\langle w, \frac{z}{|z|}\rangle\in (\frac{1}{4N}, \frac{1}{2N})\}}\eta(|w|)
\Big(f(Exp_{x_0}(\epsilon w)) - f(Exp_{y_0}(\epsilon P_{x_0, y_0}
\mbox{Refl} (w)))\Big)\dw\Big|\\ & \qquad\qquad  \leq \alpha \cdot 
\sup \Big\{|f(x) - f(y)|;~ x\in B_\M(x_0, \epsilon), ~  d_\M (x,y)\leq  d_\M (x_0,y_0) - \frac{\epsilon}{4N}\Big\} 
\\ & \qquad\qquad\quad + (1- \alpha) \cdot 
\sup \Big\{|f(x) - f(y)|;~ x\in B_\M(x_0, \epsilon), ~  d_\M (x,y)\leq 5\epsilon\Big\}.
\end{split}
\end{equation}
The last estimate above is a consequence of (\ref{tre}) as
follows. For any $w\in B(0,1)$ satisfying: $\langle w, \frac{z}{|z|}\rangle\in
(\frac{1}{4N}, \frac{1}{2N})$, there holds: 
$|z-\frac{2\langle w, z\rangle}{|z|^2}z| = \big| |z| - \frac{2\langle
  w, z\rangle}{|z|}\big| = \frac{ d_\M (x_0, y_0)}{\epsilon} -
2\langle w, \frac{z}{|z|}\rangle > 0$, because
$ d_\M (x_0, y_0) = |\epsilon z|$. Thus, we may
apply square root to the expansion in (\ref{tre}) and control the
error terms, for all $\epsilon$ sufficiently small:
\begin{equation*}
\begin{split}
 d_\M \big(Exp_{x_0}&(\epsilon w), Exp_{y_0}(\epsilon P_{x_0, y_0}\mbox{Refl}\, w)\big)
= \big|\epsilon z - \frac{2\epsilon \langle w, \epsilon z\rangle}{|\epsilon
  z|^2}\epsilon z\big| + \mathcal{O}(\epsilon^2) \\ & \leq \epsilon \Big| |z| - 
\frac{1}{2N} + \mathcal{O}(\epsilon^2) \leq \epsilon|z| - \frac{\epsilon}{4N}.
\end{split}
\end{equation*}
This proves (\ref{bd3}). Combining with (\ref{bd1}), we obtain:
\begin{equation}\label{eqn:aux4.4}
\begin{split}
\big|f(x_0)-&f(y_0)\big|  \leq  \alpha \cdot 
\max\big\{\kappa_1, \tilde\kappa_1\big\}
\\ & + (1- \alpha) \cdot 
\sup \Big\{|f(x) - f(y)|;~ x\in B_\M(x_0, \epsilon), ~  d_\M (x,y)\leq 5\epsilon\Big\}
\\ & + 2  \|\bar\A_\epsilon f  - f\|_{L^\infty(B_\M(x_0, 4\epsilon))}
+ 2\epsilon \|\nabla (\log\rho)\|_{L^\infty(\M)} \cdot\|f\|_{L^\infty(B_\M(x_0, 5\epsilon))},
\end{split}
\end{equation}
where for all  $i=1\ldots 4(3N-1)$ we have denoted:
\begin{equation*}
\begin{split}
& \kappa_i\doteq\sup \Big\{|f(x) - f(y)|;~ x\in B_\M(x_0, i\epsilon), ~
\frac{\epsilon}{N} < d_\M(x,y)\leq  d_\M (x_0,y_0) - \frac{i\epsilon}{4N}\Big\}, \\ 
& \tilde\kappa_i\doteq\sup \Big\{|f(x) - f(y)|;~ x\in B_\M(x_0, i\epsilon), ~
d_\M(x,y)\leq  \frac{\epsilon}{N}\Big\}. 
\end{split}
\end{equation*}
Observe that since $d_\M(x_0, y_0)\leq 3\epsilon$, there exists $i$ as
above for which $\tilde\kappa_i\geq \kappa_i$ (indeed, the set over
which the supremum is taken in the definition of $\kappa_{4(3N-1)}$ is
empty). Iterating $i$ times the expression (\ref{eqn:aux4.4}), we
arrive at:
\begin{equation}\label{bd4}
\begin{split}
\big|f(x_0)-f(y_0)\big|  \leq & \, \alpha^i \cdot 
\sup \Big\{|f(x) - f(y)|;~ x\in B_\M(x_0, i\epsilon), ~  d_\M (x,y)\leq \frac{\epsilon}{N}\Big\} 
\\ & + (1- \alpha^i) \cdot 
\sup \Big\{|f(x) - f(y)|;~ x\in B_\M(x_0, i\epsilon),  ~  d_\M (x,y)\leq 5\epsilon\Big\}
\\ & + \frac{2(1-\alpha^i)}{1-\alpha}  \|\bar\A_\epsilon f  - f\|_{L^\infty(B_\M(x_0, 4i\epsilon))} 
\\ & + \frac{2\epsilon (1-\alpha^i)}{1-\alpha} \|\nabla
(\log\rho)\|_{L^\infty (\M)}\cdot\|f\|_{L^\infty(B_\M(x_0, 5i\epsilon))}.
\end{split}
\end{equation}

\smallskip

{\bf 4.} Finally, we combine (\ref{bd5}) and (\ref{bd4}) to conclude the
estimate in the Theorem, in which the term: $\sup \Big\{|f(x) -
f(y)|;~ x\in B_\M(x_0, 12 N\epsilon), ~  d_\M (x,y)\leq 5\epsilon\Big\}$
is multiplied by the factor: 
$$\alpha^i \big(1-\frac{c|S|}{2}\big) + (1-\alpha^i) =
1- \frac{\alpha^i c |S|}{2} \leq 1- \frac{\alpha^{4(3N-1)} c |S|}{2} \doteq\theta.$$
Namely, we get:
\begin{equation*}
\begin{split}
|f(x_0) - f(y_0)|\leq ~& \theta \cdot \sup \Big\{|f(x) - f(y)|;~ x\in B_\M(x_0, 12N\epsilon), ~
 d_\M (x,y)\leq 5\epsilon\Big\} \\ & + 24 N\|\bar\A_\epsilon f -
 f\|_{L^\infty(B_\M(x_0, 48 N\epsilon))} \\ & + 24 N\epsilon \|\nabla
(\log\rho)\|_{L^\infty(\M)}\cdot \|f\|_{L^\infty(B_\M(x_0, 60
  N\epsilon))} + \mathcal{O}(\epsilon^2)\|f\|_{L^\infty(B_\M(x_0, 12 N\epsilon))}. 
\end{split}
\end{equation*}
The proof is complete.
\endproof

\section{Closing the bounds and proofs of Theorems \ref{th_main1} and
  \ref{th_main2}}\label{sec_closebound} 

In this section, we work under hypothesis (\ref{H2}) and the
additional property \ref{asuS}.
We will conclude the proofs of the local and global
approximate Lipschitz estimates in the continuum setting. The proof of
the global estimate is less involved, as it can directly utilize the contraction
property in Theorem \ref{th2}. For the local estimate, we needan extra
iteration argument.

\subsection{The global bound and a proof of Theorem \ref{th_main1}}

We first deduce a version of the approximate Lipschitz estimate in
Theorem \ref{th_main1}, involving the auxiliary averaging operator
$\bar\A_\epsilon$ rather than $\A_\epsilon$. This allows for a more
precise estimate, where the factors $C$ below depend only on the manifold $(\M, g)$ and
the radial weight function $\eta$, while the dependence on the drift field
$\rho$, occurring through $\|\nabla (\log \rho)\|_{L^\infty}$, is present
in just one specific error term. This observation is consistent with
the analysis of the pure Laplace-Beltrami operator case $\rho\equiv const$.

\begin{theorem}\label{Thlipschitz3}
Let $\epsilon\ll 1$. Then, for every bounded, Borel function $f:\M\to\R$ and
every $x_0, y_0\in\M$ we have the following estimate, where constants
$C$ depend only on $(\M, g)$ and $\eta$:
\begin{equation*}
\begin{split}
|f(x_0) - f(y_0)|\leq & \; C\Big(\|f\|_{L^\infty(\M)} +
\frac{\|\bar\A_\epsilon f- f\|_{L^\infty(\M)}}{\epsilon^2}\Big)\cdot  d_\M (x_0, y_0) \\ & + 
C\epsilon \Big((\|\nabla (\log \rho)\|_{L^\infty(\M)} + 1)\cdot
\|f\|_{L^\infty(\M)} + \frac{\|\bar\A_\epsilon f- f\|_{L^\infty(\M)}}{\epsilon^2}\Big).
\end{split}
\end{equation*}
\end{theorem}
\begin{proof}
By Theorems \ref{Thestimate1} and \ref{th2} we obtain:
\begin{equation}\label{bd6}
\begin{split}
|f(x_0) - f(&y_0)|\leq \theta\cdot \sup \Big\{|f(x) - f(y)|;~ x,y\in\M, ~  d_\M (x,y)\leq 5\epsilon\Big\}
\\ & + C\Big(\|f\|_{L^\infty(\M)} +
\frac{\|\bar\A_\epsilon f- f\|_{L^\infty}}{\epsilon^2}\Big)\cdot  d_\M (x_0, y_0) \\ & + 
C\Big(\epsilon\|\nabla (\log \rho)\|_{L^\infty(\M)}\cdot
\|f\|_{L^\infty(\M)} + \epsilon^2 \|f\|_{L^\infty(\M)}  + \|\bar\A_\epsilon f- f\|_{L^\infty(\M)}\Big),
\end{split}
\end{equation}
valid whenever $ d_\M (x_0, y_0)<r$ and for $\epsilon\ll 1$ small enough.
Taking the supremum over the set $\{x,y\in\M, ~  d_\M (x,y)\leq
5\epsilon\}$ in the left hand side above, and recalling that $\theta<1$, we arrive at:
\begin{equation*}
\begin{split}
\sup \Big\{|f(x) & - f(y)|;~ x,y\in\M, ~  d_\M (x,y)\leq 5\epsilon\Big\}
\\  \leq \, & C\epsilon \big(\|\nabla (\log \rho)\|_{L^\infty(\M)}+1\big)
\|f\|_{L^\infty(\M)} + C\frac{\|\bar\A_\epsilon f- f\|_{L^\infty(\M)}}{\epsilon}.
\end{split}
\end{equation*}
Inserting the above bound into (\ref{bd6}) achieves the proof when
$ d_\M (x_0, y_0)<r$. The general case and the global estimate follow by compactness of $\M$.
\end{proof}

\medskip

\noindent {\bf A proof of Theorem \ref{th_main1}.} The result follows
from Theorem \ref{Thlipschitz3}, upon replacing $\bar\A_\epsilon
f$ by $\A_\epsilon f$, invoking Lemma \ref{AepAepbar}, and recalling
that $\A_\epsilon f(x) - f(x) = -(\rho(x) +
\mathcal{O}(\epsilon^2))\epsilon^2 \Delta_\epsilon f(x)$.
\endproof

\subsection{The local bound and a proof of Theorem \ref{th_main2}}
We now present the interior counterpart of the previously established estimates.

\begin{theorem}\label{Thlipschitz3_2}
Let $\epsilon\ll r\ll 1$. Then, for every bounded, Borel function $f:\M\to\R$ and
every $x_0, y_0\in\M$ with $d_\M(x_0, y_0)<r$, there holds:
\begin{equation*}
\begin{split}
|f(x_0) - f(y_0)|\leq & \; C\Big(\frac{d_\M(x_0, y_0)}{r} +
\frac{\epsilon |\log\epsilon|}{r} \Big)\cdot \|f\|_{L^\infty(B_\M(x_0,
  3r))} \\ & + C\Big(r d_\M(x_0, y_0) + \frac{\epsilon r}{|\log\epsilon|}\Big)\cdot 
\frac{\|\bar\A_\epsilon f- f\|_{L^\infty(B_\M(x_0, 3r))}}{\epsilon^2}.
\end{split}
\end{equation*}
The constant $C$ depends on $(\M, g)$, $\eta$ and $\|\nabla (\log\rho)\|_{L^\infty(\M)}$.
\end{theorem}
\begin{proof}
Let $d_\M(x_0, y_0)<r$. Combining Theorems \ref{Thestimate1} and
\ref{th2}, we obtain:
\begin{equation}\label{bd7}
\begin{split}
|f(x_0) - f(y_0)|\leq & \; \theta\cdot \sup\Big\{|f(x) - f(y)|;~ x\in
B_\M(x_0, r+\bar N\epsilon), ~ d_\M(x,y)\leq 5\epsilon\Big\}\\ &
+ C\Big(\epsilon + \frac{d_\M(x_0, y_0)}{r} \Big)\cdot \|f\|_{L^\infty(B_\M(x_0,
  2r+ \bar N\epsilon))} \\ & + C\Big(\epsilon^2 + r d_\M(x_0, y_0) \Big)\cdot 
\frac{\|\bar\A_\epsilon f- f\|_{L^\infty(B_\M(x_0, 2r+\bar N\epsilon))}}{\epsilon^2},
\end{split}
\end{equation}
where $C$ depends on $(\M, g)$ and $\eta$, while both $\theta\in
(0,1)$ and $\bar N$ depend only on $\eta$. 

Next, we bound the first term in the right hand side above. Let
$x,y\in \M$ satisfy $d_\M(x,y)\leq 5\epsilon$. Using (\ref{bd7}) with $r$ replaced by a new radius
$R=\frac{\log\theta}{\log\epsilon}\cdot\frac{r}{2}$, we deduce:
\begin{equation*}
\begin{split}
|f&(x) - f(y)|\leq \; \theta\cdot \sup\Big\{|f(\bar x) - f(\bar y)|;~
\bar x\in B_\M(x, R+\bar N\epsilon), ~ d_\M(\bar x,\bar y)\leq 5\epsilon\Big\}\\ &
+ C \frac{\epsilon |\log\epsilon|}{r}\cdot \|f\|_{L^\infty(B_\M(x,
  2R+ \bar N\epsilon))} + C\frac{\epsilon r}{ |\log\epsilon|} \cdot 
\frac{\|\bar\A_\epsilon f- f\|_{L^\infty(B_\M(x, 2R+\bar N\epsilon))}}{\epsilon^2}.
\end{split}
\end{equation*}
We now iterate the above inequality for a total of $k=\lceil
\frac{\log\epsilon}{\log \theta}\rceil$ times, to get:
\begin{equation*}
\begin{split}
|f&(x) - f(y)|\leq \; \theta^k\cdot \sup\Big\{|f(\bar x) - f(\bar y)|;~
\bar x\in B_\M(x_0, k(R+\bar N\epsilon)), ~ d_\M(\bar x,\bar y)\leq 5\epsilon\Big\}\\ &
+ C\Big(\sum_{i=0}^{k-1}\theta^i\Big)\cdot\Big( \frac{\epsilon |\log\epsilon|}{r}\|f\|_{L^\infty(B_\M(x,
  k(2R+ \bar N\epsilon)))} + \frac{\epsilon r}{ |\log\epsilon|}
\frac{\|\bar\A_\epsilon f- f\|_{L^\infty(B_\M(x, k(2R+\bar N\epsilon)))}}{\epsilon^2}\Big).
\end{split}
\end{equation*}
Observe that $\theta^k<\theta\cdot \theta^{\log \epsilon /\log\theta}
= \theta\cdot\epsilon <\epsilon$, and also:
$$k(R+\bar N\epsilon) < k(2R+\bar N\epsilon)
<\Big(1+\frac{\log\epsilon}{\log\theta}\Big)\Big(\frac{\log
  \theta}{\log\epsilon} +1\Big) <\frac{3}{2}r$$ 
if only $\epsilon\ll 1$. Thus, the previous estimate may be rewritten as:
\begin{equation*}
\begin{split}
|f(x) - f(y)|\leq & \; 2\epsilon \|f\|_{L^\infty(B_\M(x, \frac{3}{2}r))}\\ &
+ \frac{C}{1-\theta} \Big( \frac{\epsilon
  |\log\epsilon|}{r}\|f\|_{L^\infty(B_\M(x, \frac{3}{2}r)} + \frac{\epsilon r}{ |\log\epsilon|}
\frac{\|\bar\A_\epsilon f- f\|_{L^\infty(B_\M(x, \frac{3}{2}r))}}{\epsilon^2}\Big)\\
& \leq C \frac{\epsilon
  |\log\epsilon|}{r}\|f\|_{L^\infty(B_\M(x, \frac{3}{2}r)} + C\frac{\epsilon r}{ |\log\epsilon|}
\frac{\|\bar\A_\epsilon f- f\|_{L^\infty(B_\M(x, \frac{3}{2}r))}}{\epsilon^2}\Big).
\end{split}
\end{equation*}
In view of (\ref{bd7}), this implies:
\begin{equation}\label{bd8}
\begin{split}
|f(x_0) - f(y_0)|\leq & \; C\Big(\frac{d_\M(x_0, y_0)}{r} +
\frac{\epsilon |\log\epsilon|}{r} \Big)\cdot \|f\|_{L^\infty(B_\M(x_0,
  \frac{8}{3}r))} \\ & + C\Big(r d_\M(x_0, y_0) + \frac{\epsilon r}{|\log\epsilon|}\Big)\cdot 
\frac{\|\bar\A_\epsilon f- f\|_{L^\infty(B_\M(x_0, \frac{8}{3}r))}}{\epsilon^2}.
\end{split}
\end{equation}
which proves the claim.
\end{proof}

\medskip

\noindent {\bf A proof of Theorem \ref{th_main2}.}

We first replace $\bar\A_\epsilon f(x)$ by $\A_\epsilon f(x)$ in
(\ref{bd8}) and invoke Lemma \ref{AepAepbar} to absorb the error term
$\big(r d_\M(x_0, y_0) + \frac{\epsilon r}{|\log\epsilon|}\big)
\mathcal{O}(\epsilon^2)\|f\|_{L^\infty(B_\M(x_0, 3r)}/ \epsilon^2$
in the term $\big(\frac{d_\M(x_0, y_0)}{r} + \frac{\epsilon |\log\epsilon|}{r}\big)\|f\|_{L^\infty(B_\M(x_0, 3r)}$,
provided that $\epsilon$ is small enough. 

Further, recall Remark \ref{rem_explA} (ii) to replace $\A_\epsilon f - f$ by
$C\epsilon^2\Delta_\epsilon f$. Finally, we use the obtained bound
with $2r$ instead of $r$ for $x,y\in B_\M(x_0, r)$ satisfying $d_\M(x,y)<2r$.
\endproof

\begin{remark}\label{rem_reduction}
The estimates in Theorems \ref{th_main1} and \ref{th_main2} with any
given probability density $\rho$ may be deduced from the same
statements with $\rho_0\equiv 1$. Indeed, write:
$$ |f(x) -f(x)| \leq \frac{1}{\min \rho} \Big(\big| (\rho
f)(x) - (\rho f)(y)\big| + \big|\rho(x) - \rho(y)\big| |f(y)|\Big), $$
and note the following formula for the non-local Laplacian $\Delta_\epsilon^0$ with
respect to $\rho_0$:
$$\Delta_\epsilon^0\big(\rho f\big)(x) = \Delta_\epsilon f(x)
+ \frac{f(x)}{\epsilon^{m+2}}\int_\M\eta\Big(\frac{d_\M(x,y)}{\epsilon}\Big)
(\rho(x) - \rho(y))\dVol(y).$$
The usual calculation in normal coordinates at $x$ yields: 
$\int_\M\eta\big(\frac{d_\M(x,y)}{\epsilon}\big) (\rho(x) - \rho(y))\dVol(y) =
\epsilon^{m}\int_{B(0,1)\subset T_x\M}\eta(|w|)
\big(\epsilon \langle \nabla \rho(x),w\rangle + \mathcal{O}(\epsilon^2)\big)
(1+\mathcal{O}(\epsilon^2))\dw = \mathcal{O}(\epsilon^{m+2})$, for
$\epsilon\ll 1$ when the integration variable $y$ is
close to $x\in \M$. Consequently, applying Theorems \ref{th_main1} and
\ref{th_main2} to the density $\rho_0$ and to the function $\rho f$,
yields the claimed bounds for $|f(x) - f(y)|$ relative to $\rho$.

While the above argument can be used to reduce our regularity
estimates to the case of constant $\rho$, we want to highlight that
the partial estimates which obtained along the way in our proofs (Theorems \ref{Thestimate1},
\ref{Thlipschitz3} and \ref{Thlipschitz3_2}) carry some precise information about the
dependence of constants, which we feel are of independent interest. Likewise, the construction of the biased
random walk modeled on the averaging operator
$\bar{\mathcal{A}}_\epsilon$ which admits completely arbitrary $\rho$, is
new in this context and worthy of presentation.
\end{remark}

\begin{remark}\label{rem:91}
We note that condition (\ref{asuS}) used in Theorem \ref{th2}, can be
relaxed to a more general assumption that $\eta$ is only bounded and
measureable. 
First, we write:
\[f - \A_\epsilon^2 f = (f - \A_\epsilon f) + \A_\epsilon (f - \A_\epsilon f).\]
By Remark \ref{rem_explA} (ii) and the expansion $d_\epsilon = \rho + \mathcal{O}(\epsilon^2)$ we deduce: 
\[|f - \A_\epsilon^2 f| \leq C\epsilon^2 |\Delta_\epsilon f|.\]
This allows us to replace the averaging operator $\A_\epsilon$ by
$\A_\epsilon^2$ in the proof of Theorem \ref{th2} (more correctly, we
replace $\overline{\A}_\epsilon$ by $\overline{\A}_\epsilon^2$, but
the proof may proceed with either due to Lemma \ref{AepAepbar}).  We
now note, again using the expansion $d_\epsilon = \rho +
\O(\epsilon^2)$, that we can write: 
\[\A_\epsilon^2 f(x) = \frac{1}{d_\epsilon(x)} \int_{\M} \frac{1}{\epsilon^m}\phi(x,z)f(z)\rho(z) {\rm dVol}_\M(z) + \mathcal{O}(\|f\|_\infty \epsilon^2),\]
where:
\begin{equation}\label{eq:newkernel}
\phi(x,z) = \int_{\M} \frac{1}{\epsilon^m}\eta\left(
  \frac{d_\M(x,y)}{\epsilon}\right) \eta\left(
  \frac{d_\M(y,z)}{\epsilon}\right){\rm dVol}_\M(y). 
\end{equation}
Using the change of variables $y = {\rm Exp}_x(\epsilon \tilde{y})$ in \eqref{eq:newkernel}, one can show that: 
\[\phi(x,z) = \int_{T_x\M} \eta\left( |\tilde{y}|\right)
\eta\left(|\tilde{z} - \tilde{y}|\right)d\tilde{y} +
\mathcal{O}(\epsilon^2).\] 
Hence, by working with $\A^2_\epsilon$ in Theorem \ref{th2} we can
essentially replace the kernel $\eta$ with the convolution
$\zeta=\eta*\eta$ in all arguments. When $\eta$ is bounded and
measureable, the convolution $\zeta$ is continuous, and so
(\ref{asuS}) holds for $\zeta$ without any additional assumptions on
$\eta$. The proof of Theorem \ref{th2} now proceeds by replacing
$\A_\epsilon$ with $\A_\epsilon^2$. 
\end{remark}

\bigskip

\begin{center}
{\bf \large PART 3}
\end{center}

\section{Applications to Lipschitz regularity in graph-based learning} \label{sec:Spectrum}

In this section, we work under hypothesis (\ref{H1}). Note that the
additional assumption \ref{asuS} holds automatically, in view
of continuity of $\eta$ on $[0,1]$.
Below, we prove our main results, concerning Lipschitz regularity for solutions of graph PDEs.

\subsection{Graph Poisson equations: proofs of Theorem
  \ref{thm:globalLip}, Theorem \ref{thm:interiorLip} and Corollary \ref{cor:ImpDiscreteToNonlocal} }
We deduce the three announced global-local Lipschitz regularity estimates.

\medskip

\noindent {\bf A proof of Theorem \ref{thm:globalLip}.}

We assume that the events indicated in Theorem \ref{thm:DiscreteToNonlocal} and
in Lemma \ref{lem:CM1} with $t=\eps^2$, hold. Let $f:
\X_n\to\R$. Since $\osc_{\X_n}f \leq 2\|f\|_{L^\infty(\X_n)}$, we obtain:
\[\|\NL (\Ext f)\|_{L^\infty(\M)} \leq C\left( \|\GL f\|_{L^\infty(\X_n)} + \|f\|_{L^\infty(\X_n)}\right).\]
Inserting the above into the conclusion of Theorem \ref{th_main1} yields, for all $x,y\in\M$:
\begin{equation}\label{eq:estone}
\begin{split}
|\Ext f(x) - \Ext f(y)|\leq & \; C\big( \|f\|_{L^\infty(\X_n)}+ \|\GL f\|_{L^\infty(\X_n)}\big)\cdot (d_\M(x, y)+\eps),
\end{split}
\end{equation}
where we used that $\|\Ext f\|_{L^\infty(\M)}\leq \|f\|_{L^\infty(\X_n)}$. Since
by Corollary \ref{good_def} we have $\Deg\geq c$,  recalling
(\ref{m12}) we get:
\[|f(x_i) - \Ext f(x_i)| \leq \frac{\eps^{2}}{\Deg(x_i)}|\GL f(x_i)|\leq C\eps^2\|\GL f\|_{L^\infty(\M)}
\qquad \mbox{for all }\; i=1,\ldots, n. \]
Therefore:
\begin{align*}
|f(x_i) - f(x_j)| &\leq |f(x_i) - \Ext f(x_i)| + |\Ext f(x_i) - \Ext f(x_j)| + |\Ext f(x_j) - f(x_j)|\\
&\leq C\eps^2\|\GL f\|_{L^\infty(\X_n)} + |\Ext f(x_i) - \Ext f(x_j)|.
\end{align*}
Combining the above with the estimate \eqref{eq:estone} completes the proof.
\endproof

\medskip

\noindent {\bf A proof of Corollary \ref{cor:ImpDiscreteToNonlocal}.}

Assume that both events indicated in Theorems \ref{thm:globalLip}
and \ref{thm:DiscreteToNonlocal} hold. Given $f:\X_n\to\R$, the
conclusion of Theorem  \ref{thm:globalLip} implies that: 
\[\osc_{\X_n\cap B(x,2\eps)} f \leq C\big(\|f\|_{L^\infty(\M)} + \|\GL
f\|_{L^\infty(\M)}\big)\eps\qquad\mbox{ for all }\; x\in\M.\]
The result follows then by invoking Theorem \ref{thm:DiscreteToNonlocal}.
\endproof

\medskip

\noindent {\bf A proof of Theorem \ref{thm:interiorLip}.}

The proof follows the same steps as in the proof of Theorem
\ref{thm:globalLip}, except that now we use the local estimates from
Theorem \ref{th_main2} instead of the global estimates from Theorem \ref{th_main1}.
\endproof

\subsection{Graph Laplacian eigenvectors:  proofs of Theorem
  \ref{thm:eigenLip} and Corollary \ref{cor:eigenLip}}

We apply Theorem \ref{thm:globalLip} to deduce both results, as follows.

\medskip

\noindent {\bf A proof of Theorem \ref{thm:eigenLip} and Corollary \ref{cor:eigenLip}.}

Assume that the event indicated in Theorem \ref{thm:globalLip} holds. Let $f:\X_n\to\R$ be a non-zero function
such that $\lambda_f\leq\Lambda$, where $\lambda_f = 
\frac{\|\Delta_{\epsilon, \X_n} f\|_{L^\infty(\X_n)}}{\|f\|_{L^\infty(\X_n)}}$.
Fix $\epsilon\leq \frac{c}{\lambda_f+1}$ where $c>0$ is, as usual,
a sufficiently small constant depending only on $\M$, $\rho$, $\eta$
that will be specified later. 

Assume further that the event in Corollary \ref{prop:ball} holds for
all $x\in \X_n$ and with $\epsilon$ replaced by a sufficiently small
radius $\frac{r}{2}$, depending on $\Lambda$ and also specified later.
In particular, noting that by (\ref{eq:d}) we have: $B_\M(x,r)\supset B\big(x,\frac{r}{2}\big)$,
the result in Corollary \ref{prop:ball} yields:
\begin{equation}\label{m13}
\P_n\Big(\sum_{j=1}^n \one_{\{d_\M(x_j, x_i)\leq r\}}\geq C'nr^m, \quad
\mbox{for all } x_i\in\X_n\Big) \geq 1-2n\exp\big(-cnr^m\big)
\end{equation}

Let now $x_i\in\X_n$ be such that $\|f\|_{L^\infty(\X_n)}=|f(x_i)|$. By
Theorem \ref{thm:globalLip} it follows that:
\begin{equation}\label{eqn:aux1}
\begin{split}
 |f(x_i) - f(x_j)|& \leq C(\lambda_f + 1)\lVert f \rVert_{L^\infty(\X_n)} \big(d_\M(x_i,x_j) + \eps\big)
\\ & = C (\lambda_f + 1) \cdot |f(x_i)| \big(d_\M(x_i,x_j) + \eps \big)\qquad \mbox{ for all } x_j\in\X_n,
\end{split}
\end{equation} 
which implies, provided that $r+\eps \leq \frac{1}{2C(\Lambda +1)}\leq\frac{1}{2C(\lambda_f+1)}$:
\begin{equation*}
\begin{split}
|f(x_j)| & \geq |f(x_i)| - \big|f(x_i)-f(x_j)\big|  \geq 
|f(x_i)|-C(\lambda_f+1) |f(x_i)| (r+\epsilon) 
\\ & = \big(1-C(\lambda_f+1)(r+\epsilon)\big) |f(x_i)| \geq
\frac{1}{2}\|f\|_{L^\infty(\X_n)} \qquad \quad \mbox{ for all } x_j\in B_\M(x_i, r)
\end{split}
\end{equation*}
In conclusion:
\begin{equation*}
\begin{split}
\|f\|_{L^\infty(\X_n)}\leq & \; \frac{2}{N_i}\sum_{j=1}^n\one_{\{d_\M(x_i,
  x_j)\leq r\}} f(x_j)\leq \frac{2n}{N_i}\|f\|_{L^1(\X_n)}
\\ & \mbox{where }\; N_i= \sum_{j=1}^n\one_{\{d_\M(x_i, x_j)\leq r\}}.
\end{split}
\end{equation*}

Choose $r\ll 1$ which satisfies: $r\in\big[\frac{c}{\Lambda+1}, \frac{1}{4C(\Lambda+1)}\big]$.
It then follows by (\ref{m13}) that $N_i\geq
\frac{cn}{(\Lambda+1)^m}$ and, consequently, the formula displayed above becomes
the bound in Corollary \ref{cor:eigenLip}:
\begin{equation*}
\|f\|_{L^\infty(\X_n)}\leq C(\Lambda+1)^m\|f\|_{L^1(\X_n)}.
\end{equation*}
Inserting this into the first estimate of \eqref{eqn:aux1} yields, in turn, the
bound in Theorem \ref{thm:eigenLip}. This completes
the argument, since we also easily observe that the 
probability of the two assumed events is bounded from below by: 
$$1 - C\epsilon^{-6m}
\exp(-cn\epsilon^{m+4}) - 2n\exp(-cnr^m)\geq 1 - C\epsilon^{-6m}
\exp(-cn\epsilon^{m+4}) - 2n\exp\big(-cn(\Lambda + 1)^{-m}\big),$$
as claimed.
\endproof

\subsection{\texorpdfstring{$\mathcal{C}^{0,1}$}{C} convergence rates for graph Laplacian
  eigenvectors in the large data limit: a proof Theorem \ref{thm:eigenrate} } 

We apply Theorem \ref{thm:globalLip} to conclude our final result.

\medskip

\noindent {\bf A proof of Theorem \ref{thm:eigenrate}.}

Fix $\epsilon\ll 1$ and let $f$ be a normalised eigenvector of $\GL$
with eigenvalue $\lambda$, i.e.:
$$ \GL f = \lambda f\quad \mbox{ and } \quad \|f\|_{L^2(\X_n)}=1.$$
Assume that $\X_n$ is an element of the intersection of events defined
in Theorem \ref{thm:globalLip} and Corollary \ref{cor:eigenLip}.
According to \cite[Theorem 2.6]{calder2019improved}, with
probability at least $1-Cn\exp\left( -cn\eps^{m+4} \right)$ there
exist a normalized eigenfunction $\tilde{f}$ of $\Delta_\M$ with
eigenvalue $\widetilde{\lambda}$, so that:
$$ \Delta_\M \tilde f = \tilde \lambda f \quad \mbox{ and } \quad \|\tilde f\|_{L^2(\M)}=1,$$
for which there holds: 
\[ |\lambda - \tilde{\lambda}| + \lVert  f - \tilde{f} \rVert_{L^2(\X_n)}\leq C\veps. \]
Since $\M$ is smooth, compact and boundaryless and $\tilde{f}$ is
smooth, then the pointwise consistency result in \cite[Theorem 
3.3]{calder2019improved} yields that  with probability
at least $1-2n\exp\left( -c n\eps^{m+4} \right)$ we have: 
\[ \| \Delta_\M \tilde{f} - \GL \tilde{f} \|_{L^\infty(\X_n)} \leq C\eps.  \]

Here, and in the rest of the proof, $C$ and the Landau symbol
$\mathcal{O}$ depend on $\lambda$. Denote $g = f - \tilde{f}$.
We may, without loss of generality (since otherwise the claimed result
is trivially true) assume that $g\not\equiv 0$, so that $\lambda_g
=\frac{\|\GL g\|_{L^\infty(\X_n)}}{\|g\|_{L^\infty(\X_n)}}$ is well defined. By a direct computation:
\begin{equation*}
\begin{split}
\GL g &= \big(\GL f- \Delta_\M \tilde{f}\big) + \big(\Delta_\M \tilde{f} - \GL \tilde{f}\big) 
= \lambda f - \tilde{\lambda} \tilde{f} + \O(\eps)\\
& = \lambda(f-\tilde{f}) + (\lambda - \tilde{\lambda})\tilde{f} +\O(\eps).
\end{split}
\end{equation*}
Consequently: $\|\GL g\|_{L^\infty(\X_n)}\leq \lambda
\|g\|_{L^\infty(\X_n)}+C\epsilon(1+\|\tilde f\|_{L^\infty(\M)})\leq
C\epsilon$, since $\tilde f$ is a normalised eigenvalue of $\Delta_\M$.
Hence: \[\lambda_g \leq \lambda + \frac{C\epsilon}{\|g\|_{L^\infty(\X_n)}},\]
which, in case $\lambda_g\geq \lambda+1$, clearly implies:
$\|g\|_{L^\infty(\X_n)}\leq C\epsilon$. 
On the other hand, when $\lambda_g\leq \lambda+1$, 
Corollary \ref{cor:eigenLip} yields, in view of $\|\cdot\|_{L^1(\X_n)}\leq\|\cdot\|_{L^2(\X_n)}$:
\[\|g\|_{L^\infty(\X_n)} \leq C(\lambda_g +1)^{m+1} \|g\|_{L^1(\X_n)}
\leq C(\lambda_g+1)^{m+1}\eps\leq C(\lambda+2)^{m+1}\epsilon=C\epsilon.\]

In either case, we see that there holds: 
$\|g\|_{L^\infty(\X_n)} \leq C\epsilon$.
We now invoke Theorem \ref{thm:globalLip} to get:
\begin{align*}
|g(x_i) - g(x_j)| & \leq  C\big( \|g\|_{L^\infty(\X_n)}+ \|\GL g\|_{L^\infty(\X_n)}\big)\cdot \big(d_\M(x_i, x_j)+\eps\big) \\
& \leq C\eps\left(d_\M(x_i, x_j) + \eps\right) \qquad \mbox{ for all } x_i, x_j\in\X_n.
\end{align*}
This completes the proof by union bounding on the indicated events.
\endproof

\appendix

\section{Riemannian geometry notation} \label{appendixA}

In this section we review the basic notions from differential geometry. 

\subsection{Riemannian geometry and parallel transport}\label{A1}

Let $\M$ be a smooth, compact, boundaryless, connected
and orientable manifold of dimension $m$, equipped with a smooth Riemannian metric
$g$. For $x\in\M$, we write $T_x\M$ for the tangent space at $x$
and $T\M$ for the tangent bundle of $\M$. The scalar product of any
two tangent vectors $v,w\in T_x\M$, given by the quadratic form $g(x)$
evaluated on $(v,w)$, is denoted by $\langle v, w\rangle_x$ while the length of $v$ is
$|v|_x=\langle v, v\rangle_x^{1/2}$. We will usually omit the
subscript $x$ if no ambiguity arises.
A smooth assignment of tangent vectors: $\M\ni x\mapsto v(x)\in
T_x\M$ is called a vector field.

Associated to $g$ is the Levi-Civit\`a connection $\nabla$, which is
the unique torsion-free connection that is metric $g$-compatible. Given two
vector fields $v$ and $w$, the connection allows to differentiate $v$
in the direction $w$, returning a new vector field $\nabla_w v$. The
value of  $\nabla_w v$ at $x$ depends only on the values of $w$ along a
curve $\gamma$ satisfying $\gamma(0)=x$ and $\dot\gamma(0)=v$. Keeping
this in mind, we can differentiate vector fields $v$ that are defined
only along a given smooth curve $\gamma:[a,b]\to\M$ (rather than on
the whole $\M$), in the direction $\dot\gamma$. If
$\nabla_{\dot\gamma(t)}v(\gamma(t))=0$ for all $t\in [a,b]$, then $v$
is called {\em parallel} along $\gamma$. 

For every $\bar v\in T_{\gamma(a)}\M$ there always exists the unique
vector field $v$ parallel along $\gamma$, such that
$v(\gamma(a))=\bar v$. This construction gives raise to the linear isometry map:
$$T_{\gamma(a)} \M\ni\bar v\mapsto P^\gamma_{\gamma(a),
  \gamma(b)}\bar v\doteq v(\gamma(b))\in T_{\gamma(b)}\M,$$ 
called the {\em parallel transport} along $\gamma$. We omit the reference to the curve
$\gamma$ in $P^\gamma$ if there is no ambiguity. In particular, we write $P_{x,y}$
for parallel transport along the unique geodesic connecting nearby
points $x,y\in\M$.

\subsection{Geodesics and normal coordinates}\label{sec2.2}

A smooth curve $\gamma:[a,b]\to\M$ such that its tangent vector field
$\frac{d}{dt}\gamma (t)\doteq \dot\gamma(t)$ is parallel along $\gamma$,
i.e. $\nabla_{\dot\gamma(t)}\dot\gamma(t)=0$ for all $t\in [a,b]$, is
called a {\em geodesic}. Here, $\nabla$ is used to denote the
Levi-Civit\`a connection as in \ref{A1}. For every $x\in\M$ and $v\in T_x\M$ there
exists a unique geodesic $\gamma^v: (-\infty, \infty)\to\M$ satisfying:
\begin{equation}\label{geodv}
\gamma^v(0)=x \qquad \mbox{and} \qquad \dot\gamma^v(0)=v.
\end{equation}
We will often consider a flow of
geodesics $\gamma:[-\epsilon, \epsilon]\times [0,t_0]\to \M$, where
$\nabla_{\frac{d}{dt}\gamma(s,t)}\frac{d}{dt}\gamma(s,t) =0$ for all
$(s,t)\in [-\epsilon, \epsilon]\times [0,t_0]$. Then the variation
field $J(t) = \frac{d}{ds}\gamma(0,t)$, called the {\em Jacobi field} along
the curve $\gamma(0,\cdot)$, satisfies the second order ODE:
\begin{equation}\label{Jaco}
\big(\nabla_{\frac{d}{dt}\gamma(0,t)}\big)^2J(t) + R\Big(J(t),
\frac{d}{dt}\gamma(0,t)\Big) \frac{d}{dt}\gamma(0,t) = 0\quad\mbox{ for all } t\in [0,t_0].
\end{equation}
Here, $R$ stands for the Riemann curvature form, which for three
vector fields $u,v,w$ returns the following vector field: 
$$R(u,v)w = \nabla_u\nabla_vw- \nabla_v\nabla_uw- \nabla_{[u,v]}w,$$ 
where $[u,v]=\nabla_uv-\nabla_vu$ is the commutator of $u$ and $v$.
We also recall the {\em symmetry lemma} (independent of the
geodesic property for the flow of curves $\gamma$):
$$\nabla_{\frac{d}{dt}\gamma(s,t)} \frac{d}{ds}\gamma(s,t) =
\nabla_{\frac{d}{ds}\gamma(s,t)} \frac{d}{dt}\gamma(s,t)\qquad
\mbox{for all } (s,t)\in [-\epsilon, \epsilon]\times [0,t_0].$$ 

The length of a smooth curve $\gamma:[a,b]\to\M$ is computed as: $length(\gamma) =
\int_{a}^b|\dot\gamma(t)|_{\gamma(t)}\dt$. For every $x,y\in\M$ this
gives raise to the well-defined metric distance:
\begin{equation}\label{dist}
 d_\M (x,y) = \min\big\{length(\gamma);~ \gamma(a)=x, ~\gamma(b)=y\big\}.
\end{equation}
The open ball in this metric is denoted by $B_\M (x,r) = \{y\in\M;~
 d_\M (x,y)<r\}$. When $ d_\M (x,y)<\iota$ for a sufficiently small radius
$\iota>0$, depending (in view of compactness) only on $(\M, g)$, then the
minimization in (\ref{dist}) is realised by the unique, up to
re-parametrisation, curve $\gamma:[0,1]\to\M$ which is a
geodesic. Automatically, one has: $|\dot\gamma(t)|_{\gamma(t)}\equiv
length(\gamma) =  d_\M (x,y)$.

For every $x\in\M$ one considers
the {\em exponential mapping} $T_x\M\supset B(0,\iota)\ni v \mapsto \gamma^v(1)\in\M$,
where $\gamma^v$ is the geodesic as in (\ref{geodv}). This
mapping is usually denoted by
$Exp_x(v)\doteq\gamma^v(1)$, and it is a smooth diffeomorphism onto its image.
By compactness of $\M$, also the mapping:
\begin{equation*}
T\M\supset\mathcal{U}\ni \big(x\in\M, v\in B(0,\iota)\subset
T_x\M\big)\mapsto \psi(x,v) \doteq \big(x, Exp_x(v)\big)\in\M\times \M,
\end{equation*}
is well defined on an open neighbourhood $\mathcal{U}$ of a
zero-section in $\M$, and it is a smooth diffeomorphism onto its
image. Any orthonormal basis for $T_x\M$ gives an isomorphism $T_x\M\backsimeq
\R^m$. Then, the inverse of the exponential map:
\begin{equation}\label{geod}
B_\M(x,\iota)\ni y \mapsto Exp_x^{-1}(y)=v \in T_x\M \backsimeq\R^m
\end{equation}
is called the {\em normal coordinate chart} centered at $x$. 

\subsection{Formulas in coordinates}\label{sec2.3}

We now recall that in a given local coordinate chart:
$$\mathcal{U}\ni x\mapsto (x^1,\ldots, x^m)\in\R^m$$ 
(not necessarily normal as in (\ref{geod})) on an
open subset $\mathcal{U}\subset\M$, the tangent space $T_x\M$ is
spanned by coordinate vectors $\big\{\frac{\partial}{\partial
  x^i}\big\}_{i=1}^m$, each corresponding to the smooth curve
$\gamma(t) = x+t \frac{\partial}{\partial x^i}$ on $\M$ passing
through $\gamma(0) = x$. The metric $g$ is represented as the symmetric
matrix field $[g_{ij}(x)]_{i,j=1\ldots m}$ on $\mathcal U$,
so that $\langle v, w\rangle_x=g_{ij}(x)v^iw^j$ for all $v= v^i\frac{\partial}{\partial
  x^i}$, $w= w^i\frac{\partial}{\partial x^i}\in T_x\M$. Here and
below, we use the Einstein summation convention on repeated lower and upper indices. 

For a vector field  $v= v^i\frac{\partial}{\partial x^i}$ on $\mathcal
U$, the corresponding coordinates of $\nabla_{\frac{\partial}{\partial
  x^j}} v = (\nabla_jv^i) \frac{\partial}{\partial x^i}$ are:
$\nabla_jv^i = \frac{\partial v^i}{\partial x^j} + \Gamma_{jk}^iv^k$,
given through the Christoffel symbols:
$$\Gamma_{jk}^i = \frac{1}{2}g^{is}\Big(\frac{\partial g_{sk}}{\partial x^j} + 
\frac{\partial g_{sj}}{\partial x^k} - \frac{\partial g_{jk}}{\partial x^s}\Big).$$ 
As customary, $[g^{ij}(x)]_{i,j=1\dots m}$ is the inverse matrix of
$[g_{ij}(x)]_{i,j=1\dots m}$, so that $g^{is}g_{sj} =\delta_j^i$ equaling
$1$ for $i=j$ and $0$ otherwise.

For a smooth curve $\gamma:[a,b]\to\M$ written in coordinates:
$\gamma(t) = \big(\gamma^1(t),\ldots, \gamma^m(t)\big)$, a vector field $[a,b]\ni
t\mapsto v(t) = v^i(t)\frac{\partial}{\partial x^i}\in
T_{\gamma(t)}\M$ is parallel along $\gamma$ if it satisfies the ODE system:
$$\dot v^i(t) +\Gamma_{jk}^i(\gamma(t)) v^j(t)\dot\gamma^k(t) =
0\qquad \mbox{for all } i=1\ldots m \quad \mbox{and all } ~t\in [a,b].$$ 
Consequently, equations of geodesic are:
$$\ddot \gamma^i + \Gamma_{jk}^i\dot\gamma^j\dot\gamma^k = 0 \qquad \mbox{for all } i=1\ldots m.$$
We also have: $R(v,w)z = -R^k_{ijs}v^iw^jz^s \frac{\partial}{\partial x^k}$, where:
$$R_{ijs}^k = \frac{\partial\Gamma^k_{is}}{\partial x^j} - \frac{\partial\Gamma^k_{js}}{\partial x^i} +
\Gamma_{jl}^k\Gamma_{is}^l - \Gamma_{il}^k\Gamma_{js}^l$$ 
are the components of the Riemann curvature tensor.
Denoting: $\nabla_{\frac{\partial}{\partial x^i}}
\nabla_{\frac{\partial}{\partial x^j}}v = \big(\nabla_i\nabla_j
v^k\big) \frac{\partial}{\partial x^k}$, it follows that:
$\nabla_i\nabla_j v^k - \nabla_j\nabla_i v^k = - R_{ijs}^kv^s$.

In the normal coordinates (\ref{geod}), centered at a point $x\in\M$, 
corresponding to $0\in \R^m$, we have the useful identities: 
$$[g_{ij}(0)]_{i,j=1\ldots m}=Id_m, \quad \Gamma_{ij}^k(0)=0,\quad
\frac{\partial g_{ij}}{\partial x^k}(0)=0, \quad \frac{\partial^2
  g_{ij}}{\partial x^k\partial x^s}(0)= -\frac{2}{3}R^i_{kjs}(0),$$
valid for all $i,j,k,s=1,\ldots, m$. Consequently, we obtain the Taylor expansion:
\begin{equation} \label{Tayg}
g_{ij}(y) = \delta_{ij}-\frac{1}{3}R_{kjs}^i(0) y^ky^s +
\mathcal{O}(|y|^3)\quad \mbox{ for all } i,j=1,\ldots, m \mbox{ and
  all } y\in B(0,\iota)\subset\R^m.
\end{equation}

The contravariant derivative (the gradient) of a scalar field
$f:\M\to\R$ is the vector field
$\nabla^*f=(\nabla^if)\frac{\partial}{\partial x^i}$ on $\M$, whose
coordinates are: $\nabla^if=g^{ik}\frac{\partial f}{\partial
  x^k}$. The divergence of $\nabla^*f$, called the Laplace-Beltrami
operator of $f$ is given by:
$$\Delta f \doteq Div(\nabla^*f) = \nabla_i\nabla^if =
g^{ik}\frac{\partial^2f}{\partial x^i\partial x^k} +
\Big(\frac{\partial g^{ij}}{\partial x^i} + \Gamma_{ik}^ig^{jk}\Big)
\frac{\partial f}{\partial x^j} =
g^{ik}\Big(\frac{\partial^2f}{\partial x^i\partial x^k} -
\Gamma_{ik}^j\frac{\partial f}{\partial x^j}\Big).$$
We will also  consider the following weighted Laplace-Beltrami
operator with respect to a positive scalar field $\rho:\M\to \R$:
\begin{equation*}
\begin{split}
\frac{1}{\rho^2} Div(\rho^2\nabla^*f) & = \Delta f + 2
g^{ik}\frac{\partial (\log \rho)}{\partial x^i} \frac{\partial f}{\partial x^k} =
\Delta f + 2\nabla^i f\cdot \nabla^s(\log\rho)g_{is} \\ & = \Delta f
+\langle \nabla^*f, \nabla^*(\log\rho)\rangle_x.
\end{split}
\end{equation*}
We remark that in normal coordinates centered at $x$, the
Laplace-Beltrami operator computed at $x$, coincides with the
usual Laplacian $\sum_{i=1}^m\frac{\partial^2f}{(\partial x^i)^2}(0)$
in those coordinates, and similarly:
$$\frac{1}{\rho^2} Div(\rho^2\nabla^*f) =
\sum_{i=1}^m\frac{\partial^2f}{(\partial x^i)^2}(0) +
2\sum_{i=1}^m\frac{\partial f}{\partial x^i} \frac{\partial (\log\rho)}{\partial x^i}(0).$$

\subsection{The embedded manifold}\label{sec2.4}

When $\M$ is embedded in the ambient space $\R^d$, it inherits its
Euclidean structure.  There are two related facts that will be
frequently used in the sequel. First, it follows from the Rauch
Comparison Theorem \cite{BIK,calder2019improved,carmo1992riemannian}  that there exists
$C>0$ depending on $\M$ (more precisely, on the upper bound of the
sectional curvatures) such that for all $r\ll 1$:
\begin{equation}\label{eq:vol}
\big|{Vol}_\M(B_\M(x,r))- |B(0,1)| r^m\big|\leq Cr^{m+2} \qquad\mbox{ for all } x\in\M.
\end{equation}
Here, $|B(0,1)|$ denotes the volume of the unit ball in $\R^m$. The
volume ${Vol}_\M(B_\M(x,r))$ of the indicated geodesic ball,
is related to the Riemannian volume form $\dVol$ in:
$${Vol}_\M(B_\M(x,r)) = \int_{B_\M(x,r)} 1\dVol.$$
In local coordinates $\dVol$ is given by: $\dVol(y) = \big(\det [g_{ij}]_{i,j=1\ldots
  m}\big)^{1/2}\dy^1\wedge\ldots \wedge \dy^m$.

The second fact \cite[Proposition 2]{trillos2018spectral} states that for all
$x,y\in\M$ whose distance in $\R^d$ satisfies $|x-y|\leq
\frac{R}{2}$ with a sufficiently small $R$ (more precisely, $R$ is the reach
of $\M$), there holds:
\begin{equation}\label{eq:d}
|x-y|\leq  d_\M (x,y)\leq |x-y|+ \frac{8}{R^2}|x-y|^3.
\end{equation}

\bibliography{ref1}
\bibliographystyle{abbrv}

\end{document}